              \def\version{March 20, 2017}     	                  %

\newif\ifmarek
\marektrue

\documentclass[reqno,twoside,11pt]{amsart}

\usepackage{srcltx}
\usepackage{amsmath}
\usepackage{amsfonts}
\usepackage{amssymb}
\usepackage{verbatim}
\usepackage{epsfig}
\usepackage{color}
\usepackage{hyperref}

\usepackage{amssymb,amsthm,latexsym,amscd,bbm,stmaryrd}

\usepackage[english]{babel}
\usepackage{graphicx}
\usepackage{tikz}
\usepackage{courier}
\usepackage{enumerate}
\usepackage{psfrag}
\usepackage{wasysym}
\usepackage{type1cm}

\IfFileExists{epsf.def}{\input epsf.def}{\usepackage{epsf}}

\IfFileExists{mathpazo.sty}{\usepackage{mathpazo}\usepackage{mathrsfs}}
{\usepackage{times}\usepackage{mathrsfs}}

\DeclareFontFamily{OT1}{eusb}{} \DeclareFontShape{OT1}{eusb}{m}{n} {<5> <6> <7> <8> <9> <10> <11> <12> <14.4> eusb10}{}
\DeclareMathAlphabet{\eusb}{OT1}{eusb}{m}{n}

\DeclareFontFamily{OT1}{eusm}{} \DeclareFontShape{OT1}{eusm}{m}{n} {<5> <6> <7> <8> <9> <10> <11> <12> <14.4> eusm10}{}
\DeclareMathAlphabet{\eusm}{OT1}{eusm}{m}{n}

\DeclareFontFamily{OT1}{eufm}{} \DeclareFontShape{OT1}{eufm}{m}{n} {<5> <6> <7> <8> <9> <10> <11> <12> <14.4> eufm10}{}
\DeclareMathAlphabet{\mathfrak}{OT1}{eufm}{m}{n}

\DeclareFontFamily{OT1}{fraktura}{}
\DeclareFontShape{OT1}{fraktura}{m}{n} {<5> <6> <7> <8> <9> <10> <11> <12> <13> <14.4> [1.1] eufm10}{}
\DeclareMathAlphabet{\fraktura}{OT1}{fraktura}{m}{n}

\DeclareFontFamily{OT1}{cmfi}{} \DeclareFontShape{OT1}{cmfi}{m}{n} {<5> <6> <7> <8> <9> <10> <11> <12> <13> <14.4> [0.9] cmfi10}{}
\DeclareMathAlphabet{\cmfi}{OT1}{cmfi}{b}{n}

\DeclareFontFamily{OT1}{cmss}{} \DeclareFontShape{OT1}{cmss}{m}{n} {<5> <6> <7> <8> <9> <10> <11> <12> <13> <14.4> cmss10}{}
\DeclareMathAlphabet{\cmss}{OT1}{cmss}{m}{n}

\ifmarek
\else
\fi

\newtheoremstyle{thm}{1.5ex}{1.5ex}{\itshape\rmfamily}{} {\bfseries\rmfamily}{}{2ex}{}

\newtheoremstyle{def}{1.5ex}{1.5ex}{\slshape\rmfamily}{} {\bfseries\rmfamily}{}{2ex}{}

\newtheoremstyle{rem}{1.3ex}{1.3ex}{\rmfamily}{} {\itshape}
{} {1.5ex}{}

\theoremstyle{thm}
\newtheorem{theorem}{Theorem}[section]
\newtheorem{lemma}[theorem]{Lemma}
\newtheorem{proposition}[theorem]{Proposition}

\newtheorem*{Main Theorem}{Main Theorem.}
\newtheorem{corollary}[theorem]{Corollary}
\newtheorem*{special theorem}{Lindeberg-Feller Theorem for Martingales}

\theoremstyle{def}

\newtheorem{assumption}[theorem]{Assumption}

\theoremstyle{rem}
\newtheorem{remark}[theorem]{{\itshape Remark}}

\numberwithin{equation}{section}

\renewcommand{\section}{\secdef\sct\sect}
\newcommand{\sct}[2][default]{%
\refstepcounter{section}
\addcontentsline{toc}{section}{{\tocsection {}{\thesection}{\!\!\!\!#1\dotfill}}{}}
\vspace{0.7cm}
\centerline{\scshape\thesection.\ #1} \nopagebreak \vspace{0.2cm}}
\newcommand{\sect}[1]{%
\vspace{0.4cm} \centerline{\large\scshape\rmfamily #1}
\vspace{0.2cm}}

\renewcommand{\subsection}{\secdef\subsct\sbsect}
\newcommand{\subsct}[2][default]{\refstepcounter{subsection}
\addcontentsline{toc}{subsection}
{{\tocsection{\!\!}{\hspace{1.2em}\thesubsection}{\!\!\!\!#1\dotfill}}{}}\nopagebreak\vspace{0.45\baselineskip} {\flushleft\bf
\thesubsection~\bf #1.~}
\\*[3mm]\noindent
\nopagebreak}
\newcommand{\sbsect}[1]{\vspace{0.1cm}\noindent
\textbf{#1.~}\vspace{0.1cm}}

\renewcommand{\subsubsection}{%
\secdef \subsubsect\sbsbsect}
\newcommand{\subsubsect}[2][default]{%
\refstepcounter{subsubsection} 
\addcontentsline{toc}{subsubsection}{{\tocsection{\!\!}
{\hspace{3.05em}\thesubsubsection}{\!\!\!\!#1\dotfill}}{}}
\nopagebreak
\vspace{0.15\baselineskip} \nopagebreak {\flushleft\rmfamily
\itshape\thesubsubsection
\ \rmfamily #1\/.}\ }
\newcommand{\sbsbsect}[1]{\vspace{0.1cm}\noindent
\rmfamily \itshape
\arabic{section}.\arabic{subsection}.\arabic{subsubsection} \
\sffamily #1\/.\ }

\renewcommand{\caption}[1]{%
\vglue0.5cm
\refstepcounter{figure}
\begin{minipage}{0.9\textwidth}\small {\sc Figure~\thefigure. }#1\end{minipage}}

\newcommand{\dist}{\operatorname{dist}}
\newcommand{\supp}{\operatorname{supp}}
\newcommand{\diam}{\operatorname{diam}}
\newcommand{\esssup}{\operatorname{esssup}}
\newcommand{\essinf}{\operatorname{essinf}}

\newcommand{\e}{\operatorname{e}}

\newcommand{\textd}{\text{\rm d}\mkern0.5mu}

\newcommand{\texte}{\text{\rm e}}

\newcommand{\1}{\operatorname{\mathbbm{1}}\!}

\newcommand{\Prob}{\text{\rm Prob}}
\newcommand{\eps}{\varepsilon}

\renewcommand{\AA}{\mathcal A}

\newcommand{\CC}{\mathcal C}
\newcommand{\DD}{\mathcal D}
\newcommand{\EE}{\mathcal E}
\newcommand{\FF}{\mathcal F}
\newcommand{\GG}{\mathcal G}
\newcommand{\HH}{\mathcal H}
\newcommand{\II}{\mathcal I}

\newcommand{\LL}{\mathcal L}
\newcommand{\MM}{\mathcal M}
\newcommand{\NN}{\mathcal N}

\newcommand{\PP}{\mathcal P}

\newcommand{\RR}{\mathcal R}

\newcommand{\TT}{\mathcal T}

\newcommand{\E}{\mathbb E}

\newcommand{\N}{\mathbb N}

\newcommand{\Q}{\mathbb Q}
\newcommand{\R}{\mathbb R}

\newcommand{\Z}{\mathbb Z}

\newcommand{\scrC}{\mathscr{C}}

\newcommand{\scrM}{\mathscr{M}}

\newcommand{\scrP}{\mathscr{P}}

\newcommand{\twoeqref}[2]{(\ref{#1}--\ref{#2})}
\newcommand{\cc}{{\text{\rm c}}}

\def\myffrac#1#2 in #3{\raise 2.6pt\hbox{$#3 #1$}\mkern-1.5mu\raise 0.8pt\hbox{$#3/$}\mkern-1.1mu\lower 1.5pt\hbox{$#3 #2$}}

\newcommand{\ssup}[1] {{\scriptscriptstyle{({#1}})}}

\renewcommand{\P}{\mathbb P}
\newcommand{\scrMp}{\scrM_{\text{\rm P}}}

\newcommand{\RV}{}
\newcommand{\eRV}{\normalcolor}

\newcommand{\RS}{}
\newcommand{\eRS}{\normalcolor}

\begin{document}

\title[Parabolic Anderson model\hfill]
{Mass concentration and aging in the parabolic\\Anderson model with doubly-exponential tails}

\author[\hfill  Biskup, K\"onig and dos~Santos]
{Marek Biskup$^{1,2}$,\,\, Wolfgang K\"onig$^{3,4}$\, \and\,\,Renato S.\ dos Santos$^3$}
\thanks{\hglue-4.5mm\fontsize{9.6}{9.6}\selectfont\copyright\,\textrm{2017}\ \ \textrm{M.~Biskup, W.~K\"onig and R.S.~dos Santos.
Reproduction, by any means, of the entire
article for non-commercial purposes is permitted without charge.\vspace{2mm}}}
\maketitle

\vglue-4mm

\maketitle

\centerline{\textit{$^1$Department of Mathematics, UCLA, Los Angeles, California, USA}}
\centerline{\textit{$^2$Center for Theoretical Study, Charles University, Prague, Czech Republic}}
\centerline{\textit{$^3$Weierstra\ss-Institut f\"ur Angewandte Analysis und Stochastik, Berlin, Germany}}
\centerline{\textit{$^4$Institut f\"ur Mathematik, Technische Universit\"at Berlin, Berlin, Germany}}

\vskip2mm
\centerline{{Version: \version}}

\vspace{-2mm}
\begin{abstract}
We study the non-negative solution $u=u(x,t)$ to the Cauchy problem for the parabolic equation $\partial_t u=\Delta u+\xi u$ on $\Z^d\times[0,\infty)$  with initial data $u(x,0)=\1_0(x)$. Here $\Delta$ is the discrete Laplacian on~$\Z^d$ and~$\xi=(\xi(z))_{z\in\Z^d}$ is an i.i.d.\ random field with doubly-exponential upper tails.
We prove that, for large~$t$ and with large probability, 
 most of the total mass $U(t):=\sum_x u(x,t)$ of the solution
resides in a boun\-ded neighborhood of a site~$Z_t$
that achieves an optimal compromise between the local Dirichlet eigenvalue of the Anderson Hamiltonian $\Delta+\xi$ and the distance to the origin.
The processes~$t\mapsto Z_t$ and~$t \mapsto \tfrac1t \log U(t)$ are shown to converge in distribution under suitable scaling of space and time. Aging results for $Z_t$, as well as for the solution to the parabolic problem, are also established. 
The proof uses the characterization of eigenvalue order statistics for~$\Delta+\xi$ in large sets
recently proved by the first two authors.
\end{abstract}

\tableofcontents

\section{Introduction}
\label{sec-Intro}
\noindent
Random Schr\"odinger operators --- most notably, the Anderson Hamiltonian $H=\Delta + \xi$ --- have been a subject of intense research over several decades. Most of the attention has been paid to the character of the spectrum and the ensuing physical consequences for the \emph{quantum} evolution. However, the associated 
\emph{parabolic} problem 
--- characterized by the PDE $\partial_t u=\Delta u +\xi u $ --- is of as much interest both for theory and applications.
Here we study the latter facet of this problem for a specific class of random potentials.
Our main result is the proof of localization of 
the solution to the above PDE for large time 
in a neighborhood of a process determined solely by the random potential.

A standard  way to describe the \emph{parabolic Anderson model} (PAM) is via a non-negative solution $u \colon \Z^d \times [0,\infty) \to [0,\infty)$
of the Cauchy problem
\begin{alignat}{3}
\label{PAM}
\partial_t u(z,t) &=  \Delta u(z,t) + \xi(z) u(z,t),
\qquad &&z \in \Z^d, \, t \in (0,\infty),
\\
\label{PAMinitial}
u(z,0) &=\1_0(z),\qquad &&z\in\Z^d.
\end{alignat}
Here $\xi=(\xi(z))_{z\in \Z^d}$ is an i.i.d.~random potential taking values in $[-\infty,\infty)$,
$\1_x$ is the indicator function of a point $x \in \Z^d$,
$\partial_t$ abbreviates the derivative with respect to $t$, and~$\Delta$ is the discrete Laplacian acting on~$f\colon\Z^d\to\R$ as
\begin{equation}
\Delta f(z):= \sum_{y\colon|y-z|=1} \bigl[f(y)-f(z)\bigr],
\end{equation}
where $|\cdot|$ denotes the $\ell^1$ norm on $\Z^d$. 

The interest in \twoeqref{PAM}{PAMinitial} for mathematics as well as applications 
comes from the competing effect of the two terms on the right-hand side  of~\eqref{PAM}. 
Indeed, the Laplacian tends to make the solution smoother over time, while the field makes it rougher. 
The problem~\eqref{PAM} appears in the studies of chemical kinetics \cite{GM90}, 
hydrodynamics \cite{CM94}, and magnetic phenomena \cite{MR94}. 
We refer to the reviews \cite{M94,CM94} for more background, and to \cite{GM90} for the fundamental mathematical properties of the model. 
A recent comprehensive survey of mathematical results on the PAM and related models can be found in \cite{K15}; the related spectral order-statistics questions are reviewed in~\cite{A16}.

A non-negative solution to the Cauchy problem \twoeqref{PAM}{PAMinitial} exists and is unique as soon as the upper tail of $[\xi(0)/\log\xi(0)]^d$ is integrable~\cite{GM90}. Under this condition, there is also a representation in terms of the changed-path measure,
\begin{equation}
\label{QT}
Q_t^{\ssup\xi}(\textd X):=\frac1{U(t)}\exp\Bigl\{\int_0^t\xi(X_s)\textd s\Bigr\}\P_0(\textd X),
\end{equation}
on nearest-neighbor paths~$X=(X_s)_{s \ge 0}$ on~$\Z^d$, where $\P_0$ stands for the law of a continuous-time random walk on~$\Z^d$ (with generator $\Delta$)
started at zero. Indeed, the Feynman-Kac formula shows
\begin{equation}
\label{FK}
u(z,t)=U(t)Q_t^{\ssup\xi}(X_t=z) = \E_0 \left[\e^{ \int_0^t \xi(X_s) \textd s } \1_{\{X_t = z\}}\right],
\end{equation}
whereby the normalization constant~$U(t)$ obtains the meaning
\begin{equation}
\label{e:deftotalmass}
U(t)=\sum_{x\in\Z^d}u(x,t)
= \E_0\biggl[\exp \int_0^t\xi(X_s)\textd s \biggr].
\end{equation}
The aforementioned competition is now obvious probabilistically: the walk would like to maximize the ``energy'' $\int_0^t\xi(X_s)\textd s$, by spending its time at the places where $\xi$ is large, against the ``entropy'' of such trajectories under the path measure $\P_0$.

An alternative and equally useful way to view~\eqref{PAM} is
as the definition of a semigroup $t\mapsto\texte^{t(\Delta+\xi)}$ on~$\ell^2(\Z^d)$.
The solution to \twoeqref{PAM}{PAMinitial} is then given by
\begin{equation}
\label{e:introl2descr}
u(x,t)=\bigl\langle \1_x,\texte^{t(\Delta+\xi)}\1_0\bigr\rangle_{\ell^2(\Z^d)}.
\end{equation}
This opens up the possibility to control the large-$t$ behavior through spectral analysis of the Anderson Hamiltonian.
To this end, it is useful to restrict the problem to a 
sufficiently large (in $t$-dependent fashion) 
finite volume $\Lambda \subset \Z^d$ (with $0 \in \Lambda$) as follows.
Denote by $H_\Lambda$ the Anderson Hamiltonian in $\Lambda$ with (zero) Dirichlet boundary conditions,
i.e., for $\phi \in \R^\Lambda$, $H_\Lambda \phi = H \tilde{\phi}$ where $H=\Delta+\xi$ and~$\tilde \phi$ is the extension of~$\phi$
to $\R^{\Z^d}$ that is equal to zero on $\Lambda^\cc$.
Let $u_\Lambda$ be the solution to \twoeqref{PAM}{PAMinitial} restricted to~$\Lambda$ and with the right-hand side of~\eqref{PAM}
substituted by $H_\Lambda u$.
Then the above interpretation yields

\begin{equation}
\label{e:specdecompLambda}
u_\Lambda(x,t)=\sum_{k=1}^{|\Lambda|} \texte^{t\lambda^{\ssup k}_\Lambda}\phi^{\ssup k}_\Lambda(x)\phi_\Lambda^{\ssup k}(0),
\end{equation}
where $\lambda^{\ssup k}_\Lambda$ are the eigenvalues and~$\phi^{\ssup k}_\Lambda$ the corresponding eigenvectors of $H_\Lambda$, 
which we assume to be orthonormal in~$\ell^2(\Lambda)$.
Hereafter, we extend both the solution $u_\Lambda(\cdot, t)$ and the eigenfunctions of $H_\Lambda$ to $\Z^d$ by setting them to be equal to $0$ on $\Lambda^\cc$.

The competition we described in the context of the changed-path measure~\eqref{QT} now manifests itself as follows.
The term in the sum in~\eqref{e:specdecompLambda} that grows the \emph{fastest} in $t$ is that with the largest eigenvalue. However, 
there is no \emph{a priori} reason for it to be
the \emph{dominant} term at a fixed time. Indeed,
an eigenvalue will only contribute to~\eqref{e:specdecompLambda} when its eigenvector puts non-trivial mass on both~$0$ and~$x$. Since the leading eigenvectors decay exponentially away from their localization centers (Anderson localization), $|\phi^{\ssup k}_\Lambda(0)|$ will in fact be typically extremely small. It is thus the combined effect of both $\texte^{t\lambda^{\ssup k}_\Lambda}$ and~$\phi^{\ssup k}_\Lambda(x)\phi_\Lambda^{\ssup k}(0)$ that decides which index $k$ will give the main contribution to the sum. 

In the present paper, we analyze these competing effects for a class of random potentials 
with upper tails close to the doubly-exponential distribution, characterized by
\begin{equation}
\label{tail}
\Prob\bigl(\xi(0)>r\bigr) = \exp\bigl\{-\texte^{r/\rho}\bigr\}, \qquad r \in \R,
\end{equation}
where~$\rho\in(0,\infty)$. (Precise definitions will appear in Section~\ref{s:results}.) 
For these potentials we show that, at all large~$t$, most of the total mass $U(t)$ of the solution
resides in a bounded neighborhood of a random point $Z_t$ determined entirely by~$\xi$. 
This point marks the optimal local peak of~$\xi$
for the strategy where
the random walk in~\eqref{QT} traverses to~$Z_t$ in time~$o(t)$,  and thereafter ``sticks around'' $Z_t$ in order to enjoy the benefits of a ``strong'' local Dirichlet eigenvalue.
We also characterize the scaling limits of $Z_t$ and $\tfrac1t \log U(t)$,
and obtain aging results for both $Z_t$ and $u(x,t)$.

Our results build on a large body of literature on the PAM whose full account here would divert from the main message of the paper. 
For now let us just say that we extend results from \cite{MOS11, LM12, ST14, FM14},
dealing with localization on one lattice site,
to a benchmark class of random potentials exemplified by~\eqref{tail},
where the localization takes place in large domains, albeit not growing with~$t$. 
An important technical input for us is the recent work~\cite{BK16}, where eigenvalue order statistics for the Anderson Hamiltonian~$H=\Delta+\xi$ was characterized for this class of~$\xi$. Further connections will be given in Section~\ref{s:literature}.

\section{Main results}
\label{s:results}\nopagebreak\noindent
We now move to the statements of our main results. Throughout the paper, $\ln x$ denotes the natural logarithm of~$x$, 
and $\ln_2x:=\ln\ln x$, $\ln_3 x := \ln \ln \ln x$, etc denote its iterates. 
We will use ``Prob'' to denote the probability law of the i.i.d.\ random field~$\xi$.

\subsection{Assumptions}
\label{ss:assumptions}\noindent
We begin by identifying the class of potentials that we will consider in the sequel. 
Besides some regularity, the following ensures that the upper tails of~$\xi(0)$ are in the vicinity of the doubly-exponential distribution~\eqref{tail}.

\begin{assumption}[Upper tails]
\label{A:dexp}
Suppose that $\esssup \xi(0) = \infty$ and let
\begin{equation}\label{e:defF}
F(r) := \ln_2 \,\frac{1}{\Prob (\xi(0) > r)} \; , \qquad r > \essinf \xi(0).
\end{equation}
We assume that $F$ is differentiable on its domain and that
\begin{align}\label{e:assumpF'}
\lim_{r \to \infty} F'(r) = \frac{1}{\rho} \;\; \text{ for some } \;\;\rho \in (0,\infty).
\end{align}
\end{assumption}

The assumption above is exactly as Assumption~1.1 in \cite{BK16},
and implies Assumption~(F) of~\cite{GM98}.
While the latter would be enough for most of our needs,
the extra requirements of Assumption~\ref{A:dexp} are used in the crucial step, performed in \cite{BK16},
of identifying the max-order class of the local principal eigenvalues of the Anderson Hamiltonian.
In order to avoid technical inconveniences, we will also assume the following
condition on the lower tail of $\xi$.

\begin{assumption}[Lower tails]
\label{A:integlowtailxi}
Let $\xi^-(x) := \max\{0, -\xi(x)\}$. We assume that
\begin{equation}\label{e:integlowtailxi}
\int_0^{\infty} \textnormal{Prob} \bigl( \xi^-(0) > \texte^s \bigr)^{\frac1d} \textd s < \infty.
\end{equation}
\end{assumption}

Assumption~\ref{A:integlowtailxi} is only used in the proof of Lemma~\ref{l:lbinitialmasscontrib},
which is used in Proposition~\ref{prop:lowerbound_forproof} to give a lower bound for the total mass $U(t)$.
Note that~\eqref{e:integlowtailxi} holds whenever $\ln(1+\xi^{-}(0))$ has a $(d+\eps)$-th finite moment (cf.\
\cite{M02}).
We believe that, with the use of percolation arguments, this assumption can be relaxed to $\xi(0) > -\infty$ almost surely in $d\ge2$. 
In $d=1$,~\eqref{e:integlowtailxi} is equivalent to $\ln(1+\xi^{-}(0))$ having the first moment,
which is known in the case of bounded potentials to be ``essentially necessary'' in the sense that,
when $|\ln(1+\xi^{-}(0))|^\delta$ is not integrable for some $\delta \in (0,1)$, the solution might scale differently.
See \cite{BK01b}, in particular Remarks~3 and 4 therein.

We will assume the validity of Assumptions~\ref{A:dexp}--\ref{A:integlowtailxi} throughout the rest of the paper without explicitly stating this in each instance.

\subsection{Results: Mass concentration}
\label{ss:results}\noindent 
Recall that $|x|$ denotes the $\ell^1$-norm of~$x$. Our first result concerns the concentration of the total mass of the solution to the Cauchy problem~\twoeqref{PAM}{PAMinitial}:

\begin{theorem}[Mass concentration]
\label{thm:massconc}
There is a $\Z^d$-valued c\`adl\`ag stochastic process $(Z_t)_{t > 0}$
depending only on $\xi$ such that $t\mapsto |Z_t|$ is non-decreasing and such that the following holds: 
For each $\delta > 0$, there exists $R \in \N$ such that,
for any $l_t>0$ satisfying $\lim_{t \to \infty} \frac{1}{t} l_t = 0$,
\begin{equation}\label{e:massconc}
\lim_{t \to \infty} \textnormal{Prob} \left( \sup_{s \in [t-l_t,t+l_t]} \; \sum_{x \colon\, |x-Z_t|>R} \frac{u(x,s)}{U(s)} > \delta \right) = 0.
\end{equation}
\end{theorem}
In words,~\eqref{e:massconc} means that the solution at time~$t$ is with large probability concentrated near a single point~$Z_t$, and the control in fact extends to sublinearly-growing intervals of time around~$t$.
This cannot be extended to linearly growing time-intervals due to the jumps of the process $s\mapsto Z_s$ (cf.\ Theorem~\ref{thm:aging_locus} below), 
but a refinement of our methods would show that, in this case,
\emph{two} islands would suffice, i.e., \eqref{e:massconc} would still hold if the sum is taken over boxes of radius~$R$ centered around two processes $Z^{\ssup 1}_s$, $Z^{\ssup 2}_s$ (see~\eqref{e:defZk}).
We also believe that the almost-sure version of this statement, 
dubbed as a ``two-cities theorem''
and proved in \cite{KLMS09} for the case of Pareto potentials,
could be obtained with more work but prefer not to pursue this here.

In terms of the path measure $Q^{\ssup \xi}_t$, Theorem~\ref{thm:massconc} can be interpreted as concentration for the law of the position of the path at time~$t$.  By letting the radius $R$ grow slowly to infinity, this can be improved to include a majority of the random walk path:
\begin{theorem}[Path localization]
\label{thm:pathconcentration}
For any $\epsilon_t \in (0,1)$ satisfying $\lim_{t \to \infty} \epsilon_t  \ln_3 t = \infty$,
\begin{equation}\label{e:pathconc}
\lim_{t \to \infty} Q^{\ssup \xi}_t \left(\, \sup_{s \in [\epsilon_t t , t]}  \left|X_s-Z_t\right|>  \epsilon_t \ln t \right) = 0 \;\;\; \text{ in probability,}
\end{equation}
where $(Z_t)_{t > 0}$ is the stochastic process in Theorem~\ref{thm:massconc}.
\end{theorem}

To the best of our knowledge,
statements about path localization
such as Theorem~\ref{thm:pathconcentration}
were not yet available in the literature of the Parabolic Anderson Model.
The scales above come out of our methods and may be artificial;
in particular, we do not know if $\ln t / \ln_3 t$ is the correct scaling for $\sup_{\epsilon_t t \le s \le t}|X_s - Z_t|$.

\subsection{Results: Scaling limit}
Our next theorem identifies the large-$t$ behavior of the pair of processes $t\mapsto Z_t$ and $t \mapsto \tfrac1t \ln U(t)$. 
While $U(t)$ is continuous, $Z_t$ is only c\`adl\`ag and thus it is natural to use
the Skorohod topology to discuss distributional convergence. 
Two relevant scales are
\begin{equation}\label{e:def_fundam_scales}
d_t := \frac{\rho}{d \ln t} \;\;\; \text{ and } \;\;\;  r_t:= \frac{t \, d_t }{\ln_3 t} =  \frac{\rho}{d \ln t}
\, \frac{t}{\ln_3t},
\end{equation}
marking, respectively, the size of fluctuations
of $\tfrac1t \ln U(t)$, and the typical size of $|Z_t|$.

To describe the scaling limit, consider a sample $\{(\lambda_i,z_i)\colon i\in\N\}$ from the Poisson point process on~$\R\times\R^d$  with intensity measure $\texte^{-\lambda}\textd\lambda\otimes\textd z$.
For~$\theta>0$, define
\begin{equation}\label{e:defpsitheta}
\psi_\theta(\lambda,z):=\lambda-\frac{|z|}{\theta}, \qquad (\lambda, z) \in \R \times \R^d.
\end{equation}
It can be checked that, for every~$\theta>0$, the set $\{\psi_\theta(\lambda_i,z_i)\colon i\in\N\}$ is bounded and locally finite.  Moreover, the maximizing point is unique at all but at most a countable set of~$\theta$'s and we can thus define $(\overline{\Lambda}_\theta, \overline{Z}_\theta)$ to be the c\`adl\`ag maximizer of $\psi_\theta$ over the sample points of the process
 (cf.\ Section~\ref{ss:order_stat_pen_func}). 
We set
\begin{equation}
\label{E:2.8ua}
\overline{\Psi}_\theta := \psi_\theta(\overline{\Lambda}_\theta, \overline{Z}_\theta).
\end{equation}
Then we have:

\begin{theorem}[Scaling limit of the localization process and the total mass]
\label{thm:locus}
There is a non-decreasing scale function $a_t>0$ obeying
\begin{equation}
\lim_{t \to \infty} \frac{a_t}{\ln_2 t} = \rho
\end{equation}
such that the following holds: 
The stochastic process $(Z_t)_{t > 0}$ in Theorems~\ref{thm:massconc} and \ref{thm:pathconcentration}
can be chosen such that, 
for all 
$s \in (0,\infty)$
and relative to the Skorohod topology on $\DD([s, \infty), \R \times \R^d)$,
\begin{equation}
\left( \frac{\tfrac{1}{\theta t} \ln U(\theta t) - a_{r_t}}{d_t}, \frac{Z_{\theta t}}{{r_t}} \right)_{\theta \in [s,\infty)} \,\,\underset{t\to\infty}{\overset{\text{\rm law}}\longrightarrow}\,\, \left(\overline{\Psi}_\theta, \overline{Z}_\theta \right)_{\theta \in [s, \infty)}.
\end{equation}
In particular, for each $\theta > 0$, the pair $([\tfrac{1}{\theta t}\ln U(\theta t) - a_{r_t}]/d_t, Z_{\theta t}/r_t)$ converges in law
to the pair $(\overline{\Psi}_\theta, \overline{Z}_\theta) \in \R \times \R^d$ whose coordinates are independent and distributed as follows: 
$\overline{\Psi}_\theta$ follows a Gumbel distribution with scale $1$ and location $d \ln(2\theta)$, 
while $\overline{Z}_\theta$ has i.i.d.\ coordinates, each of which is Laplace-distributed with location $0$ and scale $\theta$.
\end{theorem}

The scaling function~$a_t$ characterizes the leading-order scale 
of the principal Dirichlet eigenvalue of the Anderson Hamiltonian
in a box of radius $t$, as identified in \cite{BK16}.
See~\eqref{e:defat} below for a precise definition.

\subsection{Results: Aging}
The techniques used to prove the above theorems also 
permit us to address the phenomenon of \emph{aging} in the problem under consideration.
The term ``aging'' usually refers to the fact that certain decisive changes in the system occur 
at time scales that increase \emph{proportionally} to the age of the system. 
Our next result addresses aging in the process~$(Z_t)_{t>0}$:
\begin{theorem}[Aging for the localization process]
\label{thm:aging_locus}
For each $s>0$, and for $(Z_t)_{t > 0 }$ and $(\overline Z_t)_{t > 0 }$ as in Theorems~\ref{thm:massconc}, \ref{thm:pathconcentration} and~\ref{thm:locus},
\begin{equation}
\begin{aligned}
\label{e:aging_locus}
\lim_{t\to\infty}\textnormal{Prob}\bigl(Z_{t+\theta t}  =Z_t \; \forall \theta \in [0,s] \bigr)
& = \lim_{t\to\infty}\textnormal{Prob}\bigl(Z_{t+st}=Z_t\bigr) \\
& = \textnormal{Prob} \bigl( \overline{Z}_{1+s} = \overline{Z}_1 \bigr) = \textnormal{Prob} \left( \Theta > s \right),
\end{aligned}
\end{equation}
where the random variable
\begin{equation}\label{e:defTheta}
\Theta := \inf\{\theta>0 \colon\; \overline{Z}_{1+\theta} \neq \overline{Z}_1\}
\end{equation}
is positive and finite almost surely.
Moreover,
\begin{equation}\label{e:tailTheta}
\lim_{s \to \infty} \frac{s^d}{(\log s)^d} \Prob \left( \Theta > s \right)= \frac{d^d}{d!}.
\end{equation}
\end{theorem}
In light of Theorem~\ref{thm:locus}, Theorem~\ref{thm:aging_locus} can be seen as a reflection of the fact that 
the functional convergence stated in Theorem~\ref{thm:locus} is not achieved through 
a large number of microscopic jumps, but rather through sporadic macroscopic jumps.

Our second aging result deals with the jumps in the profile of the normalized solution $u(\cdot,t)/U(t)$.
It comes as a consequence of the mass concentration of the normalized solution around $Z_t$ together with Theorem~\ref{thm:aging_locus}.
\begin{theorem}[Aging for the solution]
\label{thm:aging_solution}
For any $\varepsilon \in (0,1)$, the random variable
\begin{equation}
\label{e:agingsolution}
\frac{1}{t} \inf \left\{ s > 0 \colon\, \sum_{x \in \Z^d} \left| \frac{u(x,t+s)}{U(t+s)} - \frac{u(x,t)}{U(t)} \right| > \varepsilon \right\}
\end{equation}
converges in distribution as $t \to \infty$ to the random variable $\Theta$ defined in~\eqref{e:defTheta}.
\end{theorem}

A key point to note about Theorem~\ref{thm:aging_solution} is that the limiting random variable does not depend on~$\varepsilon$. 
This suggests that, in fact, the sum in~\eqref{e:agingsolution} jumps from values near~$0$ to values near~$1$ 
as~$s$ varies in a time interval of sublinear length in~$t$.

\subsection{Results: Limit profiles}
The localization stated in Theorem~\ref{thm:massconc} can be given in a more precise form 
provided that we make an additional uniqueness assumption. 
In order to state this assumption, we need further definitions. Given a potential $V\colon\Z^d\to\R$, let
\begin{equation}
\label{defcL}
\LL(V) := \sum_{x \in \Z^d} \texte^{\frac{V(x)}{\rho}}.
\end{equation}
The functional $\LL$ plays the role of a large deviation rate function 
for random potentials~$\xi$ with doubly-exponential tails.
Whenever~$\LL(V) < \infty$ (in fact, whenever~$V(x)\to-\infty$ as $|x|\to\infty$), 
$\Delta + V$ has a compact resolvent as an operator on $\ell^2(\Z^d)$,
and its largest eigenvalue $\lambda^{\ssup 1}(V)$ is well-defined and simple.
The constant
\begin{equation}\label{e:variationalproblemchi}
\chi = \chi(\rho) := - \sup \{\lambda^{\ssup 1}(V) \colon\, V \in \R^{\Z^d},\, \LL(V) \le 1\} \; \in \, [0,2d]
\end{equation}
 is key in the analysis of the asymptotic growth of $U(t)$. 
The set of centered maximizers
\begin{equation}\label{e:defMrho}
\MM^*_{\rho} := \left\{ V \in \R^{\Z^ d} \colon\, 0 \in \textnormal{argmax}(V), \mathcal{L}(V) \le 1 \, \textnormal{ and } \, \lambda^{\ssup 1}(V) = - \chi \right\}
\end{equation}
is known to be non-empty.
The assumption below deals with uniqueness:
\begin{assumption}[Uniqueness of maximizer]
\label{A:uniqueness}
We assume that $\MM^*_{\rho} = \{V_\rho\}$, i.e., the variational problem~\eqref{e:variationalproblemchi} admits a unique centered solution~$V_\rho$.
\end{assumption}

The uniqueness of the centered minimizer is conjectured to hold for all $\rho > 0$, but has so far only been proved for~$\rho$ large enough; see \cite{GH99}.
In the latter paper it is also shown that, for any $V \in \MM^*_\rho$,
the non-negative principal eigenfunction of the operator $\Delta + V$ 
is strictly positive and lies in $\ell^1(\Z^d)$. 
Under Assumption~\eqref{A:uniqueness},
we will denote henceforth by $v_\rho$ the principal eigenfunction of $\Delta + V_\rho$, normalized so that
\begin{equation}
%
 v_\rho > 0
\quad\text{and}\quad \|v_\rho\|_{\ell^1(\Z^d)}=1.
\end{equation}
Then we have:
\begin{theorem}[Limiting profiles]
\label{thm:massconc_withuniq}
Suppose that Assumption~\ref{A:uniqueness} holds and let~$(Z_t)_{t > 0}$ be the process from 
Theorems~\ref{thm:massconc}, \ref{thm:pathconcentration} and~\ref{thm:locus}. 
There exist $\mu_t \in \N$ and $\widehat{a}_t > 0$ satisfying $\lim_{t \to \infty} \mu_t=\infty$ and $\lim_{t \to \infty} \widehat{a}_t/(\rho \ln_2 t)=1$
such that, for all $\epsilon\in(0,1)$,
\begin{equation}\label{e:shapepot_withuniq}
\sup_{s \in [\epsilon t,\, \epsilon^{-1}t]} \, \sup_{x \in \Z^d \colon |x| \le \mu_t} \bigl| \xi(x+Z_s)- \widehat{a}_t - V_\rho(x) \bigr| \,\underset{t\to\infty}\longrightarrow\, 0 \;\; \text{ in probability.}
\end{equation}
Moreover, for any $l_t>0$ satisfying $\lim_{t \to \infty} \frac{1}{t} l_t  =0 $,
\begin{equation}\label{e:massconc_withuniq}
\sup_{s \in [t-l_t,t+l_t]} \sum_{x \in \Z^d} \left| \frac{u(Z_t+x,s)}{U(s)} - v_\rho(x) \right|  \,\underset{t\to\infty}\longrightarrow\, 0 \;\; \text{ in probability.}
\end{equation}
\end{theorem}
The scale $\widehat{a}_t$ in~\eqref{e:shapepot_withuniq} coincides 
(up to terms that vanish as~$t\to\infty$) with the maximum of~$\xi$ inside a box of radius $t$ 
(see \eqref{e:defhata} for the definition, and also Lemma~\ref{l:maxpotential}).
Moreover, the scales $a_t$ and $\widehat{a}_t$
 (with $a_t$ as in Theorem~\ref{thm:locus}) satisfy $\lim_{t \to \infty} \widehat{a}_t - a_t = \chi$. The scale $\mu_t$ provided in the proof of Theorem~\ref{thm:massconc_withuniq}
satisfies $\mu_t \ll (\ln t)^\kappa$ for some arbitrary $\kappa < 1/d$,
but its actual rate of growth is not controlled explicitly.

\vspace{10pt}
The rest of the paper is organized as follows.
In Section~\ref{s:connections} below we discuss connections to the literature and provide some heuristics.
Section~\ref{s:proofmassconc} contains an extensive overview of our proofs
including the definition of the localization process $Z_t$.
The technical core of the paper is formed by Section~\ref{s:preparation} (properties of the potential and spectral bounds),
Section~\ref{s:pathexpansions} (path expansions) and Section~\ref{s:cost} (a point process approach).
The bulk of the proofs related to our main results is carried out in Sections~\ref{s:massdecomp}--\ref{s:localprofiles}, concerning respectively negligible contributions to the Feynman-Kac formula, localization of relevant eigenfunctions, path localization properties and the analysis of local profiles.
The proofs of some technical results are given in
Appendices~\ref{s:tailestimate}--\ref{s:propertiescost}.

\section{Connections and heuristics}
\label{s:connections}
\nopagebreak\noindent
In this section, we make connections to earlier work on this problem,
and also provide a short heuristic argument motivating the definition of the scales in \eqref{e:def_fundam_scales}.

\subsection{Relations to  earlier work}
\label{s:literature}\noindent
Let us give a quick survey on earlier works on the particular question that we consider; 
we refer to \cite{K15} for a comprehensive account on the parabolic Anderson model, 
and to \cite{M11} for a survey on certain aspects closely related to the present paper.

Since 1990, much of the effort went into developing a characterization of the logarithmic asymptotics of $t\mapsto U(t)$ and its moments, which are all finite if and only if all the positive exponential moments of $\xi(0)$ are finite. For this case, under a mild regularity assumption, \cite{HKM06} identified \emph{four universality classes} of asymptotic behaviors:
potentials with tails heavier than~\eqref{tail}
(corresponding formally to $\rho=\infty$),
double-exponential tails of the form~\eqref{tail},
the so-called ``almost bounded'' potentials (corresponding formally to $\rho=0$),
and bounded potentials. The first two cases were treated in \cite{GM98}, and the last two in \cite{HKM06} and \cite{BK01a},
respectively.
Potentials with infinite exponential moments were analysed in \cite{HMS08} (more precisely, Pareto and Weibull tails),
where weak limits and almost sure asymptotics for $U(t)$ were obtained.

In all of the classes mentioned above, the asymptotics of~$U(t)$ is expressed in terms of a variational principle for the local time of the path in~$Q_t^{\ssup\xi}$ and/or the ``profile'' of~$\xi$ that maximizes a local eigenvalue. The picture that emerges is that a typical path sampled from~$Q_t^{\ssup\xi}$ for~$t$ large will spend an overwhelming majority of time in a relatively small volume whose location is characterized by a favourable value of the local Dirichlet eigenvalue. Proofs of such statements have first been available 
for a related version of the model using the method of enlargement of obstacles~\cite{S98} 
and later also for the double-exponential class by probabilistic path expansions~\cite{GKM07}. 
However, neither of these approaches was sharp enough to distinguish among the many ``favourable eigenvalues.''
In fact, while the expectation was that only a finite number of such eigenvalues needs to be considered, the best available bound on their number was~$t^{o(1)}$.

For distributions with tails heavier than~\eqref{tail},
progress on the mass-concentration question has been made in \cite{KLMS09} and more recently in \cite{LM12, ST14, FM14}.
The distributions therein considered are, respectively, Pareto, exponential, Weibull with parameter $\gamma \in (0,2)$ and general Weibull.
In these papers it is proven that, with large probability, the solution is asymptotically concentrated on a single lattice point, 
which is an extremely strong localization property.
In the doubly-exponential case considered here, due to less-heavy tails, the localization phenomenon is not so strong; 
indeed, restricting to any bounded region misses some fraction of the total mass of the solution.

The analysis leading to our result depends crucially on the characterization of the order statistics of local principal eigenvalues for the Anderson Hamiltonian performed in \cite{BK16},
which allows us to conveniently represent local eigenvalues through a point process approach.
In this aspect, our paper shares similarities with \cite{FM14}, 
which draws heavily upon the analysis of the spectral order statistics in~\cite{A12,A13}. 
However, our case also harbors many significant differences, caused mainly by the non-degenerate structure of the dominant eigenfunctions.

For the remaining two universality classes of~$\xi$ --- namely, 
the bounded and ``almost bounded'' fields --- the mass-concentration question is yet more difficult 
because the relevant eigenvectors extend over spatial scales that diverge with time. 
Nevertheless, we believe that our approach could provide a strategy to study these cases as well.

\subsection{Some heuristics}
We present next a heuristic calculation based on \cite{BK16} to motivate the appearance of the scale~$r_t$ defined in~\eqref{e:def_fundam_scales}.
We will describe a strategy to obtain a lower bound for the total mass $U(t)$ defined in \eqref{e:deftotalmass}.
Our actual proof of the corresponding result (cf.\ Proposition~\ref{prop:lowerbound_forproof} below) follows similar but somewhat different steps.

Write $B_t \subset \Z^d$ for the $\ell^\infty$-ball with radius $t$,
and denote by $\lambda^{\ssup k}_{B_t}$, $\phi^{\ssup k}_{B_t}$, $1 \le k \le |B_t|$, 
the eigenvalues and corresponding orthonormal eigenfunctions of the Anderson Hamiltonian in $B_t$ with zero Dirichlet boundary conditions. If $Y^{\ssup k}_{B_t} \in B_t$ are points maximizing $|\phi^{\ssup k}_{B_t}|^2$,
it can be shown via spectral methods that

\begin{equation}
\E_{Y^{\ssup k}_{B_t}} \left[\e^{\int_0^t \xi(X_r) \textd r} \1\{X_r \in B_t \,\forall\, r \in [0,t]\} \right]
\gtrsim \e^{t \lambda^{\ssup k}_{B_t}}.
\end{equation}
Inserting in \eqref{e:deftotalmass} the event where the random walk
$X$ reaches $Y^{\ssup k}_{B_t}$ at a time $s<t$ and then remains in $B_t$ until time $t$,
and using the Markov property at time $s$, we obtain

\begin{align}\label{e:heurist1}
U(t) \ge \E_0\Bigl[\e^{\int_0^t \xi(X_r)\,\textd r}\1\{X_s= Y^{\ssup k}_{B_t}, X_r\in B_t \,\forall r\in[s,t]\}\Bigr] 
& \gtrsim \P_0(X_s = Y^{\ssup k}_{B_t} )\e^{(t-s)\lambda^{\ssup k}_{B_t}} \nonumber\\
&\approx \e^{-|Y^{\ssup k}_{B_t}|\ln (| Y^{\ssup k}_{B_t}|/s)}\e^{(t-s)\lambda^{\ssup k}_{B_t}},
\end{align}
 where for simplicity we assumed that $\xi$ is non-negative,
and to approximate the probability $\P_0(X_s= Y^{\ssup k}_{B_t})$, 
we assume $|Y^{\ssup k}_{B_t}|\gg s$.
Optimizing over $s$ gives the candidate $s=|Y^{\ssup k}_{B_t}|/\lambda^{\ssup k}_{B_t}$, which we may plug in \eqref{e:heurist1}
provided that we also assume $|Y^{\ssup k}_{B_t}|/\lambda^{\ssup k}_{B_t} < t$.
With this choice, \eqref{e:heurist1} becomes approximately
\begin{equation}\label{e:heurist2}
\exp \left\{ t \lambda^{\ssup k}_{B_t} - |Y^{\ssup k}_{B_t}| \ln \lambda^{\ssup k}_{B_t} \right\} = \e^{t a_t} \exp \left\{t d_t \frac{\lambda^{\ssup k}_{B_t} - a_t}{d_t} - |Y^{\ssup k}_{B_t}| \ln \lambda^{\ssup k}_{B_t} \right\},
\end{equation}
where $a_t \sim \rho \ln_2 t$ is the leading order of the principal Dirichlet eigenvalue of $H$ in a box of radius $t$
as identified in \cite{BK16} 
 (and is also the same scale appearing in Theorem~\ref{thm:locus}).
In \cite{BK16}, it is shown that the collection of rescaled points $\{ (\lambda^{\ssup k}_{B_t} - a_t)/d_t\}_{1 \le k \le |B_t|}$ 
converges in distribution to (the support of) a Poisson point process.
Assuming thus that $(\lambda^{\ssup k}_{B_t} - a_t)/d_t$ is of finite order,
an index $k$ optimizing \eqref{e:heurist2} should balance out the two competing terms,
implying $|Y^{\ssup k}_{B_t}| \approx r_t$.


\section{Main results from key propositions}
\label{s:proofmassconc}\nopagebreak\noindent
We give in this section an outline to the proof of Theorems~\ref{thm:massconc}, \ref{thm:pathconcentration}, \ref{thm:aging_solution} and~\ref{thm:massconc_withuniq}. 
This will be achieved by way of a sequence of propositions that encapsulate the key technical aspects of the whole argument. 
The proofs of these propositions and of Theorems~\ref{thm:locus}--\ref{thm:aging_locus} constitute the remainder of this paper and are the subject of Sections~\ref{s:preparation}--\ref{s:localprofiles} as well as the three appendices.
Note that Theorem~\ref{thm:aging_locus} will be assumed in Sections~\ref{ss:proofconcentration}--\ref{ss:proofaging} below.

Throughout the rest of this work, we set $\N:=\{1,2,\ldots\}$ and $\N_0:= \N \cup \{0\}$.
We denote by $\dist(\cdot, \cdot)$ the metric derived from the $\ell^1$-norm $|\cdot|$, and by $\diam(\cdot)$ the corresponding diameter.
For a real-valued function~$f$ and a positive function~$g$,
we write $f(t) = O(g(t))$ as $t \to \infty$ to denote that there exists $C >0$ such that $|f(t)| \le C g(t)$ for all large enough $t$,
and we write $f(t) = o(g(t))$ in place of $\lim_{t \to \infty} |f(t)|/g(t) = 0$.
In the latter case, we may also alternatively write $|f(t)| \ll g(t)$ or $g(t) \gg |f(t)|$.
By $o(\cdot)$ or $O(\cdot)$ we will always mean \emph{deterministic} bounds, i.e., independent of the realization of $\xi$.

\subsection{Definition of the localization process}
\label{ss:concentration_locus}\noindent
For $\Lambda \subset \Z^d$ finite, we denote by $\lambda^{\ssup 1}_\Lambda$ 
the largest Dirichlet eigenvalue (i.e., with zero boundary conditions)
of $\Delta + \xi$ in $\Lambda$.
For $L \in \N$ and $x \in \Z^d$,  we let
\begin{equation}
B_L(x) := x+[-L,L]^d \cap \Z^d,
\end{equation}
and when $x=0$ we write $B_L$ instead of $B_L(0)$.

Fix $ \kappa \in (0,1/d)$. 
For each $z \in \Z^d$, we define a $\xi$-dependent radius
\begin{equation}\label{e:defvarrhoz}
\varrho_z := \left\lfloor \exp \left\{ \frac{\kappa}{\rho} \, \xi(z) \right\} \right\rfloor
\end{equation}
and we let
\begin{equation}\label{e:defcapitals}
\scrC := \left\{z \in \Z^d \colon\, \xi(z) \ge \xi(y) \;\forall\, y \in B_{\varrho_z}(z) \right\}
\end{equation}
denote the set of local maxima of $\xi$ in neighborhoods of radius $\varrho_z$,
which we call \emph{capitals}.
For $z \in \scrC$, we abbreviate
\begin{equation}\label{e:deflambdascrCz}
\lambda^\scrC(z) := \lambda^{\ssup 1}_{B_{\varrho_z}(z)}.
\end{equation}
For $t > 0$, we define a \emph{cost functional} over the points $z \in \scrC$ by setting 
\begin{equation}\label{e:defPsi}
\Psi_t(z) := \lambda^{\scrC}(z) - \frac{\ln^+_3 |z|}{t}|z|, \;\;\;\; \text{ where } \ln_3^+ x := \ln_3 (x \vee \texte^\texte).
\end{equation}
\noindent
The functional $\Psi_t$ measures the relevance at time $t$ of a capital $z \in \scrC$ by weighting 
the principal eigenvalue in $B_{\varrho_z}(z)$ against the $\ell^1$-distance to the origin $|z|$.
The next proposition shows that $\Psi_t$ admits a maximizer:

\begin{proposition}\label{prop:welldefined}
Almost surely, $|\scrC|=\infty$ and,
for all $t > 0$ and all $\eta \in \R$,
\begin{equation}
\label{e:welldefined}
\left| \left\{z \in \scrC \colon\, \Psi_t(z) > \eta \right\}\right| < \infty.
\end{equation}
\end{proposition}
\noindent
The proof of Proposition~\ref{prop:welldefined}
will be given in Section~\ref{s:preparation}.
In order to define $Z_t$ as a c\`adl\`ag maximizer of $\Psi_{t}$, we proceed as follows.
Write $(\lambda,z)\succeq(\lambda',z')$ 
for the usual lexicographical order of $\R \times \R^d$,
i.e., $(\lambda, z) \succeq (\lambda', z')$ if either
$\lambda > \lambda'$, or $\lambda = \lambda'$ and $z \succeq z'$ according to the usual (non-strict) lexicographical order of $\R^d$.
Now define,
recursively for $k \in \N$,
\begin{equation}
\Psi^{\ssup k}_t := \sup_{z \in \scrC \setminus \{Z^{\ssup 1}_t, \ldots, Z^{\ssup {k-1}}_t\} } \Psi_t(z), \label{e:defPsik}
\end{equation}
\begin{equation}
\mathfrak{S}^{\ssup k}_t  := \left\{ z \in \scrC \setminus \{Z^{\ssup 1}_t, \ldots, Z^{\ssup {k-1}}_t\} \colon\, \Psi_t(z) = \Psi^{\ssup k}_t \right\},
\end{equation}
and
\begin{align}
\label{e:defZk}
Z^{\ssup k}_t \; \in \, \left\{ z \in \mathfrak{S}^{\ssup k}_t \colon\, \bigl(\lambda^{\scrC}(z), z \bigr) \succeq \bigl(\lambda^{\scrC}(\hat{z}), \hat{z} \bigr) \; \forall\,  \hat{z} \in \mathfrak{S}^{\ssup k}_{t} \right\}. 
\end{align}
Observe that~\eqref{e:defZk} determines~$Z^{\ssup k}_t$ uniquely.
Then we set
\begin{equation}\label{defZt}
Z_t := Z^{\ssup 1}_t.
\end{equation}
The above definitions ensure that the maps $t \mapsto \Psi^{\ssup k}_t$ are continuous while $t \mapsto Z^{\ssup k}_t$ are c\`adl\`ag, with $t \mapsto |Z_t|$ non-decreasing (see Lemma~\ref{l:propPhi} and~\eqref{e:reprZt} below).
 We point out that the choice of $\kappa$ in \eqref{e:defvarrhoz} is of minor relevance,
not affecting the asymptotic behaviour of $\Psi_t$ or its maximizers.

Note that we can have $B_{\varrho_z}(z) \cap B_{\varrho_{z'}(z')} \neq \emptyset$ for distinct $z, z' \in \scrC$.
Nevertheless, as is shown next, the relevant points of $\scrC$ are well-separated with large probability:
\begin{proposition}[Separation of relevant capitals]
\label{prop:seprelcap}
There exist subsets $\scrC_t \subset \scrC$
such that, for any $k \in \N$, $\beta \in (0,1)$, and $0<a \le b < \infty$, with probability tending to $1$ as $t \to \infty$,
\begin{equation}\label{e:seprelcap}
\{Z^{\ssup 1}_s, \ldots, Z^{\ssup k}_s \} \subset \scrC_t \; \forall \, s \in [at, bt] \;\;\; \text{ and } \;\;\; \dist(z, z') > t^\beta \, \text{ for all distinct } z, z' \in \scrC_t.
\end{equation}
\end{proposition}
\noindent
Proposition~\ref{prop:seprelcap} will be proved in Section~\ref{s:cost}.

\begin{remark}\label{r:formfunctional}
It would have been perhaps more natural to define $\Psi_t$
with $\ln_3^+|z|$ substituted by $\ln \lambda^{\scrC}(z)$,
which is a form that appears in the literature
(see also the proof of Proposition~\ref{prop:lowerbound_forproof}).
The analysis is slightly simpler with our definition,
cf.\ Section~\ref{s:cost} below.
Substituting however $\ln_3^+ |z|$ by $\ln_3 t$ (which is the leading order of $\ln_3^+ |Z_t|$)
would not be as convenient, as this would complicate our proof of functional convergence.
\end{remark}

\subsection{Properties of the cost functional}
\label{ss:propertiescost}\noindent
The technical statements start with a discussion of the properties of the above cost functional~$\Psi_{t}$ and the process~$Z_t$. 
Recall the definitions of~$r_t$ and $d_t$ from~\eqref{e:def_fundam_scales}.
The various error estimates that are to follow will require a host of auxiliary scales. 
First we fix $t \mapsto \epsilon_t \in (0,1)$, $\epsilon_t \gg (\ln_3 t)^{-1}$ arbitrary as in the statement of Theorem~\ref{thm:pathconcentration}; 
 note that $\epsilon_t$ may converge to $0$.
Then, similarly to \cite{MP16}, 
we fix~$e_t$, $f_t$, $g_t$, $h_t$ and~$b_t$ such that
\begin{equation}
\label{e:relscales1}
e_t,f_t,h_t,b_t\,\underset{t\to\infty}\longrightarrow\,0\quad\text{and}\quad
g_t\,\underset{t\to\infty}\longrightarrow\,\infty
\end{equation}
while also
\begin{equation}\label{e:relation_scales}
\frac{g_t}{\epsilon_t \ln_3 t} \ll b_t \ll f_t h_t
\quad\text{and}\quad g_t h_t \ll  e_t.
\end{equation}
As an example of scales satisfying \twoeqref{e:relscales1}{e:relation_scales}, one may take 
suitable
powers of $\epsilon_t \ln_3 t$. We then have:

\begin{proposition}
\label{prop:goodevent_forproof}
Fix $0<a \le b < \infty$. Then, with probability tending to $1$ as $t \to \infty$,
\RV
\begin{equation}
\inf_{s \in [at, bt]} \Psi^{\ssup 1}_s > \left(\rho + o(1) \right) \ln_2 t,
\end{equation}
\eRV
\begin{equation}
\label{e:defgoodevent_forproof}
\left(\Psi^{\ssup 1}_{at}-\Psi^{\ssup 2}_{at} \right) \wedge \left(\Psi^{\ssup 1}_{bt}-\Psi^{\ssup 2}_{bt} \right) > d_te_t
\end{equation}
and
\begin{equation}\label{e:boundsZ}
\quad r_t f_t <  \inf_{s \in [at,bt]}\, |Z_s| \le \sup_{s \in [at,bt]}\, |Z_s| < r_t g_t.
\end{equation}
\end{proposition}

Proposition~\ref{prop:goodevent_forproof} is proved in Section~\ref{s:cost},
together with Theorems~\ref{thm:locus}--\ref{thm:aging_locus}.
The proofs rely strongly on the extreme order statistics of the principal Dirichlet eigenvalue in a box identified in \cite{BK16} 
and, similarly to the approach of \cite{KLMS09, MOS11, LM12, ST14, FM14, MP16}, on a Poisson point process approximation.
However, in order to deal with the fact that the local eigenvalues do not depend on bounded regions in space, 
a coarse-graining scheme taken from \cite{BK16} is required.
Our approach provides a quite direct implication of functional convergence and aging for $Z_t$
from the convergence of the underlying point process (in a suitable topology),
see in particular Lemmas~\ref{l:PPPconvcompact}, \ref{l:continuityPhiskorohod} and \ref{l:convergencenojumps} below.
We believe that this approach could be useful to prove analogous results in other contexts, e.g., 
the PAM with lighter potential tails.

Notice that in~\eqref{e:defgoodevent_forproof} we only require a gap between $\Psi^{\ssup 1}_{s}$ and $\Psi^{\ssup 2}_{s}$
for $s \in \{at, bt\}$. 
This is because, 
while the gap is greater than $d_t e_t$ with large probability 
at both $at$ and $bt$, 
there is by~\eqref{e:aging_locus} 
a non-zero probability that $s \mapsto Z_{s}$ jumps in the interval $[at,bt]$,
leading to a zero gap at the jump time.
Notwithstanding, if no such jump occurs,
then the gap remains uniformly positive throughout the interval.
Indeed, define
\begin{equation}
\label{e:defgapevent}
\GG_{t,s} := \left\{ \Psi^{\ssup 1}_{s} - \Psi^{\ssup 2}_{s} \ge d_t e_t \right\}.
\end{equation}
Then we have:
\begin{proposition}
\label{prop:stabilitygap}
With probability one,
for any $0 < a \le b < \infty$ and any $t > 0$,
\begin{equation}\label{e:stabilitygap}
\GG_{t, at} \cap \GG_{t, bt} \cap \left\{Z_{at}=Z_{bt} \right\} = \bigcap_{s \in [at,bt]} \bigl(\GG_{t, s} \cap \{Z_s = Z_{at} \}\bigr).
\end{equation}
\end{proposition}
\noindent
The proof of Proposition~\ref{prop:stabilitygap}
is related to that of Theorem~\ref{thm:aging_locus}, and so it
is relegated to Section~\ref{s:cost} as well.

\subsection{Mass decomposition and negligible contributions}
\label{ss:massdecomp}\noindent
Having dealt with the cost functional and localization process, 
we proceed by giving estimates on the solution to \twoeqref{PAM}{PAMinitial}. 
As noted already earlier, this solution can be written using the Feynman-Kac formula 
\eqref{FK}, which offers the strategy to control $u(t,x)$ 
by decomposing the expectation based on various restrictions on the underlying random walk. 
A starting point is a good lower bound on the total mass~$U(t)$:

\begin{proposition}\label{prop:lowerbound_forproof}
For any $0<a \le b < \infty$,
\begin{equation}
\label{e:lowerbound_forproof}
\inf_{s \in [at,bt]}\Bigl\{\ln U(s) - s \Psi^{\ssup 1}_s\Bigr\} \ge  o(t  d_t  b_t \epsilon_t)
\end{equation}
holds with probability tending to $1$ as $t\to\infty$.
\end{proposition}

For $\Lambda \subset \Z^d$, let
\begin{equation}\label{e:deftauLambda}
\tau_{\Lambda} := \inf \{s > 0 \colon\, X_s \in \Lambda\}
\end{equation}
denote the first hitting time of $\Lambda$ \RS by the random walk $X$. \eRS 
Our decomposition of~\eqref{FK} begins by restricting the expectation to paths that never leave a box of side-length
\begin{equation}\label{e:defLt}
L_t := \lfloor t \ln^+_2 t\rfloor, \quad \text{ where } \; \ln^+_2 t := \ln_2 (t \vee \texte).
\end{equation}
This restriction comes at little loss since we have:
\begin{proposition}
\label{prop:macroboxtruncation_forproof}
For any $0<a\le b < \infty$, there is a $t_0=t_0(\xi)$ with $t_0<\infty$ a.s.\ such that
\begin{equation}
\label{e:macroboxtruncation_forproof}
\sup_{s \in [at,bt]}\ln \E_0 \left[ \texte^{\int_0^s \xi(X_u) \textd u } \1\{ \tau_{B^\cc_{L_t}} \le s\} \right] \le 
-\frac18 t(\ln_2 t)\ln_3 t
\end{equation}
holds whenever $t>t_0$.
\end{proposition}

Next we show that the bulk of the contribution to
the Feynman-Kac formula comes from paths that do not 
even leave the random domain
\begin{equation}\label{e:def_innerdomain}
D^\circ_{t,s} := \bigl\{x \in \Z^d \colon\, |x| \le |Z_s|(1+h_t)\bigr \}.
\end{equation}
Indeed, the contribution of paths that leave this set is bounded via: 
\begin{proposition}\label{prop:randomtruncation_forproof}
For any $0<a \le b < \infty$,
\begin{multline}
\qquad
\sup_{s \in [at,bt]} \biggl\{\ln \E_0 \left[ \texte^{\int_0^{s} \xi(X_u) \textd u } \1 \{ \tau_{(D^\circ_{t,s})^\cc} \le s < \tau_{B^\cc_{L_t}} \} \right]
\\
- \max \left\{ s \Psi^{\ssup 2}_{s} \,,\; s \Psi^{\ssup 1}_{s} - h_t |Z_s| \ln_3 t \right\} \biggr\}\le o(t d_t b_t)
\qquad
\end{multline}
holds with probability tending to $1$ as $t \to \infty$.
\end{proposition}
\noindent
We also control the contribution of paths that do not enter a fixed 
neighborhood of~$Z_t$:

\begin{proposition}
\label{prop:pathsavoidingZ_forproof}
For all large enough $\nu \in \N$ and all $0 < a \le b < \infty$,
\begin{equation}
\sup_{s \in [at,bt]} \left\{\ln \E_0 \left[ \texte^{\int_0^{s} \xi(X_u) \textd u } \1 \{ \tau_{B_\nu(Z_s)} \wedge \tau_{B_{L_t}^\cc} > s \} \right]
- s \Psi^{\ssup 2}_s
\right\} \le o(t d_t b_t)
\end{equation}
holds with probability tending to $1$ as $t \to \infty$.
\end{proposition}

The above propositions will allow us to restrict the Feynman-Kac formula to the event
\begin{equation}\label{e:relevantevent}
\mathcal{R}^\nu_{t,s} := \left\{\tau_{(D^\circ_{t,s})^\cc} > s \ge \tau_{B_{\nu}(Z_s)}\right\},
\end{equation}
and proceed to control the result using spectral techniques; see Section~\ref{ss:localization}.

Our proofs of Propositions~\ref{prop:lowerbound_forproof} and~\ref{prop:macroboxtruncation_forproof},
given respectively in Sections~\ref{ss:lower_bound} and~\ref{ss:macrobox_truncation},
are relatively simple and follow similar results in the literature.
Propositions~\ref{prop:randomtruncation_forproof} and~\ref{prop:pathsavoidingZ_forproof} 
are proven in Section~\ref{ss:negligcontrib};
their main technical point is a path expansion scheme developed in Section~\ref{s:pathexpansions},
based on an approach from \cite{MP16}.
Additional difficulties arise in our case
due to smaller gaps in the potential, and to
the fact that the effective support of the relevant local eigenvalues is unbounded in the limit of large times.
This is
overcome through a careful analysis of the connectivity properties of the level sets of the potential
and their implications for the bounds derived via path expansions.

\RS An important observation is \eRS that $\lambda^\scrC(Z_s)$
is the largest possible over all capitals inside $D^\circ_{t,s}$ (cf.\ Lemma~\ref{l:spectralgap}). 
This comes as a consequence of the choice of $h_t$ in~\eqref{e:relation_scales}, which is of special relevance
as it simultaneously allows the proofs of Proposition~\ref{prop:randomtruncation_forproof} above
(for which $h_t$ should be large enough) and Proposition~\ref{prop:localizationinnerdomain} below
(for which $h_t$ should be small enough).
We also note that a \RV complementary upper bound \eRV to~\eqref{e:lowerbound_forproof} holds as well
(cf.\ Lemma~\ref{l:upperboundUt}), 
which will be important for the proof of Theorem~\ref{thm:locus} in Section~\ref{ss:upperboundUt}.

\subsection{Localization}
\label{ss:localization}\noindent
Once the path has been shown to enter a neighborhood of~$Z_t$ by time~$t$ with large probability, the next 
item of concern is to show that it will actually not be found far away from~$Z_t$ at time~$t$. This will be 
done 
by bounding the end-point distribution using the principal eigenfunction $\phi^\circ_{t,s}$ corresponding to the largest Dirichlet eigenvalue of the Anderson Hamiltonian in~$D^\circ_{t,s}$,
which we assume to be normalised so that
\begin{equation}
\phi^\circ_{t,s} > 0 \text{ on } D^\circ_{t,s}, \quad \phi^\circ_{t,s} =  0 \text{ on } (D^\circ_{t,s})^\cc \quad \text{and} \quad \|\phi^\circ_{t,s}\|_{\ell^2(\Z^d)}=1.
\end{equation}
We have:
\begin{proposition}\label{prop:boundbyprincipalef_forproof}
For any $\nu \in \N$ and $0 < a \le b < \infty$, 
the following holds with probability tending to $1$ as $t \to \infty$: For
all $s \in [at, bt]$ and all $x \in D^\circ_{t,s}$,
\begin{equation}\label{e:boundbyprincipalef_forproof}
\E_0 \left[ \texte^{\int_0^s \xi(X_u) \textd u} \1_{\mathcal{R}^\nu_{t,s} \cap \{X_s = x\}} \right]
\le U(s) \sup_{y \in B_\nu(Z_s)} \left\{\phi^\circ_{t,s}(y)^{-3} \right\} \, \phi^\circ_{t,s}(x).
\end{equation}
\end{proposition}

In order to use the bound in~\eqref{e:boundbyprincipalef_forproof}, we will need an estimate on the decay of~$\phi^\circ_{t,s}$ away
from~$Z_s$. On the event~$\GG_{t,s}$ from~\eqref{e:defgapevent}, this
is the subject of:
\begin{proposition}
\label{prop:localizationinnerdomain}
There exist $c_1, c_2>0$ and,
for all $\nu \in \N$, also
$\varepsilon_\nu >0$ such that,
for all $0 < a \le b < \infty$,
the following holds on with probability tending to~$1$ as $t \to \infty$: 
For all $s \in [at,bt]$, on $\GG_{t,s}$ we have
\begin{align}
\text{(i) } \;\;\; \phi^\circ_{t,s}(x) & \le c_1 \texte^{-c_2|x-Z_s|} \, \quad \forall  x \in \Z^d,\label{e:localizationinnerdomain1}\\
\text{(ii) } \;\;\; \phi^\circ_{t,s}(y) & \ge \varepsilon_\nu \qquad \qquad \quad \, \forall  y \in B_\nu(Z_s). \label{e:localizationinnerdomain2}
\end{align}
\end{proposition}

Propositions~\ref{prop:boundbyprincipalef_forproof}--\ref{prop:localizationinnerdomain} are proven in Section~\ref{s:localization}.
Proposition~\ref{prop:boundbyprincipalef_forproof}
is similar to Proposition~3.11 in~\cite{MP16}, and is obtained by adaptation of \cite[Theorem~4.1]{GKM07}.
The proof of Proposition~\ref{prop:localizationinnerdomain}(i)
an adaptation of \cite[Theorem~1.4]{BK16},
while part (ii) relies on results from \cite{GM98}, \cite{GH99} and \cite{GKM07}
regarding the optimal shapes of the potential.


\subsection{Proof of mass concentration results}
\label{ss:proofconcentration}\noindent
We have now amassed enough information for the proof of Theorem~\ref{thm:massconc},
assuming Theorem~\ref{thm:aging_locus} and the above propositions:
\begin{proof}[Proof of Theorem~\ref{thm:massconc}]
Fix $\nu \in \N$ large enough 
so that Proposition~\ref{prop:pathsavoidingZ_forproof} is available.
Fix $0<a \le b < \infty$. 
We will first
show that, for all $\delta>0$, there exists an $R \in \N$ such that
\begin{equation}\label{e:massconcwithgap}
\lim_{t \to \infty} \textnormal{Prob} \left(\exists \, s \in [at, bt] \colon\, \Psi^{\ssup 1}_s - \Psi^{\ssup 2}_s \ge d_t e_t, \,  Q^{\ssup \xi}_s \left(|X_s - Z_s| > R \right) > \delta \right) = 0,
\end{equation}
and derive the desired claim from this at the very end.

We begin by noting that Propositions~\ref{prop:lowerbound_forproof}--\ref{prop:pathsavoidingZ_forproof} imply that
\begin{multline}
\label{e:proofthm_massconc1}
\qquad
\ln \left( \frac{1}{U(s)} \E_0 \left[ \texte^{\int_0^s \xi(X_u) \textd u} \1_{(\mathcal{R}^\nu_{t,s})^\cc} \right] \right) \\
\le - s \min\left\{\Psi^{\ssup 1}_s-\Psi^{\ssup 2}_s, \, h_t |Z_s| \ln_3 t, \, 
\frac{t \ln_2 t\ln_3 t}{8s} \RV + \Psi^{\ssup 1}_s \eRV \right\} + o(t d_t b_t)
\end{multline}
holds true for all $s \in [at,bt]$ \RS with probability tending to $1$ as $t \to \infty$. \eRS
By \RV Proposition~\ref{prop:goodevent_forproof}, \eRV
on $\GG_{t,s} = \{\Psi^{\ssup 1}_s - \Psi^{\ssup 2}_s \ge d_t e_t \}$
we may further bound~\eqref{e:proofthm_massconc1} by 
\begin{equation}\label{e:proofthm_massconc2}
- at \min\left\{d_t e_t, \, h_t r_t f_t \ln_3 t, \, \RV \tfrac12 \rho \ln_2 t \eRV \right\} + o(t d_t b_t)
\end{equation}
which goes to $-\infty$ as $t \to \infty$ by~\eqref{e:def_fundam_scales} and~\eqref{e:relation_scales} 
--- indeed,~\eqref{e:relation_scales} shows that $e_t\ln_3t\to\infty$ (in fact, $e_t\gg g_t/\ln_3t$ with $g_t\to\infty$) and so $td_te_t\gg ct/[(\ln t)\ln_3 t]$ --- implying that
\begin{equation}\label{e:proofthm_massconc3}
\lim_{t \to \infty} \sup_{s \in [at, bt]} \frac{\1_{\GG_{t,s}}}{U(s)} \E_0 \left[ \texte^{\int_0^s \xi(X_u) \textd u} \1_{(\mathcal{R}^\nu_{t,s})^\cc} \right] = 0 \;\;\; \text{ in probability.}
\end{equation}
Fix now $\delta>0$ and let $R \in \N$ be large enough such that
\begin{equation}\label{e:proofthm_massconc4}
\varepsilon_\nu^{-3} c_1 \sum_{|x| > R} \texte^{-c_2 |x|} < \frac{\delta}{2},
\end{equation}
where $c_1$, $c_2$ and $\varepsilon_\nu$ are as in Proposition~\ref{prop:localizationinnerdomain}.
By Propositions~\ref{prop:boundbyprincipalef_forproof}--\ref{prop:localizationinnerdomain},
\begin{equation}\label{e:proofthm_massconc5}
\sup_{s \in [at,bt]}\frac{\1_{\GG_{t,s}}}{U(s)} \sum_{x\colon |x-Z_s| > R} \E_0 \left[ \texte^{\int_0^s \xi(X_u) \textd u} \1_{\mathcal{R}^\nu_{t,s} \cap \{X_s = x \}} \right] < \frac{\delta}{2}
\end{equation}
with probability tending to $1$ as $t \to \infty$,
which together with~\eqref{e:proofthm_massconc3} implies~\eqref{e:massconcwithgap}.

To conclude the desired statement from~\eqref{e:massconcwithgap},
fix $l_t > 0$, $l_t = o(t)$ and note that, by Theorem~\ref{thm:aging_locus} and
Propositions~\ref{prop:goodevent_forproof}--\ref{prop:stabilitygap},
with probability tending to $1$ as $t \to \infty$,
\begin{equation}\label{e:nojumpandgap}
Z_s = Z_t \; \text{ and } \; \Psi^{\ssup 1}_s - \Psi^{\ssup 2}_s \ge d_t e_t\qquad \forall  s \in [t - l_t, t+l_t].
\end{equation}
This together with~\eqref{e:massconcwithgap} (with $a<1<b$) implies~\eqref{e:massconc}.
\end{proof}

\RV
The presence of the scale $\epsilon_t$ in \eqref{e:lowerbound_forproof}
was not needed in the proof above, but it will be important for the proof of Theorem~\ref{thm:pathconcentration}.
More precisely, it will be used to obtain the following improvement of Proposition~\ref{prop:pathsavoidingZ_forproof}:
\eRV
\begin{proposition}\label{prop:closetoZinshorttime}
For all sufficiently large $\nu \in \N$,
\begin{equation}\label{e:closetoZinshorttime}
\frac{1}{U(t)} \E_0 \left[ \texte^{\int_0^t \xi(X_s) \textd s} \1\{ \tau_{(D^\circ_{t,t})^\cc} > t \ge \tau_{B_\nu(Z_t)} > \epsilon_t t\} \right] \,\underset{t \to \infty}\longrightarrow\, 0 \quad \text{ in probability.}
\end{equation}
\end{proposition}

We will also need the following proposition, 
which bounds the contribution of paths starting at a point $x \in B_\nu(Z_t)$
and reaching a distance greater than $\tfrac12 \epsilon_t \ln t$:
\begin{proposition}\label{prop:boundforpathconc}
\RV For any $k \in \N$ \eRV and any $\nu \in \N$,
the following holds with probability tending to $1$ as $t \to \infty$: 
For all $x \in B_\nu(Z_t)$ and all $0 \le s \le t$,
\begin{equation}
\label{e:boundforpathconc}
\E_x \left[ \texte^{\int_0^s \xi(X_u) \textd u} \1 \Big\{ \tau_{(D^{\circ}_{t,t})^\cc} > s, \, \sup_{0\le u \le s}|X_u-x|> \tfrac12 \epsilon_t \ln t  \Big\} \right]  \le  \RV t^{-k} \eRV \, \E_x \left[ \texte^{\int_0^s \xi(X_u) \textd u} \right].
\end{equation}
\end{proposition}

\RS Propositions~\ref{prop:closetoZinshorttime}--\ref{prop:boundforpathconc} will be proved in Section~\ref{s:pathconc}.
They allow us to give: \eRS

\begin{proof}[Proof of Theorem~\ref{thm:pathconcentration}]
Fix $\nu \in \N$ large enough so that the conclusion of Proposition~\ref{prop:closetoZinshorttime} becomes available.
Write $\widetilde{\tau} := \tau_{B_\nu(Z_t)}$ and 
note that, \RS since $\epsilon_t \gg (\ln_3 t)^{-1}$, \eRS when $t$ is large,
\begin{equation}
\label{e:prpathconc1}
\left\{ \sup_{s \in [\epsilon_t t, t]} \left| X_s - Z_t \right| > \epsilon_t \ln t \right\} 
\subset (\RR^\nu_{t,t})^\cc \cup \left\{\tau_{(D^{\circ}_{t,t})^\cc} > t \ge \widetilde{\tau} > \epsilon_t t \right\} \cup A_t,
\end{equation}
where 
\begin{equation}
A_t:=\left\{ \tau_{(D^{\circ}_{t,t})^\cc} > t, \widetilde{\tau} \le \epsilon_t t, \, \sup_{s \in [\widetilde{\tau}, t]} \left| X_s - X_{\widetilde{\tau}}\right| > \tfrac12 \epsilon_t \ln t \right\}.
\end{equation}
By~\eqref{e:proofthm_massconc3}, Proposition~\ref{prop:goodevent_forproof} and Proposition~\ref{prop:closetoZinshorttime},
\begin{equation}\label{e:prpathconc2}
Q^{\ssup \xi}_t \left( (\RR^{\nu}_{t,t})^\cc \right) \; \vee \; Q^{\ssup \xi}_t \left(\tau_{(D^\circ_{t,t})^\cc} > t \ge \widetilde{\tau} > \epsilon_t t \right) \underset{t\to\infty}\longrightarrow\, 0 \; \text{ in probability.}
\end{equation}
To control $Q^{\ssup \xi}_t(A_t)$, let 
\begin{equation}
G_t(x,s)
:= \E_x \left[ \texte^{\int_0^{s} \xi(X_u) \textd u} \1{\scriptstyle \big\{\tau_{(D^\circ_{t,t})^{\cc}}>s, \; \sup_{0 \le u \le s} |X_u - x| > \tfrac12 \epsilon_t \ln t \big\}} \right]
\end{equation}
and use the strong Markov property and Proposition~\ref{prop:boundforpathconc} to get
\begin{align}
\label{e:prpathconc3}
\E_0 \Bigl[\texte^{\int_0^t \xi(X_s) \textd s} \1_{A_t}\Bigr] 
& = \sum_{x \in B_{\nu}(Z_t)} \E_0 \Bigl[ \texte^{\int_0^{\widetilde{\tau}} \xi(X_s) \textd s} \1_{\{\tau_{(D^\circ_{t,t})^\cc}> \widetilde{\tau} =\tau_x \le \epsilon_t t\}} G_t(x,t-\widetilde\tau) \Bigr]
\nonumber \\
& \le t^{-1} U(t)
\end{align}
with probability tending to $1$ as $t \to \infty$.
The desired claim now readily follows from \eqref{e:prpathconc1}, \eqref{e:prpathconc2} and~\eqref{e:prpathconc3}. 
\end{proof}


\subsection{Proof of aging and limit profiles}
\label{ss:proofaging}\noindent
The last set of propositions to be introduced here concern the proof of Theorems~\ref{thm:aging_solution} and~\ref{thm:massconc_withuniq}. 
We start with some supporting notation. 
Given a function $t\mapsto\mu_t$ with $\mu_t\in \N$, let~$\phi^\bullet_{t,s}$ 
denote the eigenfunction corresponding to the largest Dirichlet eigenvalue of the Anderson operator
in~$B_{\mu_t}(Z_s)$, normalised so that
\begin{equation}
\phi^\bullet_{t,s} > 0 \text{ on } B_{\mu_t}(Z_s), \quad \phi^\bullet_{t,s} = 0 \text{ on } B_{\mu_t}^\cc(Z_s) \quad\text{and}\quad
\|\phi^{\bullet}_{t,s}\|_{\ell^1(\Z^ d)}=1.
\end{equation}
(Notice our use of the $\ell^1$-norm here.) 
When $s=t$ we omit one index from the notation. 
Recall the choice of $\kappa \in (0,1/d)$ in~\eqref{e:defvarrhoz}.
We then have:
\begin{proposition}\label{prop:improvmassconc}
For any $\mu_t \in \N$ with $1 \ll \mu_t \ll (\ln t)^\kappa$, and any $0 < a \le b < \infty$,
\begin{equation}\label{e:improvmassconc}
\lim_{t \to \infty} \sup_{s \in [at,bt]} \1_{\GG_{t,s}} \, \left\| \frac{u(\cdot,s)}{U(s)} - \phi^\bullet_{t,s}(\cdot) \right\|_{\ell^1(\Z^d)} = 0 \;\; \textnormal{ in probability.}
\end{equation}
\end{proposition}
We may thus obtain information about the profile of $u(\cdot, s)$
via that of $\phi^\bullet_{t,s}$.
\RV As shown next, the latter can be controlled under Assumption~\ref{A:uniqueness},
along with the shape of $\xi$: \eRV
\begin{proposition}\label{prop:shapes}
If Assumption~\ref{A:uniqueness} holds,
then there exists $\mu_t\in \N$ 
with $1 \ll \mu_t \ll (\ln t)^\kappa$ 
and a function $\widehat a_t$
satisfying $\lim_{t \to \infty}\widehat{a}_t/ \ln_2 t = \rho$
such that,
for any $0<a \le b < \infty$, both
\begin{equation}\label{e:shapepotdeps}
\sup_{s \in [at,bt]} \sup_{x \in B_{\mu_t}} \left| \xi(x + Z_s) - \widehat{a}_t - V_\rho(x) \right|
\end{equation}
and
\begin{equation}\label{e:efcomparison_uniq}
\sup_{s \in [at,bt]}\left\| \phi^\bullet_{t,s}(Z_s+\cdot) - v_\rho(\cdot) \right\|_{\ell^1(\Z^d)}
\end{equation}
converge to~$0$ in probability as $t \to \infty$.
\end{proposition}

The proofs of Propositions~\ref{prop:improvmassconc}--\ref{prop:shapes} are based on an approach from \cite{GKM07}
and will be given in Section~\ref{s:localprofiles} below.
Together with Theorem~\ref{thm:aging_locus}, they imply:
\begin{proof}[Proof of Theorem~\ref{thm:massconc_withuniq}]
Note that~\eqref{e:shapepot_withuniq} follows directly from~\eqref{e:shapepotdeps}.
For~\eqref{e:massconc_withuniq}, use~\eqref{e:improvmassconc}, \eqref{e:efcomparison_uniq},
the triangle inequality for the $\ell^1$-norm and~\eqref{e:nojumpandgap}.
\end{proof}

\RS We finish the section with: \eRS

\begin{proof}[Proof of Theorem~\ref{thm:aging_solution}]
We adapt the proof of Theorem 1.1 of \cite{MOS11}.
By Theorem~\ref{thm:aging_locus}, 
it is enough to show that, for any $\varepsilon \in (0,1)$ and $b > 1$,
\begin{equation}\label{e:pragingsol1}
\sup_{s \in [t,bt]} \sum_{z \in \Z^d} \left| \frac{u(z,s)}{U(s)} - \frac{u(z,t)}{U(t)} \right| < \varepsilon \;\;\; \text{ if and only if } \;\;\;  Z_s = Z_t \,\forall\, s \in [t,bt]
\end{equation}
holds with probability tending to $1$ as $t \to \infty$.

Assume first that $Z_s \neq Z_t$ for some $s \in (t,bt]$. By Propositions~\ref{prop:goodevent_forproof} and~\ref{prop:stabilitygap},
we may assume that $Z_{bt} \neq Z_t$,
\RV
and therefore by Proposition~\ref{prop:seprelcap} also that
e.g.\ $|Z_{bt}-Z_t| > \sqrt{t}$.
\eRV
Fixing $R$ so that~\eqref{e:massconcwithgap} holds with $\delta < \tfrac12 (1-\varepsilon)$, we obtain
\begin{multline}
\label{e:pragingsol2}
\sum_{z \in \Z^d} \left| \frac{u(z,bt)}{U(bt)} - \frac{u(z,t)}{U(t)}\right|
\ge \sum_{|z -Z_{bt}| \le R} \left|\frac{u(z,bt)}{U(bt)} \right| - \sum_{|z -Z_t| > R} \left|\frac{u(z,t)}{U(t)} \right|
\ge 1-2\delta > \varepsilon
\end{multline}
with probability tending to~$1$ as $t \to \infty$, proving the ``only if'' part of~\eqref{e:pragingsol1}.

Assume now that $Z_s = Z_t \;\forall\; s \in [t,bt]$.
Then $\phi^\bullet_{t,s} = \phi^\bullet_t$ for all $s \in [t, bt]$,
and the ``if'' part of~\eqref{e:pragingsol1} follows by~\eqref{e:improvmassconc} with $a=1<b$
together with Propositions~\ref{prop:goodevent_forproof}--\ref{prop:stabilitygap}.
\end{proof}


\section{Preparations}
\label{s:preparation}\nopagebreak\noindent
In this section, we collect auxiliary results that will be used in the remainder of the paper.
We start with a few basic properties of the potential field and of the principal Dirichlet eigenvalue 
of the Anderson Hamiltonian in subdomains of $\Z^d$,
leading to the proof of Proposition~\ref{prop:welldefined}.
The two subsequent subsections concern additional properties of the potential field,
and the last one contains spectral bounds for the Feynman-Kac formula.

\subsection{Potentials and eigenvalues}
First we consider the maximum of the potential in a box.
Let $\widehat{a}_L$ be the minimal number satisfying
\begin{equation}
\label{e:defhata}
\Prob \left( \xi(0) > \widehat{a}_L \right) =L^{-d},
\end{equation}
which exists since, by Assumption~\ref{A:dexp}, $\xi(0)$ has a continuous distribution.
Note that, in the notation of \cite{GM98}, $\widehat{a}_L=\psi(d \ln L)$.
Then we have:
\begin{lemma}[Maximum of the potential]
\label{l:maxpotential}
\begin{equation}\label{e:maxpotential}
\lim_{L \to \infty} \max_{x \in B_L} \xi(x) - \widehat{a}_L =0 \;\; \textnormal{ a.s.}
\end{equation}
\end{lemma}
\begin{proof}
See Corollary 2.7 of \cite{GM98}.
\end{proof}

Let us mention here some properties of $\widehat{a}_L$.
By equation (2.1) of \cite{GM98},
\begin{equation}\label{e:prophata}
\widehat{a}_{k_L} = \widehat{a}_{L} + o(1) \;\; \text{ as } L \to \infty \;\;\; \text{ whenever } \;\; \ln k_L = \ln L (1+ o(1))
\end{equation}
and, by Remark 2.1 therein, it is straightforward to verify that $\widehat{a}_L = (\rho+o(1)) \ln_2 L$.

Next we recall the Rayleigh-Ritz formula for the principal eigenvalue of the Anderson Hamiltonian.
For $\Lambda \subset \Z^d$ and $V: \Z^d \to [-\infty, \infty)$,
let $\lambda^{\ssup 1}_\Lambda(V)$ denote the largest eigenvalue of the operator $\Delta + V$ in $\Lambda$ with Dirichlet boundary conditions.
Then
\begin{equation}\label{e:RRformula}
\begin{aligned}
\lambda^{\ssup 1}_\Lambda(V) & = \sup \left\{ \langle (\Delta + V) \phi, \phi \rangle_{\ell^2(\Z^d)} \colon\, \phi \in \R^{\Z^d}, \,\supp \phi \subset \Lambda, \, \|\phi\|_{\ell^2(\Z^d)}=1 \right\}.
\end{aligned}
\end{equation}
When $V = \xi$ we sometimes write~$\lambda^{\ssup 1}_\Lambda$ instead of $\lambda^{\ssup 1}_\Lambda(\xi)$.
Here are some straightforward consequences of the Rayleigh-Ritz formula:
\begin{enumerate}
\item for any $\Gamma \subsetneq \Lambda$, 
\begin{equation}\label{e:monot_princev}
\max_{z \in \Gamma} V(z) - 2d \le \lambda^{\ssup 1}_\Gamma(V) \le \lambda^{\ssup 1}_\Lambda(V) \le \max_{z \in \Lambda} V(z);
\end{equation}
\item the eigenfunction corresponding to $\lambda^{\ssup 1}_\Lambda(V)$ can be taken non-negative;
\item if $V$ is real-valued and $\Lambda$ is finite and connected (in the graph-theoretical sense according to the usual nearest-neighbor structure of $\Z^d$), 
then the middle inequality in ~\eqref{e:monot_princev} is strict and, moreover, the non-negative eigenfunction corresponding to~$\lambda^{\ssup 1}_\Lambda(V)$ is strictly positive;
\item for $\Lambda, \Lambda' \subset \Z^d$ such that $\dist(\Lambda, \Lambda') \ge 2$,
\begin{equation}\label{e:ev_separatedsets}
\lambda^{\ssup 1}_{\Lambda \cup \Lambda'}(V) = \max \{ \lambda^{\ssup 1}_\Lambda(V), \lambda^{\ssup 1}_{\Lambda'}(V) \}.
\end{equation}
\end{enumerate}

We can now give the proof of Proposition~\ref{prop:welldefined}.
\begin{proof}[Proof of Proposition~\ref{prop:welldefined}]
Note that, for any $R\in \N$ and $z\in\Z^d$,
\begin{equation}
\label{E:5.7}
\{z \in \scrC\} \supseteq \left\{\xi(z) \le \rho\kappa^{-1}\ln R,\,\, \xi(z) = \max_{x \in B_R(z)} \xi(x) \right\},
\end{equation}
and the probability of the event on the
right-hand side does not depend on $z$ and is  positive for some fixed large enough~$R$. 
As the events on the right of~\eqref{E:5.7} depend only on a finite number of coordinates, the second Borel-Cantelli lemma shows $|\scrC| = \infty$ almost surely.
Now, by~\eqref{e:monot_princev}, $\lambda^{\scrC}(z) \le \xi(z)$ for any $z \in \scrC$
while, by Lemma~\ref{l:maxpotential}, almost surely $\xi(z) \le 2 \rho \ln_2 |z|$ for all $|z|$ large enough.
This implies that, almost surely,
\begin{equation}\label{e:restrictionmaximizationPsi}
\limsup_{R \to \infty} \sup_{z \in \mathscr{C}, |z| = R} \Psi_t(z) \le \lim_{R \to \infty} 
\Bigl(2 \rho \ln_2 R - R \frac{\ln_3 R}{t}\Bigr) = -\infty
\end{equation}
for each $t > 0$, finishing the proof.
\end{proof}

\vspace{10pt}
Next we generalise \twoeqref{defcL}{e:variationalproblemchi} as follows.
For $\Lambda \subset \Z^d$ and $V: \Z^d \to [-\infty,\infty)$, let
\begin{equation}\label{e:defcurlyL}
\LL_\Lambda(V) := \sum_{x \in \Lambda} \texte^{\frac{V(x)}{\rho}},
\end{equation}
with the interpretation $\texte^{-\infty}:=0$.
Then set
\begin{equation}\label{e:defchiLambda}
\chi_{\Lambda} = \chi_{\Lambda}(\rho) := - \sup \left\{ \lambda^{\ssup 1}_{\Lambda}(V) \colon\, V \in [-\infty, 0]^{\Z^d}, \LL_\Lambda(V) \le 1 \right\}.
\end{equation}
When $\Lambda = \Z^d$ we  write just~$\chi$.
From the definition it follows that, if $\Gamma \subset \Lambda$, then $\chi_\Gamma \ge \chi_\Lambda$;
in particular, $0 \le \chi \le \chi_\Lambda \le 2d$ since $\chi_{\{x\}} = 2d$ for any $x \in \Z^d$.

\subsection{Islands}
\label{ss:properties_components}\noindent
Central to our analysis is a domain truncation method taken from \cite{BK16}, which we describe next. 
Recall the choice of $\kappa \in (0,1/d)$ in~\eqref{e:defvarrhoz} and
fix an increasing sequence $R_L \in \N$ such that
\begin{equation}\label{e:propertiesR_L}
R_L \le (\ln L) \vee 1 \;\; \text{ and } \;\; R_L \gg (\ln L)^{\beta} \text{ as } L \to \infty \;\; \text{ for some } \beta \in (\kappa, 1/d).
\end{equation}
This sequence will control the spatial size of the regions in $B_L$ 
where the field is large, and thus the (principal) local eigenvalue has a chance to be close to maximal.
We will often work with $R_L$ satisfying additionally
\begin{equation}\label{e:addpropR_L}
R_L \ll (\ln L)^\alpha \text{ as } L \to \infty \;\; \text{ for some } \alpha \in (\beta, 1/d), 
\end{equation}
but for the proof of Proposition~\ref{prop:boundforpathconc} in Section~\ref{ss:pathconclocus} 
we will need to consider $R_L$ growing as $\ln L$.
Given $A>0$ and $L \in \N$, let
\begin{equation}\label{defPi}
\Pi_{L,A} := \{z \in B_L \colon\, \xi(z) > \widehat{a}_L - 2A\}
\end{equation}
be the set of high exceedances of the field inside the box $B_L$,
and put
\begin{equation}
\label{def_D_L,A}
D_{L,A} := \bigcup_{z \in \Pi_{L,A}} B_{R_L}(z) \cap B_L.
\end{equation}
The parameter~$A$, providing the cutoff between the ``high'' and ``small'' values of the field, will be later fixed to a suitably large value that depends only on the dimension~$d$ and the parameter~$\rho$.

Let $\mathfrak{C}_{L,A}$ denote the set of all connected components of~$D_{L,A}$, to be called~\emph{islands}.
For $\CC \in \mathfrak{C}_{L,A}$, let
\begin{equation}\label{defzC}
z_\CC := \textnormal{argmax}\{\xi(z) \colon\, z \in \CC\}
\end{equation}
be the point of highest potential within $\CC$. 
Since $\xi(0)$ has a continuous law, $z_\CC$ is a.s.\ well defined
 for all $\CC \in \mathfrak{C}_{L,A}$.

Next we gather some useful properties of $\mathfrak{C}_{L,A}$.
The first result concerns 
a uniform bound on the size of the islands.
Hereafter we will say that an $L$-dependent event occurs ``almost surely eventually as $L \to \infty$''
if there exists a.s.\ a (random) $L_0 \in \N$ such that the event happens for all $L \ge L_0$.
Similar language will be used for events depending on other parameters (e.g.\ $t$).
\begin{lemma}[Maximum size of the islands]
\label{l:size_comps}
For any $A > 0$, there exists $n_A \in \N$ such that, 
for any $R_L$ satisfying~\eqref{e:propertiesR_L},
a.s.\ eventually as $L\to \infty$,
all $\CC \in \mathfrak{C}_{L,A}$ satisfy 
$|\CC \cap \Pi_{L,A}| \le n_A$ and $\diam(\CC) \le n_A R_L$.
\end{lemma}
\begin{proof}
See the proof of Lemma 6.6 in \cite{BK16}.
\end{proof}

For $\delta >0$, $A>0$ and $L \in \N$, let
\begin{equation}\label{e:reducbcomp1}
\mathfrak{C}^\delta_{L,A}:= \{\CC \in \mathfrak{C}_{L,A} \colon\, \lambda^{\ssup 1}_\CC > \widehat{a}_L - \chi - \delta\}
\end{equation}
denote the set of islands with large principal eigenvalue.
We call these \emph{relevant islands}, as their eigenvalue is close to the principal eigenvalue of $B_L$ (cf.\ Lemma~6.8 of \cite{BK16}).
\RV
In the proofs of our main theorems, $\delta$ will be fixed at some small enough value
so as to satisfy the requirements of some intermediate results given below.
\eRV

The next lemma is crucial for the proof of Proposition~\ref{prop:PPPconv},
which implies Proposition~\ref{prop:goodevent_forproof}
\RS and is one of the main ingredients in the proof of Theorem~\ref{thm:locus}. \eRS
It allows us to compare the principal eigenvalues of relevant islands to those of disjoint boxes.
\begin{lemma}[Coarse-graining for local principal eigenvalues]
\label{l:properties_opt_comps}
Assume $R_L$ satisfies~\eqref{e:propertiesR_L} and~\eqref{e:addpropR_L}.
Let $N_L \in \N$ satisfy $L^{\beta} \ll N_L \ll L^{\alpha}$ as $L \to \infty$ for some $0< \beta < \alpha < 1$.
For all $A > 0$ sufficiently large and $\delta > 0$ small enough, the following
occurs with probability tending to one as $L \to \infty$:
\begin{enumerate}
\item[(i)]  Each  $\CC \in \mathfrak{C}^\delta_{L,A}$ satisfies $\lambda^{\ssup 1}_{\CC}-\lambda^{\ssup 2}_{\CC} \ge \tfrac12 \rho \ln 2$.

\item[(ii)] For each  $\CC \in \mathfrak{C}^\delta_{L,A}$, there exists $z \in (2 N_L +1) \Z^d $ such that $\CC \subset B_{N_L}(z) \subset B_L$.

\item[(iii)]  Every two distinct $\CC,\CC' \in \mathfrak{C}^\delta_{L,A}$ satisfy $\dist(\CC,\CC') > 4d N_L$.

\item[(iv)] 
Let $\eta_A := \{1 + A/(4d)\}^{-1}$.
For any $z \in (2 N_L + 1)\Z^d$ such that $B_{N_L}(z) \subset B_L$ and $\lambda^{\ssup 1}_{B_{N_L}(z)} > \widehat{a}_L - \chi - \delta + (\eta_A)^{R_L}$, 
there exists a $\CC \in \mathfrak{C}^\delta_{L,A}$ satisfying $\CC \subset B_{N_L}(z)$ and
\begin{equation}\label{e:comp_eigenvalues}
\lambda^{\ssup 1}_\CC > \lambda^{\ssup 1}_{B_{N_L}(z)}  - (\eta_A)^{R_L}.
\end{equation}
\end{enumerate}
\end{lemma}

\begin{proof}
Let $A,\delta$ be as in the statement of Lemma 6.7 of \cite{BK16}; we may assume that $A>\chi+\delta$.
Items (i)--(iii) follow from items (1)--(3) in this lemma (the scales there do not match ours exactly, but the proof is the same).
For (iv), assume that $L$ is so large that $2d (\eta_A)^{2R_L-1} < (\eta_A)^{R_L}$,
and note that $\lambda^{\ssup 1}_{B_{N_L}}(z) - A > \widehat{a}_L - 2A$.
By Theorem 2.1 of \cite{BK16} applied to $D = B_{N_L}(z)$ and~\eqref{e:ev_separatedsets},
there exists $\CC \in \mathfrak{C}_{L,A}$, $\CC \cap B_{N_L}(z) \neq \emptyset$
such that~\eqref{e:comp_eigenvalues} holds. In particular, $\CC \in \mathfrak{C}^{\delta}_{L,A}$ so, by item (ii), $\CC \subset B_{N_L}(z)$.
\end{proof}

Our next goal is to control the behavior of the potential inside relevant islands.
This will be important for the proofs of Propositions~\ref{prop:pathsavoidingZ_forproof} and~\ref{prop:localizationinnerdomain} 
as well as Lemma~\ref{l:comparisoncapitalsislands} below. 
First we will need two lemmas concerning lower and upper bounds for $\LL$. 
\begin{lemma}
\label{l:lbddcurlyL}
For any $\Lambda \subset \Z^d$ and any $a \in \R$,
if $\lambda^{\ssup 1}_\Lambda \ge a$ then $\LL_\Lambda(\xi - a - \chi_\Lambda) \ge 1$.
\end{lemma}
\begin{proof}
This is a consequence of \twoeqref{e:defcurlyL}{e:defchiLambda} and
the fact that $\lambda^{\ssup 1}_\Lambda(V+a) = \lambda^{\ssup 1}_\Lambda(V)+a$.
\end{proof}

\begin{lemma}
\label{l:uppbddcurlyL}
Let $R_L$ satisfy \twoeqref{e:propertiesR_L}{e:addpropR_L}.
For any $A>0$,
\begin{equation}\label{e:uppbddcurlyL}
\limsup_{L \to \infty} \sup_{\CC \in \mathfrak{C}_{L,A}} \, \mathcal{L}_{\CC}(\xi-\widehat{a}_L)  \le 1 \;\; \textnormal{ a.s.}
\end{equation}
\end{lemma}
\begin{proof}
This is a consequence of Lemma~\ref{l:size_comps} and a straightforward extension of Corollary 2.12 in \cite{GM98} 
with $R$ substituted by $n_A R_L$.
\end{proof}

We will now combine the previous two lemmas with results from \cite{BK16}, \cite{GH99} and \cite{GKM07} to
obtain upper and lower bounds around $\widehat{a}_L$ for the potential in relevant islands.
\begin{lemma}[Upper bound for the potential inside relevant islands]
\label{l:potential_rel_islands_upbd}
Assume \twoeqref{e:propertiesR_L}{e:addpropR_L}.
For all $\delta \in (0,1)$ small enough, 
there exist $A_1 > 4d$ and $\nu_1 \in \N$ such that, 
for all $A>0$, a.s.\ eventually as $L \to \infty$,
\begin{align}\label{e:potential_rel_islands_uppbd}
\sup_{\CC \in \mathfrak{C}^{\delta}_{L,A}} \, \sup_{z \in \CC \setminus B_{\nu_1}(z_\CC)} \xi(z) & \le \widehat{a}_{L} - 2A_1. 
\end{align}
\end{lemma}
\begin{proof}
We follow the proof of Lemma 4.8 of \cite{BK16}.
Fix $\delta \in (0,1)$ small enough such that 
\begin{equation}\label{e:upbd_potrelisl1}
A_1 := - \tfrac12  \rho \ln \left( \texte^{\frac{2 \delta}{\rho}} - \texte^{-\frac{2 \delta}{\rho}} \right) > 4d > \chi+\delta,
\end{equation}
and let $r \in \N$ 
be such that $2d \eta_{A_1}^{2r-1} < \delta$ where $\eta_A := (1+A/4d)^{-1}$.
For $\CC \in \mathfrak{C}^\delta_{L,A}$, 
let 
\begin{equation}\label{e:upbd_potrelisl2}
S := \{ x \in \CC \colon\, \xi(x) > \widehat{a}_L - 2A_1\}.
\end{equation}
We claim that
\begin{equation}\label{e:upbd_potrelisl3}
\diam S \le 2(r+1) |S|.
\end{equation}
Indeed, suppose by contradiction that~\eqref{e:upbd_potrelisl3} does not hold.
Then $S=S_1 \cup S_2$ with $\dist(S_1,S_2) \ge 2(r+1)$.
Let $S^r_i := \{x \in \CC \colon\, \dist(x, S_i) \le r\}$, $i=1,2$.
Then, by~\eqref{e:ev_separatedsets},
\begin{equation}\label{e:upbd_potrelisl4}
\lambda^{\ssup 1}_{S^r_1} \vee \lambda^{\ssup 1}_{S^r_2} = \lambda^{\ssup 1}_{S^r_1 \cup S^r_2}
> \lambda^{\ssup 1}_\CC - 2d \eta_{A_1}^{2r-1} > \widehat{a}_L - \chi - 2 \delta
\end{equation}
where for the first inequality we use Theorem 2.1 of \cite{BK16} applied to $D:=\CC$ 
(note that $\lambda^{\ssup 1}_\CC - A_1 > \widehat{a}_L - 2A_1$ since $\CC$ is assumed to be in $\mathfrak C^\delta_{L,A}$, i.e., 
such that $\lambda^{\ssup 1}_\CC>\widehat{a}_L-\chi-\delta$, and by \eqref{e:upbd_potrelisl1}),
and the last inequality follows by our choice of~$r$.
Supposing without loss of generality that $\lambda^{\ssup 1}_{S^r_1} \ge \lambda^{\ssup 1}_{S^r_2}$,
by Lemma~\ref{l:lbddcurlyL} and~\eqref{e:upbd_potrelisl4} we have

\begin{equation}\label{e:upbd_potrelisl5}
\LL_{S^r_1}\left(\xi - \widehat{a}_L\right) \ge \texte^{(\chi_{S^r_1} - \chi - 2 \delta )/\rho} \ge \texte^{-\frac{2\delta}{\rho}}.
\end{equation}
By Lemma~\ref{l:uppbddcurlyL}, we may suppose that $\LL_\CC(\xi-\widehat{a}_L) \le \texte^{2\delta/\rho}$.
Then, for any $x \in S_2$,

\begin{equation}\label{e:upbd_potrelisl6}
\LL_{S^r_1}\left(\xi - \widehat{a}_L \right) \le \LL_{\CC} \left(\xi - \widehat{a}_L \right) - \texte^{\frac{\xi(x)-\widehat{a}_L}{\rho}} \le \texte^{\frac{2\delta}{\rho}} - \texte^{\frac{\xi(x)-\widehat{a}_L}{\rho}}.
\end{equation}
Combining \twoeqref{e:upbd_potrelisl5}{e:upbd_potrelisl6} we obtain
\begin{equation}\label{e:upbd_potrelisl7}
\xi(x) - \widehat{a}_L \le \rho \ln \left( \texte^{\frac{2 \delta}{\rho}} - \texte^{-\frac{2 \delta}{\rho}} \right) = - 2 A_1,
\end{equation}
contradicting $x \in S$. Therefore,~\eqref{e:upbd_potrelisl3} holds.
To conclude, note that
\begin{equation}\label{e:upbd_potrelisl8}
\texte^{\frac{2\delta}{\rho}} \ge \LL_\CC(\xi-\widehat{a}_L) \ge \texte^{-\frac{2 A_1}{\rho}}|S|.
\end{equation}
Since $z_\CC \in S$ by~\eqref{e:monot_princev} and~\eqref{e:upbd_potrelisl1}, 
the inequalities \eqref{e:upbd_potrelisl3} and~\eqref{e:upbd_potrelisl8} 
now imply~\eqref{e:potential_rel_islands_uppbd}
with $\nu_1 := \lceil 2 (r+1) \texte^{\frac{2 (A_1+\delta)}{\rho}} \rceil$.
\end{proof}

\begin{lemma}[Lower bound for the potential in relevant islands]
\label{l:potential_rel_islands_lwbd}
Suppose that $R_L$ is such that \twoeqref{e:propertiesR_L}{e:addpropR_L}
hold.
For any $\nu \in \N$, there exist $A^*, \delta > 0$ such that, 
for all $A>0$,
the following is true
a.s.\ eventually as $L \to \infty$:
\begin{align}\label{e:potential_rel_islands_lwbd}
\inf_{\CC \in \mathfrak{C}^\delta_{L,A}} \, \inf_{z \in B_{\nu}(z_\CC)} \xi(z) \ge \widehat{a}_{L} - 2A^*.
\end{align}
\end{lemma}
\begin{proof}
Recall the definition of $\MM^*_\rho $ in~\eqref{e:defMrho}.
We note that Lemma 3.2(i) of \cite{GKM07} holds for $\MM^*_\rho$ in place of $\MM_\rho$,
as can be inferred from the proof.
In particular, $\MM^*_\rho \neq \emptyset$ and, 
by Lemma 3.1 therein, all $V \in \MM^*_\rho$ satisfy $\LL(V) = 1$.
On the other hand, by (3.21) in \cite{GKM07} together with Theorem 2 and Proposition 3 of \cite{GH99} (see also (5.44) therein),
\begin{equation}\label{e:lwbd_potrelisl1}
A^* := - \inf_{V \in \MM^*_\rho} \inf_{x \in B_{\nu}} V(x) < \infty.
\end{equation}
Fix, by (3.6) in \cite{GKM07}, $\delta >0$ small enough such that
\begin{equation}\label{e:lwbd_potrelisl2}
\left\{\begin{aligned}
& V \in [-\infty,0]^{\Z^d}, \, 0 \in \textnormal{argmax}(V), \, \mathcal{L}(V) \le 1 \\
& \;\;\; \textnormal{ and } \inf_{\overline{V} \in \MM^*_\rho} \sup_{x \in B_\nu} |V(x)-\overline{V}(x)| > A^*
\end{aligned}
\right\}\quad \Rightarrow \quad \lambda^{\ssup 1}(V) < -\chi - 2\delta.
\end{equation}
Fix $\CC \in \mathfrak{C}^\delta_{L,A}$ and define
\begin{equation}\label{e:lwbd_potrelisl3}
V^*(x) := \left\{ 
\begin{array}{ll}
\xi(x+z_\CC)-\widehat{a}_L - \delta & \text{ if } x + z_\CC \in \CC,\\
- \infty & \text{ otherwise.}
\end{array}
\right.
\end{equation}
By Lemma~\ref{l:maxpotential}, $V^* \in [-\infty,0)^{\Z^d}$ a.s.\ eventually as $L \to \infty$, 
and $0 \in \textnormal{argmax}(V^*)$ by the definition of $z_\CC$.
Furthermore, $\mathcal{L}(V^*) = \mathcal{L}_\CC(\xi - \widehat{a}_L - \delta)$
which is a.s.\ smaller than $1$ for large $L$ by Lemma~\ref{l:uppbddcurlyL}.
Now, since $\CC \in \mathfrak{C}^\delta_{L,A}$, 
we have $\lambda^{\ssup 1}(V^*) = \lambda^{\ssup 1}_\CC -\widehat{a}_L - \delta > -\chi - 2\delta$,
and thus
the conclusion follows from \twoeqref{e:lwbd_potrelisl1}{e:lwbd_potrelisl2}.
\end{proof}

We end this subsection with a comparison between the islands and capitals with large local eigenvalues,
which will be crucial in the proof of Proposition~\ref{prop:PPPconv} below.
\begin{lemma}
\label{l:comparisoncapitalsislands}
Assume \twoeqref{e:propertiesR_L}{e:addpropR_L}. 
There exists a constant $c_1 > 0$ such that,
for all $A > 0$ large enough and $\delta > 0$ small enough, 
the following occurs with probability tending to one as $L \to \infty$:
\begin{enumerate}
\item[(i)] If $\CC \in \mathfrak{C}^{\delta}_{L,A}$, then $z_\CC \in \scrC$, $ (\ln L)^{\kappa/2} < \varrho_{z_\CC} < R_L$ and
\begin{equation}\label{e:compevcapisl}
0 \le \lambda^{\ssup 1}_\CC - \lambda^{\scrC}(z_\CC) \le \texte^{-c_1 (\ln L)^{\kappa/2}}.
\end{equation}
\item[(ii)] For all $z \in \scrC$ such that $B_{\varrho_z}(z) \subset B_L$ 
and $\lambda^{\scrC}(z) > \widehat{a}_L - \chi - \delta$, 
there exists $\CC \in \mathfrak{C}^\delta_{L,A}$ such that $z = z_\CC$ and~\eqref{e:compevcapisl} holds.
\end{enumerate}
\end{lemma}
\begin{proof}
Let $A, \delta>0$ satisfy the hypotheses of Lemmas~\ref{l:properties_opt_comps} and~\ref{l:potential_rel_islands_upbd},
and let $A_1>0$, $\nu_1 \in \N$ as in Lemma~\ref{l:potential_rel_islands_upbd}. 
We may assume that $2A>A_1$.
For \emph{(i)}, note that, 
if $\CC \in \mathfrak{C}^\delta_{L,A}$, 
then $(\ln L)^{\kappa/2} + \nu_1 < \varrho_{z_\CC} \le \max_{z \in B_L} \varrho_z < R_L $ for all $L$ large enough 
by~\eqref{e:defvarrhoz},~\eqref{e:maxpotential},~\eqref{e:monot_princev} and~\eqref{e:propertiesR_L},
and thus $z_\CC \in \scrC$.
By Lemma~\ref{l:potential_rel_islands_upbd}, 
the set $\{x \in \CC \colon\, \dist(x, \Pi_{L,A_1}) \le (\ln L)^{\kappa/2}\}$ is contained in $B_{\varrho_{z_\CC}}(z_\CC)$
and thus~\eqref{e:compevcapisl} follows by Theorem~2.1 of \cite{BK16} with $c_1 := \ln (1+A_1/(4d))$.
For \emph{(ii)}, note that, again by~\eqref{e:monot_princev}, $\xi(z) > \widehat{a}_L -A_1$ and thus $z \in \Pi_{L,A}$.
Letting $\CC \in \mathfrak{C}_{L,A}$ such that $z \in \CC$, note that $B_{\varrho_z}(z) \subset \CC$ since $\varrho_z < R_L$,
and thus $\CC \in \mathfrak{C}^\delta_{L,A}$. Since $\varrho_z > \nu_1$, $z=z_\CC$ by Lemma~\ref{l:potential_rel_islands_upbd},
and~\eqref{e:compevcapisl} follows by item \emph{(i)}.
\end{proof}


\subsection{Connectivity properties of the potential field}
\label{ss:connect_prop}\noindent
In this section, we provide bounds on the number of points in which the potential achieves high values inside connected sets of the lattice.
These will be important in the proof of Proposition~\ref{prop:massclass}.
\RV We will use the following Chernoff bound:\eRV
\begin{lemma}
\label{l:conc_binomial}
Let $\textnormal{Bin($p,n$)}$  denote  
a Binomial random variable with parameters $p$ and $n$.
Then
\begin{equation}\label{e:conc_binomial}
P \bigl(\textnormal{Bin($p,n$)} > u \bigr) \le \exp \left\{-u \left( \ln \frac{u}{np} - 1 \right) \right\} \quad \text{ for all } u > 0.
\end{equation}
\end{lemma}
\begin{proof}
Write $E \left[\exp\{\alpha \text{Bin($p,n$)} \}  \right] = \left\{1 + p (\texte^\alpha-1) \right\}^{n} \le \texte^{ n p \texte^\alpha}$, apply Markov's inequality and optimize over $\alpha>0$.
\end{proof}

Our first lemma reads as follows.
\begin{lemma}[Number of intermediate peaks of the potential]
\label{l:bound_mediumpoints}
For each $\beta \in (0,1)$, there exists $\varepsilon \in (0,\beta/2)$ such that, a.s.\ eventually as $L \to \infty$,
for all finite connected subsets $\Lambda \subset \Z^d$ with $\Lambda \cap B_L \neq \emptyset$ and $|\Lambda|\ge (\ln L)^{\beta}$,
\begin{equation}\label{e:bound_mediumpoints}
N_{\Lambda} :=|\{z \in \Lambda \colon\, \xi(z) > (1-\varepsilon) \widehat{a}_L \}| \le \frac{|\Lambda|}{(\ln L)^\varepsilon}.
\end{equation} 
\end{lemma}
\begin{proof}
Let $\varepsilon \in (0, \beta/2 )$ 
be small enough so that, for all $L$ large enough,
\begin{equation}\label{e:bound_mp1}
p_L := \textnormal{Prob} \left( \xi(0) > (1-\varepsilon)\widehat{a}_L \right) \le \exp\left\{-(\ln L)^{1-\frac{\beta}{2}}\right\}.
\end{equation}
This is possible by e.g.\ Lemma 6.1 in \cite{BK16}.
Now fix a point $x \in B_L$ and $n \in \N$.
The number of connected subsets $\Lambda \subset \Z^d$ with $|\Lambda|=n$ and $x \in \Lambda$
is at most $\texte^{c_0 n}$ for some $c_0 > 0$ independent of $x$ (see e.g.\ \cite{G99}, Section 4.2).
For such a $\Lambda$, the random variable~$N_{\Lambda}$ has a Bin($p_L$, $n$)-distribution.
Using~\eqref{e:conc_binomial} and a union bound, we obtain
\begin{multline}\label{e:bound_mp4}
\qquad
\textnormal{Prob} \Bigl( \exists\, \text{ connected } \Lambda \ni x, \; |\Lambda|=n \text{ and } N_{\Lambda} > n/ (\ln L)^\varepsilon \Bigr) 
\\
 \le \exp \left\{ -n \left((\ln L)^{1 - \tfrac{\beta}{2} - \varepsilon} -c_0 - \frac{1+ \varepsilon\ln_2 L}{(\ln L)^{\varepsilon}}\right) \right\}.
 \qquad
\end{multline}
When $L$ is large enough, the expression in parentheses above is at least $\frac12 (\ln L)^{1- \tfrac{\beta}{2} - \varepsilon }$. 
Summing over $n \ge (\ln L)^\beta$ and $x \in B_L$, we get
\begin{equation}\label{e:bound_mp5}
\textnormal{Prob} \left( \begin{aligned}
&\exists\, \text{ connected } \Lambda \text{ such that } \Lambda \cap B_L \neq \emptyset,
\\ 
&|\Lambda| \ge (\ln L)^\beta \text{ and }~\eqref{e:bound_mediumpoints} \text{ does not hold}
\end{aligned}
\right)
\le c_1 L^d \exp \left\{-c_2 (\ln L)^{1+\tfrac{\beta}{2}-\varepsilon} \right\}
\end{equation}
for some positive constants $c_1, c_2$. By our choice of $\varepsilon$, 
\eqref{e:bound_mp5} is summable on~$L$, so the conclusion follows from the Borel-Cantelli lemma.
\end{proof}

A similar computation bounds the number of high exceedances of the potential.
\begin{lemma}[Number of high exceedances of the potential]
\label{l:boundhighexceedances}
For each $A>0$, there is a constant $C\ge 1$ such that, for all $\delta \in (0,1)$, the following holds
a.s.\ eventually as $L \to \infty$: For
all finite connected subsets $\Lambda \subset \Z^d$ with $\Lambda \cap B_L \neq \emptyset$ and $|\Lambda|\ge C (\ln L)^\delta$
it holds that
\begin{equation}\label{e:boundhighexceedances}
|\Lambda \cap \Pi_{L, A}| \le \frac{|\Lambda|}{(\ln L)^\delta}.
\end{equation} 
\end{lemma}
\begin{proof}
Proceed as for Lemma~\ref{l:bound_mediumpoints} first noting that, by Lemma~6.1 in~\cite{BK16},
\begin{equation}\label{e:bound_he1}
p_L := \textnormal{Prob} \left( 0 \in \Pi_{L, A} \right) \le L^{-\epsilon}
\end{equation}
for some $\epsilon\in (0,1)$ and all large enough $L$, and then taking $C > 2(d+1)/\epsilon$.
\end{proof}


\subsection{Spectral bounds}
\label{ss:specbounds}\noindent
Here we state some spectral bounds for the Feynman-Kac formula.
The results in this section are deterministic, i.e.,
they hold for any fixed choice of potential $\xi \in \R^{\Z^d}$.

Fix a finite connected subset $\Lambda \subset \Z^d$, 
and let $H_\Lambda$ denote 
\RV
the Anderson Hamiltonian in $\Lambda$
with zero Dirichlet boundary conditions,
as described after \eqref{e:introl2descr}.
\eRV
For $z \in \Lambda$, let $u^z_\Lambda$ be the positive solution of
\begin{equation}\label{e:PAM}
\begin{aligned}
\partial_t u(x,t) & = H_{\Lambda} u(x,t), \quad x \in \Lambda, \, t>0,\\
u(x,0) & = \1_z(x), \qquad \quad x \in \Lambda, \\
\end{aligned}
\end{equation}
and set $U^z_\Lambda(t) := \sum_{x \in \Lambda} u^z_\Lambda(x,t)$.
The solution admits the Feynman-Kac representation
\begin{equation}\label{e:FKformula}
u^z_\Lambda(x,t) = \E_z \left[ \exp \left\{\int_0^t \xi(X_s) \textd s \right\} \1 \{\tau_{\Lambda^{\cc}}>t, X_t = x\} \right]
\end{equation}
where $\tau_{\Lambda^\cc}$ is as in~\eqref{e:deftauLambda}. It also admits the spectral representation
\begin{equation}\label{e:specrepr}
u^z_\Lambda(x,t) = \sum_{k=1}^{|\Lambda|} \texte^{t \lambda^{\ssup k}_\Lambda} \phi_\Lambda^{\ssup k}(z) \phi^{\ssup k}_\Lambda(x),
\end{equation}
where $\lambda^{\ssup 1}_\Lambda \RV \ge \eRV \lambda^{\ssup 2}_\Lambda \ge \cdots \ge \lambda^{\ssup{|\Lambda|}}_\Lambda$ and $\phi^{\ssup 1}_\Lambda, \phi^{\ssup 2}_\Lambda, \ldots, \phi^{\ssup{|\Lambda|}}_\Lambda$
are respectively the eigenvalues and corresponding orthonormal eigenfunctions of $H_\Lambda$.
One may exploit these representations to obtain bounds for one in terms of the other,
as shown by the following lemma.

\begin{lemma}[Bounds on the solution]
\label{l:bounds_mass}

For any $z \in \Lambda$ and any $t > 0$,
\begin{multline}\label{e:bounds_mass}
\qquad
\texte^{t \lambda^{\ssup 1}_\Lambda} \phi^{\ssup 1}_\Lambda(z)^2 
\le \E_z \left[ \texte^{\int_0^t \xi(X_s) \textd s} \1_{\{\tau_{\Lambda^\cc} > t, X_t = z\}} \right] 
\\
\le \E_z \left[ \texte^{\int_0^t \xi(X_s) \textd s} \1_{\{\tau_{\Lambda^\cc} > t\}} \right]
\le \texte^{t \lambda^{\ssup 1}_\Lambda} |\Lambda|^{3/2}.
\qquad
\end{multline}
\end{lemma}
\begin{proof}
The first and last inequalities follow directly from \twoeqref{e:FKformula}{e:specrepr};
 the middle inequality is elementary.
\end{proof}

The second lemma bounds the Feynman-Kac formula integrated up to an exit time.
\begin{lemma}[Mass up to an exit time]
\label{l:mass_out}
For any $z \in \Lambda$ and $ \gamma  > \lambda^{\ssup 1}_\Lambda$,
\begin{equation}\label{e:mass_out}
\E_{z} \left[ \exp \left\{ \int_0^{\tau_{\Lambda^\cc}} (\xi(X_s) -  \gamma ) \textd s \right\} \right] 
\le 1 + \frac{2d|\Lambda|}{ \gamma  - \lambda^{\ssup 1}_\Lambda}.
\end{equation}
\end{lemma}
\begin{proof}
See Lemma 4.2 in \cite{GKM07}.
\end{proof}

The next lemma is a well-known representation for the principal eigenfunction.
\begin{lemma}
\label{l:reprEF}
For any $x,y \in \Lambda$,
\begin{equation}\label{e:FKrepr_ef}
\frac{\phi^{\ssup 1}_\Lambda(x)}{\phi^{\ssup 1}_\Lambda(y)} = \E_x \left[ \exp \left\{ \int_0^{\tau_y} \left(\xi(X_u) - \lambda^{\ssup 1}_\Lambda \right) \textd u \right\} \1 \{\tau_y < \tau_{\Lambda^\cc} \} \right].
\end{equation}
\end{lemma}
\begin{proof}
See e.g.\  Proposition~3.3 in \cite{MP16}. 
\end{proof}

Our last lemma bounds the Feynman-Kac formula when the random walk is restricted to hit a subset,
 and is the principal ingredient in the proof of Proposition~\ref{prop:boundbyprincipalef_forproof}.
\begin{lemma}[Bound by principal eigenfunction]
\label{l:bound_top_ef}
For all $t>0$, $z,x \in \Lambda$ and $\Gamma \subset \Lambda$,
\begin{align}\label{e:bound_top_ef}
\E_z \left[\texte^{\int_0^t \xi(X_s) \textd s} \1\{X_t = x, \tau_{\Lambda^\cc}>t \ge \tau_{\Gamma}\} \right]
\le U_\Lambda^z(t) \,\phi^{\ssup 1}_\Lambda(x)\,
\sup_{y \in \Gamma} \left\{|\phi^{\ssup 1}_\Lambda(y)|^{-3} \right\}.
\end{align}
\end{lemma}
\begin{proof}
We adapt the proof of Theorem 4.1 of \cite{GKM07}. Fix~$z\in\Z^d$ and, for $x\in\Z^d$ and $t>0$, denote
\begin{equation}\label{e:boundtopef1}
w(x,t) := \E_x \left[\texte^{\int_0^t \xi(X_s) \textd s} \1\{X_t = z, \tau_{\Lambda^\cc} >t \ge \tau_{\Gamma}\} \right].
\end{equation}
Note that, by invariance under time reversal, 
\eqref{e:boundtopef1} is equal to the left-hand side of~\eqref{e:bound_top_ef}.
It will suffice to show that, 
for any $0 < s \le t$ and $y \in \Gamma$,
\begin{equation}\label{e:boundtopef2}
\E_y \left[ \texte^{\int_0^{t-s} \xi(X_u) \textd u}  \1_{\{X_{t-s}=z, \tau_{\Lambda^\cc} > t-s\}}\right]
\le \texte^{-s \lambda^{\ssup 1}_{\Lambda} } |\phi^{\ssup 1}_\Lambda(y)|^{-2} w(y,t).
\end{equation}
 Indeed,  
 by the strong Markov property, $w(x,t)$ equals
\begin{multline}
\label{e:boundtopef3}
\quad \sum_{y \in \Gamma} \E_x\left[ \texte^{\int_0^{\tau_y} \xi(X_u) \textd u} \1_{\{\tau_{\Lambda^\cc}>\tau_y = \tau_{\Gamma} \le t \}} 
\left(\E_y \left[ \texte^{\int_0^{t-s} \xi(X_u) \textd u} \1_{\{X_{t-s}=z, \tau_{\Lambda^\cc} > t-s\}} \right]\right)_{s = \tau_y} \right] 
\\
\le \sum_{y \in \Gamma} |\phi^{\ssup 1}_\Lambda(y)|^{-2} w(y,t) \E_x\left[ \texte^{\int_0^{\tau_y} \left( \xi(X_u) - \lambda^{\ssup 1}_\Lambda \right) \textd u} \1_{\{\tau_{\Lambda^\cc} > \tau_y \}} \right] 
\\  
=  \, \phi^{\ssup 1}_\Lambda(x) \sum_{y \in \Gamma} |\phi^{\ssup 1}_\Lambda(y)|^{-3} w(y,t)
\le \phi^{\ssup 1}_\Lambda(x) \sup_{y \in \Gamma} \left\{ |\phi^{\ssup 1}_\Lambda(y)|^{-3} \right\}
U_\Lambda^z(t),
\quad
\end{multline}
where for the second line we used~\eqref{e:boundtopef2} and, for the last one, we invoked~\eqref{e:FKrepr_ef} and one more time applied the invariance under time reversal.

To prove~\eqref{e:boundtopef2}, restrict to $X_s = y$ inside the expectation defining $w(y,t)$ to obtain
\begin{equation}\label{e:boundtopef4}
w(y,t) \ge \E_y \left[\texte^{\int_0^s \xi(X_u) \textd u} \1_{\{X_s = y, \tau_{\Lambda^\cc} > s \}} \right] \E_y \left[\texte^{\int_0^{t-s}\xi(X_u) \textd u} \1_{\{X_{t-s}=z, \tau_{\Lambda^\cc}>t-s\}}\right].
\end{equation}
By Lemma~\ref{l:bounds_mass},
\begin{equation}\label{e:boundtopef5}
\E_y \left[\texte^{\int_0^s \xi(X_u) \textd u} \1_{\{X_s = y, \tau_{\Lambda^\cc} > s \}} \right] \ge \texte^{s \lambda^{\ssup 1}_\Lambda} |\phi^{\ssup 1}_\Lambda(y)|^2,
\end{equation}
implying~\eqref{e:boundtopef2} as desired.
\end{proof}

\RV
The proof above has the following corollary, which will be used in Section~\ref{s:localprofiles}.
\begin{corollary}
\label{cor:bdEFwithtimetrans}
For $\Gamma \subset \Lambda$, let $w(x,t)$ be defined as in \eqref{e:boundtopef1}.
Then
\begin{equation}\label{e:bdEFwithtimetrans}
w(x,t-s) \le \texte^{-s \lambda^{\ssup 1}_\Lambda} \phi^{\ssup 1}_\Lambda(x) \sup_{y \in \Gamma} \left\{|\phi^{\ssup 1}_\Lambda(y)|^{-5}\right\} \sum_{y \in \Gamma} w(y,t), \quad x \in \Lambda, \, 0 \le s \le t.
\end{equation}
\end{corollary}
\begin{proof}
Applying the next-to-last inequality in \eqref{e:boundtopef3} with $t$ substituted by $t-s$, we get
\begin{equation}\label{e:prcorbdEF}
w(x,t-s) \le \phi^{\ssup 1}_\Lambda(x) \sup_{y \in \Gamma} \left\{ |\phi^{\ssup 1}_\Lambda(y)|^{-3} \right\} \sum_{y \in \Gamma} w(y,t-s).
\end{equation}
Now use \eqref{e:boundtopef2}, noting that $w(y,t-s)$ is not larger than its left-hand side.
\end{proof}
\eRV


\section{Path expansions}
\label{s:pathexpansions}\nopagebreak\noindent
In this section, we develop a setup to bound the contribution 
of certain specific
classes of random walk paths to the Feynman-Kac formula.
This leads to Propositions~\ref{prop:massclass}--\ref{prop:massclasslargeR} below,
which are the key to the proof of Propositions~\ref{prop:randomtruncation_forproof}--\ref{prop:pathsavoidingZ_forproof}
in Section~\ref{s:massdecomp}, and Propositions~\ref{prop:closetoZinshorttime}--\ref{prop:boundforpathconc} in Section~\ref{s:pathconc}.

\subsection{Key propositions}
To start, we define various sets of nearest-neighbor paths in $\Z^d$ as follows.
For $\ell \in \N_0$ and subsets $\Lambda, \Lambda' \subset \Z^d$, 
define

\begin{equation}
\scrP_\ell(\Lambda,\Lambda') := \left\{ (\pi_0, \ldots, \pi_{\ell}) \in (\Z^d)^{\ell+1} \colon\, 
\begin{aligned}
&\pi_0 \in \Lambda, \pi_{\ell} \in \Lambda',
\\
&|\pi_{i} - \pi_{i-1}|=1 \;\forall\, 1 \le i \le \ell 
\end{aligned}
\right\}
\end{equation}
and set

\RS
\begin{equation}
\scrP(\Lambda, \Lambda') := \bigcup_{\ell \in \N_0} \scrP_\ell(\Lambda,\Lambda'),
\qquad 
\scrP_\ell := \scrP_\ell(\Z^d,\Z^d),
\qquad 
\scrP := \scrP(\Z^d,\Z^d). 
\end{equation}
\eRS
When $\Lambda$ or~$\Lambda'$ 
consists of a single point, we write $x$ instead of $\{x\}$. If $\pi \in \scrP_\ell$, we set $|\pi| := \ell$.
We write $\supp(\pi) := \{\pi_0, \ldots, \pi_{|\pi|}\}$ to denote
the set of points visited by $\pi$.

Let $X=(X_t)_{t\ge0}$ be a continuous-time simple symmetric random walk with total jump rate~$2d$; 
this is the process that ``drives''
the Feynman-Kac formula. We denote by $(T_n)_{n \in \N_0}$ the sequence of
its jump times (with $T_0 := 0$). 
For $\ell \in \N_0$, let $\pi^{\ssup \ell}(X) := (X_0, \ldots, X_{T_{\ell}})$
be the path in $\scrP_\ell$ consisting of the first $\ell$ steps of $X$ and,
for $t  \ge 0$, let
\RS
\begin{equation}\label{e:defpathX0t}
\pi(X_{0,t}) = \pi^{\ssup{\ell_t}}(X), \quad \text{ where } \ell_t \in \N_0 \, \text{ satisfies } \, T_{\ell_t} \le t < T_{\ell_t+1},
\end{equation}
\eRS
denote the path in $\scrP$ consisting of all the steps
taken by $X$ between times $0$ and $t$.

For $\pi \in \scrP$, $L \in \N$ and $A>0$, we define
\begin{equation}\label{e:deflambdaLApi}
\lambda_{L,A}(\pi) := \sup \left\{\lambda^{\ssup 1}_\CC \colon\, \CC \in \mathfrak{C}_{L,A}, \, \supp(\pi)\cap \CC \cap \Pi_{L,A} \neq \emptyset \right\},
\end{equation}
with the convention $\sup \emptyset = -\infty$. 
This is the largest principal eigenvalue among the components of~$\mathfrak C_{L,A}$ 
\RS that have a point of high exceedance visited by \eRS the path.

The main results of this section are the following two propositions.
\begin{proposition}\label{prop:massclass}
Let $R_L$ satisfy \twoeqref{e:propertiesR_L}{e:addpropR_L}.
For any $A > 0$,
there exists a constant $c_A >0$
such that the following holds a.s.\ eventually as $L \to \infty$:
For each $x \in B_L$, each $\NN \subset \scrP(x,\Z^d)$
satisfying $\supp(\pi) \subset B_L$ and $\max_{1 \le \ell \le |\pi|} |\pi_\ell -x| \ge \ln L$ for all $\pi \in \mathcal{N}$, 
each assignment $\pi\mapsto ( \gamma_\pi , z_\pi)\in \R \times \Z^d$ such that
\begin{equation}
\label{e:cond_massclass1}
 \gamma_\pi \ge \lambda_{L,A}(\pi) \vee (\widehat{a}_L - A) + \texte^{-R_L}
\end{equation}
and
\begin{equation}
\label{e:cond_massclass2}
z_\pi \in \supp(\pi) \cup 
\bigcup_{ \substack{\CC \in \mathfrak{C}_{L,A} \colon \\ \supp(\pi) \cap \CC \cap \Pi_{L,A} \neq \emptyset}} \CC
\end{equation}
are true for all $\pi \in \NN$, \RS and all $t \ge 0$, \eRS
\begin{equation}\label{e:mass_class}
\ln \E_x \left[ \texte^{\int_0^t \xi(X_s) \textd s} \1\{\pi(X_{0,t}) \in \mathcal{N}\}\right]
\le \sup_{\pi \in \mathcal{N}} \Big\{ t \gamma_\pi - \left(\RV \ln_3(d L) \eRV - c_A \right) |z_\pi-x| \Big\}. 
\end{equation}
\end{proposition}
\RV
\noindent
We write $\ln_3(dL)$ instead of $\ln_3 L$ in \eqref{e:mass_class} for convenience, as $|z| \le dL$ for any $z \in B_L$.
\eRV

While we assume \twoeqref{e:propertiesR_L}{e:addpropR_L} in most of the paper,
the proof of Proposition~\ref{prop:boundforpathconc} will require us to work without~\eqref{e:addpropR_L}.
In this setting, we have the following:
\begin{proposition}\label{prop:massclasslargeR}
Fix $A>0$ and let $n_A \in \N$ as in Lemma~\ref{l:size_comps}.
For any $R_L \in \N$ that obeys \eqref{e:propertiesR_L}
and any $\vartheta_L \in \N$ such that $\vartheta_L \ll \ln_3 L$ as $L \to \infty$, 
the following holds a.s.\ eventually as $L \to \infty$:
For each $x \in B_L$, each $\NN \subset \scrP(x,\Z^d)$
satisfying $\supp(\pi) \subset B_L$ and $\max_{1 \le \ell \le |\pi|} |\pi_\ell -x| \ge (n_A+1) R_L$ for all $\pi \in \mathcal{N}$, 
each  $\pi\mapsto \gamma_\pi \in \R$ satisfying
\begin{equation}\label{e:condmassclasslargeR}
 \gamma_\pi  \ge \lambda_{L,A}(\pi) \vee (\widehat{a}_L - A) + \texte^{-\vartheta_L R_L}
\quad \forall \,  \pi \in \NN,
\end{equation}
\RS and all $t \ge 0$, \eRS
\begin{equation}\label{e:massclasslargeR}
\ln \E_x \left[ \texte^{\int_0^t \xi(X_s) \textd s} \1\{\pi(X_{0,t}) \in \mathcal{N}\}\right]
\le t \sup_{\pi \in \mathcal{N}} \gamma_\pi - \tfrac12 R_L \ln_3 L. 
\end{equation}
\end{proposition}

The key to the proof of Propositions~\ref{prop:massclass}--\ref{prop:massclasslargeR} 
is Lemma~\ref{l:fixed_class} below,
whose proof in turn depends on intermediate results obtained in the next two sections.
We emphasize that all of these results are deterministic, i.e., they hold for any fixed potential $\xi \in \R^{\Z^d}$.

\subsection{Mass of the solution along excursions}
\label{ss:mass_fixed_path}\noindent
The first step to control the contribution of a path to the mass
is to control the contribution of excursions
 outside of  $\Pi_{L,A}$  (recall~\eqref{defPi}). 
A useful result is the following:
\begin{lemma}[Path evaluation]
\label{l:path_eval}
For any $\ell\in\N_0$, $\pi \in \scrP_\ell$
and $\gamma  > \max_{i < |\pi|} \xi(\pi_i)-2d$,
\begin{equation}\label{e:path_eval}
\E_{\pi_0} \left[\exp\left\{\int_0^{T_{\ell}} (\xi(X_s) -  \gamma ) \textd s\right\} \,\middle|\, \pi^{\ssup {\ell}}(X) = \pi  \right]
= \prod_{i=0}^{\ell-1} \frac{2d}{2d +  \gamma  - \xi(\pi_i)}.
\end{equation}
\end{lemma}
\begin{proof}
The left-hand side of~\eqref{e:path_eval} can be directly evaluated using the fact that 
$T_\ell$ is the sum of~$\ell$ i.i.d.\ Exp($2d$) random variables that are 
independent of $\pi^{\ssup {\ell}}(X)$. 
The condition on  $\gamma$  ensures that all integrals are finite.
\end{proof}

For a path $\pi \in \scrP$, $L \in \N$ and $\varepsilon \in (0,1)$, we write
\begin{equation}\label{e:def_Mpi}
M^{L,\varepsilon}_\pi := \big| \bigl\{ \RS x \in \{\pi_0, \ldots, \pi_{|\pi|-1}\} \eRS \colon\, \xi(x) \le (1-\varepsilon)\widehat{a}_L\bigr\}\big|,
\end{equation}
\RS with the interpretation that $M^{L,\varepsilon}_\pi = 0$ if $|\pi|=0$. \eRS
Then we~have:

\begin{lemma}[Mass of excursions]
\label{l:mass_in}
For any $A, \varepsilon>0$, there exist $c > 0$
 and $L_0 \in \N$  such that,
for all $L \ge L_0$,
all  $\gamma > \widehat{a}_L - A$ 
and all $\pi \in \scrP$ satisfying $\pi_i \notin \Pi_{L,A}$ for all $i < \ell:=|\pi|$,
\begin{equation}\label{e:mass_in}
\E_{\pi_0} \left[ \exp \left\{ \int_0^{T_{\ell}}(\xi(X_t) - \gamma ) \textd s \right\} \,\middle|\, \pi^{\ssup {\ell}}(X) = \pi \right] 
\le q_A^{\ell} \texte^{\left(c -\ln_3 L \right) M^{L,\varepsilon}_\pi},
\end{equation}
where $q_A :=  (1+A/2d)^{-1}$.
\end{lemma}
Note that the statement of Lemma~\ref{l:mass_in} allows for $\pi_{\ell} \in \Pi_{L,A}$.
\begin{proof}
By our assumptions on $\pi$ and $\gamma$, we can use Lemma~\ref{l:path_eval}.
Splitting the product on the right-hand side of~\eqref{e:path_eval} according to whether
$\xi(\pi_i)$ is larger than $(1-\varepsilon)\widehat{a}_L$ or not,
and using  that $\xi(\pi_i) \le \widehat{a}_L - 2A$ for all  $i < |\pi|$,
we bound the left-hand side of \eqref{e:mass_in} by
\begin{equation}\label{e:mass_in1}
q_A^{\ell} \left[q_A \frac{\varepsilon \widehat{a}_L - A}{2d}\right]^{-|\{i < \ell \colon\, \xi(\pi_i) \le (1-\varepsilon)\widehat{a}_L\}|}.
\end{equation}
For large $L$, $\widehat{a}_L \ge \frac12 \rho \ln_2 L $ 
and the number within square brackets in~\eqref{e:mass_in1} exceeds $q_A \varepsilon \rho (\ln_2 L) / 5d > 1$.
Since $|\{i < |\pi| \colon\, \xi(\pi_i) \le (1-\varepsilon)\widehat{a}_L\}| \ge M^{L,\varepsilon}_\pi$, 
\eqref{e:mass_in} holds with $c := \ln (1 \vee 5d (q_A \varepsilon \rho)^{-1})$.
\end{proof}


\subsection{Equivalence classes of paths}
\label{ss:equivclasspaths}\noindent
Here we develop a setup similar as in Section 6.3 of \cite{MP16}.
The idea is to categorize paths $\pi \in \scrP$ according to their
excursions between $\Pi_{L,A}$ and $D_{L,A}^\cc$  (cf.\  \twoeqref{defPi}{def_D_L,A}) 
and then apply the results from Sections~\ref{ss:specbounds} and~\ref{ss:mass_fixed_path}.
Note that $\dist(\Pi_{L,A}, D_{L,A}^\cc) \ge R_L$. 

First we discuss the concatenation of paths.
If $\pi$ and $\pi'$ are two paths in $\scrP$
such that $\pi_{|\pi|} = \pi'_0$,
we define their concatenation as
\begin{equation}\label{def_concat}
\pi \circ \pi' := (\pi_0, \ldots, \pi_{|\pi|}, \pi'_1, \ldots, \pi'_{|\pi'|}) \in \scrP.
\end{equation}
Note that $|\pi \circ \pi'| = |\pi| + |\pi'|$.
\RV
If $\pi_{|\pi|} \neq \pi'_0$,
we can still define the \emph{shifted concatenation} of $\pi$ and $\pi'$
as $\pi \circ \hat{\pi}'$ where 
$\hat{\pi}' := (\pi_{|\pi|}, \pi_{|\pi|}  + \pi'_1 - \pi'_0, \ldots, \pi_{|\pi|} + \pi'_{|\pi'|} - \pi'_0)$.
The shifted concatenation of multiple paths is then defined inductively via associativity.
\eRV

If a path $\pi \in \scrP$ intersects $\Pi_{L,A}$,
then it can be decomposed into an initial path, a sequence of  excursions between $\Pi_{L,A}$ and $D_{L,A}^\cc$,
and a terminal path. Explicitly, 
there exists $m_\pi \in \N $ such that
\begin{equation}\label{e:concat1}
\pi = \check{\pi}^{\ssup 1} \circ \hat{\pi}^{\ssup 1} \circ \cdots \circ \check{\pi}^{\ssup {m_\pi}} \circ \hat{\pi}^{\ssup {m_\pi}} \circ \bar{\pi},
\end{equation}
where the paths in~\eqref{e:concat1} satisfy
\begin{equation}\label{e:concat2}
\begin{alignedat}{9}
\check{\pi}^{\ssup 1} & \in  \scrP(\Z^d, \Pi_{L,A}) 
&\qquad\text{and}\qquad& 
\check{\pi}^{\ssup 1}_i & \notin  \Pi_{L,A}, & \quad\, 0\le i < |\check{\pi}^{\ssup 1}|, 
\\
\check{\pi}^{\ssup k} & \in  \scrP(D_{L,A}^{c}, \Pi_{L,A}) 
&\qquad\text{and}\qquad& 
\check{\pi}^{\ssup k}_i & \notin  \Pi_{L,A}, & \quad\, 0\le i < |\check{\pi}^{\ssup k}|, \; 2 \le k \le m_\pi, 
\\
\hat{\pi}^{\ssup k} & \in  \scrP(\Pi_{L,A}, D_{L,A}^\cc) 
&\qquad\text{and}\qquad& 
\hat{\pi}^{\ssup k}_i & \in  D_{L, A}, & \quad\, 0\le i < |\hat{\pi}^{\ssup k}|, \; 1 \le k \le m_{\pi} - 1, 
\\
\hat{\pi}^{\ssup {m_\pi}} & \in  \scrP(\Pi_{L,A}, \Z^d) 
&\qquad\text{and}\qquad& 
\hat{\pi}^{\ssup {m_\pi}}_i & \in  D_{L,A}, & \quad\, 0\le i < |\hat{\pi}^{\ssup {m_\pi}}|, 
\end{alignedat}
\end{equation}
while
\begin{equation}\label{e:concat3}
\begin{array}{ll} 
\bar{\pi} \in \scrP(D^c_{L,A}, \Z^d), \, \bar{\pi}_i \notin \Pi_{L,A} \; \forall\, i \ge 0 & \text{ if } \hat{\pi}^{\ssup {m_\pi}} \in \scrP(\Pi_{L,A}, D^\cc_{L, A}), \\
\bar{\pi}_0 \in D_{L,A}, |\bar{\pi}| = 0  & \text{ otherwise.}
\end{array}
\end{equation}
Note that the decomposition \twoeqref{e:concat1}{e:concat3} is unique, and that
the paths $\check{\pi}^{\ssup 1}$, $\hat{\pi}^{\ssup {m_\pi}}$ and $\bar{\pi}$ can have zero length.
\RS If $\pi$ is contained in $B_L$, so are all the paths in the decomposition. \eRS

For $L \in \N$ and $\varepsilon > 0$, whenever $\supp(\pi) \cap \Pi_{L,A} \ne \emptyset$,
we define
\begin{align}
n_\pi := \sum_{i=1}^{m_\pi} |\check{\pi}^{\ssup i}| + |\bar{\pi}| \qquad \text{ and } \qquad
k^{L,\varepsilon}_\pi := \sum_{i=1}^{m_\pi} M^{L,\varepsilon}_{\check{\pi}^{\ssup i}} + M^{L,\varepsilon}_{\bar{\pi}} \label{e:defnpikpi}
\end{align}
to be respectively the total time spent in exterior excursions and the sum of the numbers of moderately low points of the potential visited 
by exterior excursions (excluding their last point). In the case when 
$\supp(\pi) \cap \Pi_{L,A} = \emptyset$,
we set $m_\pi := 0$, $n_\pi := |\pi|$ and $k^{L,\varepsilon}_\pi := M^{L,\varepsilon}_{\pi}$.
Recall from~\eqref{e:deflambdaLApi} that, in this case, $\lambda_{L,A}(\pi) = -\infty$.

We say that $\pi, \pi' \in \scrP$ are \emph{equivalent}, written $\pi' \sim \pi$, 
if $m_{\pi} = m_{\pi'}$, $\check{\pi}'^{\ssup i}=\check{\pi}^{\ssup i}$ for all $i=1,\ldots,m_{\pi}$ and $\bar{\pi}' = \bar{\pi}$ if $\bar{\pi}_0 \in D^\cc_{L,A}$.
If $\pi' \sim \pi$, then $n_{\pi'}$, $k^{L, \varepsilon}_{\pi'}$ and $\lambda_{L,A}(\pi')$ are all equal to the counterparts for $\pi$.

To state our key lemma, we define, for $m,n \in \N_0$,
\begin{equation}\label{e:defPmn}
\scrP^{(m,n)} = \left\{ \pi \in \scrP \colon\, m_\pi = m, n_\pi = n \right\},
\end{equation}
and we denote by
\begin{equation}\label{def_CLA}
C_{L,A}:= \max \{|\CC| \colon\, \CC \in \mathfrak{C}_{L,A}\}
\end{equation}
the maximal size of the islands in $\mathfrak{C}_{L,A}$.

\begin{lemma}
\label{l:fixed_class}
For any $A,\varepsilon > 0$, there exist $c>0$ and $L_0 \in \N$ such that, for all $L \ge L_0$, 
all $m,n \in \N_0$,  all $\pi \in \scrP^{(m,n)}$ with $\supp(\pi) \subset B_L$,  
all $ \gamma > \lambda_{L,A}(\pi) \vee (\widehat{a}_L -A)$
and all $t \ge 0$,
\begin{multline}\label{e:fixed_class}
\qquad
\E_{\pi_0} \left[ \texte^{\int_0^t (\xi(X_s) - \gamma) \textd s} \1\{\pi(X_{0,t}) \sim \pi\} \right]  
\\\le \left(C_{L,A}^{3/2} \right)^{\1_{\{m>0\}}} \left(1+\frac{ 2d \, C_{L,A}}{\gamma - \lambda_{L,A}(\pi)} \right)^m \left(\frac{q_A}{2d}\right)^n \texte^{\left(c-\ln_3 L \right) k^{L,\varepsilon}_{\pi}}.
\qquad
\end{multline}
\end{lemma}
\begin{proof}
Fix $A, \varepsilon>0$ and let $c>0$, $L_0 \in \N$ be as given by Lemma~\ref{l:mass_in}.
For $0 \le s \le t <\infty$, set $I_s^t := \texte^{\int_s^t(\xi(X_u) - \gamma ) \textd u}$. 
Our strategy is to prove the claim by induction on~$m$.

Suppose first that $m=1$, let $\ell := |\check{\pi}^{\ssup 1}|$ and set
$z := \check{\pi}^{\ssup 1}_{\ell}$.
There are two possibilities: either $\bar{\pi}_0$ belongs to $D_{L, A}$ or not. Focussing first on the case $\bar{\pi}_0\in D_{L,A}$, which in particular implies $|\bar{\pi}|=0$, the strong Markov property yields
\begin{align}
\label{e:fixedclass1}
\E_{\pi_0} \Bigl[I_0^t &\1_{\{ \pi(X_{0,t}) \sim \pi \}} \Bigr]
= \E_{\pi_0} \left[I_0^{T_\ell} I_{T_\ell}^t  \1_{\{ \pi^{\ssup \ell}(X) = \check{\pi}^{\ssup 1}\}}\mathbbm1_{\{T_\ell < t\}}\1_{\{X_{s+T_\ell} \in D_{L,A} \, \forall s\in[0, t - T_\ell]\}} \right] \nonumber\\
&= \E_{\pi_0} \left[I_0^{T_\ell} \1_{\{ \pi^{\ssup \ell}(X) = \check{\pi}^{\ssup 1}\}}\1_{\{ T_\ell < t \}} \left(\E_z \left[ I_{0}^{t-s} \1_{\{\tau_{D^\cc_{L,A}} > t-s \}} \right] \right)_{s=T_\ell} \right].
\end{align}
Since $z \in \Pi_{L,A}$, we may write $\CC_z$ to denote the island in $\mathfrak{C}_{L,A}$ containing $z$.  As $\tau_{D^\cc_{L,A}} = \tau_{\CC^\cc_z}$ $\P_z$-a.s., 
Lemma~\ref{l:bounds_mass} and our hypothesis on~$\gamma$ bound the inner expectation in~\eqref{e:fixedclass1} by $|\CC_z|^{3/2}$. 
Applying Lemma~\ref{l:mass_in}, we further bound~\eqref{e:fixedclass1} by
\begin{equation}\label{e:fixedclass2}
|\CC_z|^{3/2} \E_{\pi_0} \left[ I_0^{T_\ell} \1_{\{ \pi^{\ssup \ell}(X) = \check{\pi}^{\ssup 1} \}}\right]
\le C_{L,A}^{3/2} \left(\frac{q_A}{2d}\right)^{\ell} \texte^{\left(c-\ln_3 L \right) M^{L,\varepsilon}_{\check{\pi}^{\ssup 1}}},
\end{equation}
thus proving~\eqref{e:fixed_class} in the case $m=1$, $\bar{\pi}_0\in D_{L,A}$.

Assume next $x := \bar{\pi}_0 \in D^\cc_{L,A}$.
Abbreviating $\sigma := \inf\{ s> T_\ell \colon\, X_s \notin D_{L,A} \}$, 
write
\begin{align}\label{e:fixedclass3}
\E_{\pi_0} \left[I_0^t \1_{\{ \pi(X_{0,t}) \sim \pi \}} \right]
\le \E_{\pi_0} \left[I_0^{\sigma} \1_{\{ \pi^{\ssup \ell}(X) = \check{\pi}^{\ssup 1}, \sigma < t \}} 
\left( \E_x \left[I_0^{t-s} \1_{\{\pi(X_{0,t-s}) = \bar{\pi}\}} \right] \right)_{s=\sigma} \right].
\end{align}
Let $\ell_* := |\bar{\pi}|$ and note that, 
since $\bar{\pi}_{\ell_*} \notin \Pi_{L,A}$,
by the hypothesis on $\gamma$ we have
\begin{align}\label{e:fixedclass4}
\E_x \left[I_0^{t-s} \1_{\{\pi(X_{0,t-s}) = \bar{\pi}\}} \right]
\le \E_x \left[I_0^{T_{\ell_*}}\1_{\{\pi^{\ssup{\ell_*}}(X) = \bar{\pi}\}} \right]
\le \left(\frac{q_A}{2d}\right)^{\ell_*} \texte^{\left(c-\ln_3 L \right) M^{L,\varepsilon}_{\bar{\pi}}}
\end{align}
by Lemma~\ref{l:mass_in}.
On the other hand, by Lemmas~\ref{l:mass_out} and~\ref{l:mass_in},
\begin{align}\label{e:fixedclass5}
\E_{\pi_0} \left[I_0^{\sigma} \1_{\{ \pi^{\ssup \ell}(X) = \check{\pi}^{\ssup 1}\}} \right]
& = \E_{\pi_0} \left[I_0^{T_\ell} \1_{\{ \pi^{\ssup \ell}(X) = \check{\pi}^{\ssup 1} \}} \right] \E_z \left[I_0^{\tau_{\CC^\cc_z}} \right]\nonumber\\
& \le \left( 1 + \frac{2d \, C_{L,A}}{\gamma  - \lambda_{L,A}(\pi)} \right) \left(\frac{q_A}{2d}\right)^{\ell} \texte^{\left(c-\ln_3 L \right) M^{L,\varepsilon}_{\check{\pi}^{\ssup 1}}}.
\end{align}
Putting together~\eqref{e:fixedclass3}--\eqref{e:fixedclass5}, we finish the proof of the case $m=1$.

By induction, assume now that the statement is proven for some fixed $m \ge 1$,
and let $\pi \in \scrP^{(m+1,n)}$. 
Define $\pi' := \check{\pi}^{\ssup 2} \circ \hat{\pi}^{\ssup 2} \circ \cdots \circ \check{\pi}^{\ssup {m+1}} \circ \hat{\pi}^{\ssup {m+1}} \circ \bar{\pi}$.
Then $\pi' \in \scrP^{(m,n')}$ where $n = |\check{\pi}^{\ssup 1}| + n'$, 
and $k^{L,\varepsilon}_\pi = k^{L,\varepsilon}_{\pi'}+M^{L,\varepsilon}_{\check{\pi}^{\ssup 1}}$.
Setting $\ell := |\check{\pi}^{\ssup 1}|$, $\sigma := \inf\{ s> T_\ell \colon\; X_s \notin D_{L,A}\}$
and $x:=\check{\pi}^{\ssup 2}_0$, we get
\begin{equation}\label{e:fixedclass6}
\E_{\pi_0} \left[ I_0^t \1_{\{\pi(X_{0,t}) \sim \pi \}}\right]
\le \E_{\pi_0} \left[I_0^\sigma \1_{\{\pi^{\ssup \ell}(X) = \check{\pi}^{\ssup 1}, \sigma < t\}} \left(\E_x \left[ I_0^{t-s} \1_{\{\pi(X_{0,t-s}) \sim \pi'\}}\right] \right)_{s=\sigma} \right],
\end{equation}
from which~\eqref{e:fixed_class} follows using the induction hypothesis and~\eqref{e:fixedclass5}.
The case $m=0$ follows from equation~\eqref{e:fixedclass4} after substituting $\bar{\pi}$ by $\pi$ and $t-s$ by $t$.
\end{proof}


\subsection{Proof of Propositions~\ref{prop:massclass}--\ref{prop:massclasslargeR}}
\noindent
We are now ready to present the proofs of the above key propositions.
\begin{proof}[Proof of Proposition~\ref{prop:massclasslargeR}]
The proof is based on Lemma~\ref{l:fixed_class} 
and results from Sections~\ref{ss:properties_components}--\ref{ss:connect_prop}.
Fix $A>0$ and, for $\beta$ as in~\eqref{e:propertiesR_L}, take $\varepsilon \in (0,\beta/2)$
as in Lemma~\ref{l:bound_mediumpoints}.
Let $L_0 \in \N$ be as given by Lemma~\ref{l:fixed_class}
and take $L \ge L_0$ so large that the conclusions of Lemmas~\ref{l:bound_mediumpoints} and \ref{l:size_comps} hold.
Fix $x \in B_L$.
Recall the definition of $\scrP^{(m,n)}$. Noting that the relation $\sim$ is an equivalence relation in $\scrP^{(m,n)}$,
define
\begin{equation}\label{e:propmassclass2}
\widetilde{\scrP}^{(m,n)}_x := \{\text{equivalence classes of the paths in } \scrP(x,\Z^d) \cap \scrP^{(m,n)}  \}.
\end{equation}
We first claim that, for a constant $c_1 \in \N$, a.s.~eventually as $L\to\infty$,
\begin{equation}\label{e:propmassclass3}
|\widetilde{\scrP}^{(m,n)}_x| \le (c_1 R_{L}^{d})^m (2d)^n \qquad \forall \, m,n \in \N_0.
\end{equation}
Indeed,~\eqref{e:propmassclass3} is clear if $m=0$.
To prove it in the case $m \ge 1$,
write, for $\Lambda \subset \Z^d$, $\partial \Lambda := \{z \notin \Lambda \colon\, \dist(z, \Lambda)=1\}$.
By Lemma~\ref{l:size_comps}, there is a $c_0 \in \N$ such that 
\begin{equation}\label{e:propmassclass4}
|\partial \CC| \le 2 d|\CC| \le c_0 R_L^d \;\; \forall\,  \CC \in \mathfrak{C}_{L,A} \;\;  \text{a.s.~eventually as } L\to\infty.
\end{equation}
We then define a map $\Phi\colon\widetilde{\scrP}^{(m,n)}_x \to\scrP_n(x,\Z^d) \times \{1, \ldots, c_0 R_{L}^{d} + 1 \}^m$ as follows:
For each $\Lambda \subset \Z^d$ with $1 \le |\Lambda| \le c_0  R_{L}^{d}$,
fix an injection $f_\Lambda\colon \Lambda \to \{1, \ldots, c_0 R_{L}^{d} \}$.
\RV
Given a path $\pi \in \scrP^{(m,n)} \cap \scrP(x,\Z^d)$, decompose~$\pi$ as in~\eqref{e:concat1}, 
and denote by $\widetilde{\pi}$ the shifted concatenation, as defined after \eqref{def_concat}, 
of $\check\pi^{\ssup 1}, \ldots, \check\pi^{\ssup m}$, $\bar{\pi}$.
\eRV
Note that, for each $2\le k\le m$, the starting point $\check\pi^{\ssup k}_0$ lies in $\partial\CC_k$ for some~$\CC_k\in \mathfrak{C}_{L,A}$,
while $\bar{\pi}_0 = \bar{\pi}_0 \in \partial \overline{\CC} \cup \overline{\CC}$ for some $\overline{\CC} \in \mathfrak{C}_{L,A}$.
Thus we may set
\begin{equation}
\Phi(\pi):= \left\{
\begin{array}{ll}
\bigl(\widetilde \pi,f_{\partial \CC_2}(\check{\pi}^{\ssup 2}_0),\dots,f_{\partial \CC_m}(\check{\pi}^{\ssup{m}}_0), c_0 R_L^d + 1 \bigr)
& \text{ if } \bar{\pi}_0 \in \overline{\CC} \subset D_{L,A}, \\
\bigl(\widetilde \pi,f_{\partial \CC_2}(\check{\pi}^{\ssup 2}_0),\dots,f_{\partial \CC_m}(\check{\pi}^{\ssup{m}}_0), f_{\partial \bar{\CC}}(\bar{\pi}_0) \bigr)
& \text{ if } \bar{\pi}_0 \in \partial \overline{\CC} \subset D_{L,A}^\cc.
\end{array}\right.
\end{equation}
As is readily checked, $\Phi(\pi)$ depends only on the equivalence class of~$\pi$ and, 
when restricted to equivalence classes, $\Phi$ is injective.
Thus~\eqref{e:propmassclass3} follows with e.g.\ $c_1 := 2 c_0$.

Take now $\NN \subset \scrP(x, \Z^d)$ as in the statement, and set
\begin{equation}\label{e:propmassclass1}
\widetilde{\mathcal{N}}^{(m,n)} := \{\text{equivalence classes of paths in } \NN \cap \scrP^{(m,n)}\} \subset \widetilde{\scrP}^{(m,n)}_x.
\end{equation}
\RS Choose for each $\MM \in \widetilde{\NN}^{(m,n)}$ a representative $\pi_\MM \in \MM$ and use \eqref{e:propmassclass3} to write \eRS
\begin{align}\label{e:propmassclass6}
& \E_x \left[ \texte^{\int_0^t \xi(X_s) \textd s} \1_{\{\pi(X_{0,t}) \in \mathcal{N}\}} \right] 
= \sum_{m, n \in \N_0} \RS \sum_{\MM \in \widetilde{\mathcal{N}}^{(m,n)}}\E_x \left[ \texte^{\int_0^t \xi(X_s) \textd s} \1_{\{\pi(X_{0,t}) \sim \pi_\MM \}} \right] \eRS \nonumber\\
& \quad\qquad \le \sum_{m, n \in \N_0} (c_1 R_{L}^{d})^m (2d)^n \sup_{\pi \in \NN^{(m,n)}} \E_x \left[ \texte^{\int_0^t \xi(X_s) \textd s} \1_{\{\pi(X_{0,t}) \sim \pi\}} \right],
\end{align}
where we use the convention $\sup \emptyset = 0$.
For fixed $\pi \in \mathcal{N}^{(m,n)}$,
by \eqref{e:cond_massclass1} we may apply~\eqref{e:fixed_class}, 
Lemma~\ref{l:size_comps} and~\eqref{e:propertiesR_L} to obtain, for all $L$ large enough,
\begin{align}\label{e:propmassclass7}
(c_1 R_{L}^{d})^m (2d)^n \E_x \left[ \texte^{\int_0^t \xi(X_s) \textd s} \1_{\{\pi(X_{0,t}) \sim \pi\}} \right] 
\le \texte^{t \gamma_\pi } \left(R_L^{4d}\texte^{\vartheta_L R_L} \right)^m  q_A^n \texte^{\left(c-\ln_3 L \right) k^{L,\varepsilon}_\pi}.
\end{align}
We now  claim that, for large enough $L$,
\begin{equation}\label{e:propmassclass7.1}
k_\pi^{L,\varepsilon} \ge \left\{(m-1)\vee 1\right\} R_L \{1-(\ln L)^{-\varepsilon} - R_L^{-1}\}.
\end{equation}
Indeed, when $m=0$, $|\supp(\pi)| \ge \max_{1 \le \ell \le |\pi|} |\pi_\ell -x| \ge (n_A+1) R_L$ by assumption.
When $m\ge 2$, $|\supp(\check{\pi}^{\ssup i})| \ge R_L$ for all $2 \le i \le m$.
When $m=1$, there are two cases: if $\supp(\check{\pi}^{\ssup 1}) \cap D^\cc_{L,A} \neq \emptyset$, then $|\supp(\check{\pi}^{\ssup 1})|\ge R_L$
while, if $\supp(\check{\pi}^{\ssup 1}) \subset D_{L,A}$, then $|\supp(\bar{\pi})| \ge R_L$ by Lemma~\ref{l:size_comps}.
Thus~\eqref{e:propmassclass7.1} holds by \eqref{e:defnpikpi}, \eqref{e:def_Mpi} and Lemma~\ref{l:bound_mediumpoints}.
Using~\eqref{e:propmassclass7.1}, \eqref{e:propertiesR_L} and $\vartheta_L \ll \ln_3 L$, 
we may further bound~\eqref{e:propmassclass7} by
\begin{multline}
\label{e:propmassclass8}
\qquad
\left[ R_L^{8d} \texte^{2 \vartheta_L R_L} \texte^{-(2 \vartheta_L+\tfrac12) R_L} \right]^{(m-1)\vee 1} q_A^n \texte^{t \gamma_\pi } \texte^{\left(c+1+2\vartheta_L-\ln_3 L \right) k^{L,\varepsilon}_\pi}
\\
\le \left(\texte^{-\tfrac{R_L}{3}} \right)^{(m-1)\vee 1} q_A^n \texte^{t \gamma_\pi} \texte^{\left(c+1+2 \vartheta_L-\ln_3 L \right) k^{L,\varepsilon}_\pi}.
\end{multline}
\eRV
Inserting this back into~\eqref{e:propmassclass6}, we obtain
\begin{equation}\label{e:intermediatemassclass}
\E_x \left[ \texte^{\int_0^t \xi(X_s) \textd s} \1_{\{\pi(X_{0,t}) \in \mathcal{N}\}} \right]
\le \sup_{\pi \in \mathcal{N}} \exp \left\{ t \gamma_\pi + \left(c+1+2\vartheta_L-\ln_3 L \right) k^{L,\varepsilon}_\pi \right\}.
\end{equation}
Now~\eqref{e:massclasslargeR} follows from~\eqref{e:intermediatemassclass}, \eqref{e:propmassclass7.1}, \eqref{e:propertiesR_L} and $\vartheta_L \ll \ln_3 L$.
\end{proof}

\begin{proof}[Proof of Proposition~\ref{prop:massclass}]
Note that, for large $L$, the assumptions of Proposition~\ref{prop:massclass} 
imply those of Proposition~\ref{prop:massclasslargeR} with $\vartheta_L \equiv 1$,
and thus we may use~\eqref{e:intermediatemassclass}.
We proceed to bound $k^{L,\varepsilon}_{\pi}$ using assumption~\eqref{e:addpropR_L}.
Recall that we take $\beta$ as in \eqref{e:propertiesR_L}
and $\varepsilon \in (0,\beta/2)$ as in Lemma~\ref{l:bound_mediumpoints}. 
Let $C \ge 1$ be as in Lemma~\ref{l:boundhighexceedances}
and, for $\alpha \in (0,1/d)$ as in~\eqref{e:addpropR_L},
take $\delta \in (\alpha d, 1)$ and set $\varepsilon':= \delta - \alpha d >0$.
We assume that $L$ is so large that the conclusions of Lemma~\ref{l:bound_mediumpoints} (with $\beta$,$\varepsilon$ as above) and Lemma~\ref{l:boundhighexceedances} (with $\delta$ as above) are in place.

Note that, by Lemma~\ref{l:size_comps},
there exists a constant $c_2 \in (0, \infty)$ such that
\begin{equation}\label{e:propmassclass10}
k^{L,\varepsilon}_\pi \ge M^{L,\varepsilon}_{\pi} -  |\supp(\pi) \cap \Pi_{L, A}| c_2 R^d_L.
\end{equation}
By our assumptions on $\NN$, 
$|\supp(\pi)| \ge \ln L \ge C (\ln L)^\delta$ for large $L$.
By Lemma~\ref{l:boundhighexceedances},
\begin{equation}\label{e:propmassclass10.5}
 |\supp(\pi) \cap \Pi_{L, A}| \le \frac{|\supp(\pi)|}{(\ln L)^\delta} \le \frac{|\supp(\pi)|}{R^d_L (\ln L)^{\varepsilon'}} 
\end{equation}
by~\eqref{e:addpropR_L}.
By Lemma~\ref{l:bound_mediumpoints}, 
$M^{L,\varepsilon}_{\pi} +1 \ge |\supp(\pi)|\{1-(\ln L)^{-\varepsilon}\}$.
Thus
\begin{equation}\label{e:propmassclass11}
k^{L,\varepsilon}_\pi \ge |\supp(\pi)| \left\{1 - (\ln L)^{-1} - (\ln L)^{-\varepsilon} - c_2 (\ln L)^{-\varepsilon'} \right\}.
\end{equation} 
Now, by Lemma~\ref{l:size_comps} and \eqref{e:cond_massclass2},
$|\supp(\pi)| \ge |z_\pi -x| - n_A R_L$;
this  in conjunction
with $|\supp(\pi)|\ge \ln L$ implies
\begin{equation}\label{e:propmassclass11.5}
|\supp(\pi)| \ge |z_\pi - x| \left(1 - \frac{n_A R_L}{\ln L} \right).
\end{equation}
From \twoeqref{e:propmassclass11}{e:propmassclass11.5} and~\eqref{e:addpropR_L} we obtain
$\left(c+3-\ln_3 L \right) k^{L,\varepsilon}_\pi \le \left(c+4-\ln_3 (d L) \right) |z_\pi - x|$
for large enough $L$, which together with~\eqref{e:intermediatemassclass} (with $\vartheta_L \equiv 1$) implies~\eqref{e:mass_class}.
\end{proof}

\section{Analysis of the cost functional}
\label{s:cost}\nopagebreak\noindent
In this section, we identify the order statistics of $\Psi_t$
and give the proofs of Theorem~\ref{thm:aging_locus} and Propositions~\ref{prop:goodevent_forproof}--\ref{prop:stabilitygap}.
Motivated by Proposition~\ref{prop:massclass} and Lemma~\ref{l:comparisoncapitalsislands}, 
we define the following generalization of the cost functional:
For $t > 0$ and $c \in \R$, let
\begin{equation}\label{e:defgenPsi}
\Psi_{t,c}(z) := \lambda^{\scrC}(z) - \left(\ln_3^+ |z| - c \right)^+ \frac{|z|}{t}, \;\;\; z \in \scrC,
\end{equation}
where $\lambda^\scrC(z)$ is as in~\eqref{e:deflambdascrCz}.
Arguing as for~\eqref{e:welldefined}, we can see that, almost surely,
\begin{equation}\label{e:welldefgenPsi}
|\{z \in \scrC \colon \Psi_{t,c}(z) > \eta \}| < \infty \quad \text{ for all } t > 0, \eta \in \R,
\end{equation}
and thus we may define $\Psi^{\ssup k}_{t,c}$ and $Z^{\ssup k}_{t,c}$ analogously to the corresponding objects for $\Psi_t$.

\RV
Let us now identify the scale $a_t$ in Theorem~\ref{thm:locus}.
Noting that $r_t$ is strictly increasing for large enough $t$,
we may take $t \mapsto L^*_t \in \N$ such that $L^*_{r_t} = L_t$.
Set $N_t := \lfloor \tfrac12  \sqrt{ \rho t/ d }\rfloor$, $\widehat{N}_t := N_{L^*_t}$
and define $a_t$ as the smallest positive number satisfying
\begin{equation}\label{e:defat}
\textnormal{Prob} \left( \lambda^{\ssup 1}_{B_{\widehat{N}_t}} > a_t \right) = \left(\frac{(\ln t) (\ln_2 t) \ln_3 t}{t}\right)^{d/2}.
\end{equation}
Such an $a_t$ exists  (for $t$ large enough) since the principal Dirichlet eigenvalue of $H$ in $B_{\widehat{N}_t}$ is continuously distributed.
Moreover, since $\widehat{N}_t$ is non-decreasing and the right-hand side of \eqref{e:defat} is eventually non-increasing, 
by \eqref{e:monot_princev} we can take $a_t$ non-decreasing as well.

Note that, as $t \to \infty$,
\begin{equation}\label{e:growthL*t}
L^*_t \sim \frac d \rho t (\ln t) (\ln_2 t) \ln_3 t \quad \text{ and } \quad 2 \widehat{N}_t \sim \sqrt{t (\ln t) (\ln_2 t) \ln_3 t}.
\end{equation}
\eRV
An important result of \cite{BK16} (Theorem 2.4 therein) is that, for any $\theta \in \R$,
\begin{equation}\label{e:maxorderlambda}
\lim_{t \to \infty} \frac{t^d}{(2 \widehat{N}_t)^d} \textnormal{Prob}\left( \lambda^{\ssup 1}_{B_{\widehat{N}_t}} > a_t + \theta d_t \right) = \texte^{-\theta},
\end{equation}
where $d_t$ is as in~\eqref{e:def_fundam_scales}.
A strengthened version of this statement 
\RV
(more precisely, \eqref{e:convmeasure} with $\widehat{Y}_t(0)$ as in \eqref{e:defhatYt} below) 
\eRV
will allow us to identify the order statistics of $\Psi_{t,c}$.
Together with Theorem~2.3 and Lemma~6.8 in \cite{BK16}, 
\eqref{e:maxorderlambda} implies that $a_t = \widehat{a}_t - \chi + o(1)$.
In particular, $a_t = (\rho+o(1)) \ln_2 t$.

For $0<a\le b <\infty$, $c \in \R$ and $k \in \N$, we define the events
\begin{equation}\label{e:defgoodevents}
\begin{aligned}
\EE^{\ssup k}_{t,a,b,c} := & \left\{ \min_{i=1,\ldots, k} \left(\Psi^{\ssup i}_{at,c} - \Psi^{\ssup {i+1}}_{at,c}\right) \wedge \left(\Psi^{\ssup i}_{bt,c} - \Psi^{\ssup {i+1}}_{bt,c}\right)> d_t e_t \right\} \\
& \;\;\; \cap \bigcap_{s \in [at,bt]}  \left\{ a_{r_t}+ d_t g_t > \Psi^{\ssup 1}_{s,c} \ge \Psi^{\ssup k}_{s,c} > a_{r_t} - d_t g_t \right\}\\  
& \;\;\; \cap \bigcap_{s \in [at,bt]} \left\{r_t f_t < \min_{1\le i \le k } |Z^{\ssup i}_{s,c}|  \le \max_{1\le i \le k } |Z^{\ssup i}_{s,c}| < r_t g_t  \right\}.
\end{aligned}
\end{equation}
When $c=0$ and/or $k=1$, we omit them in the notation.

For $a \in (0,\infty)$, 
let $\CC([a,\infty),\R^n)$, resp.\ $\DD([a,\infty),\R^n)$,
denote the set of continuous, resp.\ c\`adl\`ag, functions from $[a,\infty)$ to $\R^n$, 
both equipped with the Skorohod topology (i.e., the $J_1$ topology).
The following result is the main objective of this section.
\begin{proposition}\label{prop:PPPconv}
For all $c \in \R$, $k \in \N$ and  $a > 0$, the stochastic process
\begin{equation*}
\left(\left( \frac{\Psi^{\ssup 1}_{\theta t,c} - a_{r_t}}{d_{r_t}}, \frac{\lambda^\scrC(Z^{\ssup 1}_{\theta t, c}) - a_{r_t}}{d_{r_t}}, \frac{Z^{\ssup 1}_{\theta t,c}}{r_t} \right), \ldots, \left( \frac{\Psi^{\ssup k}_{\theta t,c } -  a_{r_t}}{d_{r_t}}, \frac{\lambda^\scrC(Z^{\ssup k}_{\theta t, c}) - a_{r_t}}{d_{r_t}}, \frac{Z^{\ssup k}_{\theta t,c}}{r_t} \right) \right)_{\theta \in [a,\infty)}
\end{equation*}
belongs a.s.\ to $(\CC([a, \infty), \R) \times \DD([a, \infty), \R) \times \DD([a,\infty), \R^d))^k$ and converges in distribution as $t \to \infty$ with respect to the Skorohod topology of $\DD \left([a, \infty), (\R \times \R \times \R^d)^k \right)$ to the process
\begin{equation}\label{e:limitPPP}
\left(\left(\overline{\Psi}^{\ssup 1}_\theta, \overline{\Lambda}^{\ssup 1}_\theta, \overline{Z}^{\ssup 1}_\theta \right), \ldots, \left( \overline{\Psi}^{\ssup k}_\theta, \overline{\Lambda}^{\ssup k}_\theta, \overline{Z}^{\ssup k}_\theta  \right) \right)_{\theta \in [a,\infty)}
\end{equation}
where $\overline{\Psi}^{\ssup i}_\theta := \overline{\Lambda}^{\ssup i}_\theta - \tfrac{1}{\theta} |\overline{Z}^{\ssup i}_\theta|$
and $(\overline{\Lambda}^{\ssup i}_\theta, \overline{Z}^{\ssup i}_\theta)_{i=1}^k$ are the $k$ first ordered maximizers of the functional
$\psi_\theta(\lambda, z) = \lambda - \frac{|z|}{\theta}$
over the points $(\lambda,z)$ of a Poisson point process on $\R \times \R^d$ with intensity $\texte^{-\lambda} \textd \lambda \otimes \textd z$,
chosen in such a way that $\overline{\Psi}^{\ssup i}_\theta$ is continuous and $\overline{\Lambda}^{\ssup i}_\theta$,  $\overline{Z}^{\ssup i}_\theta$ c\`adl\`ag.
In particular, the probability of the event $\EE^{\ssup k}_{t,a,b,c}$ defined in~\eqref{e:defgoodevents} converges to $1$ as $t \to \infty$
and, for any fixed $\theta\in (0,\infty)$, the random vector
\begin{equation*}
\left( \frac{\Psi^{\ssup 1}_{\theta t,c} - a_{r_t}}{d_{r_t}}, \frac{Z^{\ssup 1}_{\theta t,c}}{r_t} \right), \ldots, \left( \frac{\Psi^{\ssup k}_{\theta t,c } -  a_{r_t}}{d_{r_t}}, \frac{Z^{\ssup k}_{\theta t,c}}{r_t} \right)
\end{equation*}
converges in law  to a random vector in $(\R \times \R^d)^k$ with distribution given by
\begin{equation}\label{e:limitingdensity1ddist}
\1\{\psi_1 > \cdots > \psi_k \}\texte^{- \left( \frac1\theta|z_1|+\cdots+ \frac1\theta|z_k| + \psi_1 + \cdots + \psi_k + (2 \theta)^d \texte^{-\psi_k}\right)} \prod_{i=1}^k \textd \psi_i \otimes \textd z_i.
\end{equation}
\end{proposition}

From this we immediately get:
\begin{proof}[Proof of Proposition~\ref{prop:goodevent_forproof}]
\RV
Directly follows from Proposition~\ref{prop:PPPconv}, \eqref{e:def_fundam_scales} and $a_t \sim \rho \ln_2 t$.
\eRV
\end{proof}

\RV
With the help of the results from Section~\ref{s:preparation}, we also obtain:
\begin{proof}[Proof of Proposition~\ref{prop:seprelcap}]
In light of Proposition~\ref{prop:PPPconv}, Lemma~\ref{l:comparisoncapitalsislands}, Lemma~\ref{l:properties_opt_comps}(iii),
and Lemmas~\ref{l:maxpotential}--\ref{l:size_comps},
the result follows by setting
\begin{equation*}
\scrC_t := \left\{ z \in \scrC \colon\, B_{\varrho_z}(z) \subset B_{L_t}, \, \lambda^{\scrC}(z) > a_{r_t} - d_t g_t  \right\}
\end{equation*}
and noting that $\lambda^{\scrC}(z) \ge \Psi_s(z)$, $a_{r_t} = \widehat{a}_{L_t} - \chi + o(1)$ and $d_t g_t = o(1)$.
\end{proof}
\eRV

Note that the part of Theorem~\ref{thm:locus} concerning $(Z_t)_{t > 0}$ already follows from Proposition~\ref{prop:PPPconv}.
Another useful consequence is the following comparison between $\Psi_{t,c}$ and $\Psi_{t}$.
\begin{lemma}
\label{l:bound_Psi_tc}
For any $c \in \R$ and $0<a \le b < \infty$, on \RS $\EE^{\ssup 2}_{t,a,b} \cap \EE^{\ssup 2}_{t,a,b,c}$ \eRS
the following holds:
\begin{equation}
\sup_{s \in [at, bt]} \Bigl|\,\,\sup_{z \neq Z_s} \Psi_{s,c}(z) - \Psi^{\ssup 2}_{s} \Bigr|   \le o(d_t b_t \epsilon_t), \label{e:bound_Psi_tc1}
\end{equation}
and
\begin{equation}
\sup_{s \in [at, bt]} \left|\Psi_{s,c}(Z_{s}) - \Psi^{\ssup 1}_{s} \right|  \le o(d_t b_t \epsilon_t). \label{e:bound_Psi_tc2}
\end{equation}
\end{lemma}
\begin{proof}
The inner supremum in~\eqref{e:bound_Psi_tc1} is attained at $Z^{\ssup 1}_{s,c}$ if $Z^{\ssup 1}_{s,c} \neq Z_s$,
or $Z^{\ssup 2}_{s,c}$ if $Z^{\ssup 1}_{s,c} = Z_s$.
\RS Since $r_t f_t < |Z^{\ssup 1}_{s,c}| \vee |Z^{\ssup 2}_{s,c}| \vee |Z^{\ssup 2}_s| < r_t g_t$ on $\EE^{\ssup 2}_{t,a,b} \cap \EE^{\ssup 2}_{t,a,b,c}$, 
we can write
\begin{equation}\label{e:prbdPsitc1}
\begin{aligned}
-|c| \frac{r_t g_t}{a t} \le \Psi_{s,c}(Z^{\ssup 2}_s) - \Psi^{\ssup 2}_s 
& \le \sup_{z \neq Z_{s}} \Psi_{s,c}(z) - \Psi^{\ssup 2}_s \\
& \le \sup_{ r_t f_t < |z| < r_t g_t} \left\{ \Psi_{s,c}(z) - \Psi_s(z) \right\} < |c| \frac{r_t g_t}{ a t},
\end{aligned}
\end{equation}
\eRS so \eqref{e:bound_Psi_tc1} follows by by \eqref{e:def_fundam_scales} and \eqref{e:relation_scales}. 
The bound \eqref{e:bound_Psi_tc2} is obtained analogously.
\end{proof}

The proof of Proposition~\ref{prop:PPPconv} is based on a point process approach, which we describe next.
This approach will also allow us to prove Proposition~\ref{prop:stabilitygap} and Theorem~\ref{thm:aging_locus}.


\subsection{A point process approach}
\label{ss:PPapproach}\noindent
The key to the proofs of Proposition~\ref{prop:PPPconv} and Theorem~\ref{thm:aging_locus} is the convergence of the set
$\{(\lambda^{\scrC}(z), z)\colon z \in \scrC\}$ after suitable rescaling to (the support of) a Poisson point process.
We follow the setup and notation of \cite{R87} for point processes;
some arguments are for brevity relegated to the appendices.

Since we will need to apply the stated Poisson convergence to infer convergence of certain non-local minimizing functions, we will
need to compactify some sets of $\R \times \R^d$ as follows.
Embed $\R \times \R^d$ in a locally compact 
Polish space $\mathfrak{E}$ such that the set
\begin{equation}\label{e:defHH}
\HH^\theta_\eta := \left\{ (\lambda, z) \in \R\times \R^d \colon\, \lambda > \frac{|z|}{\theta} + \eta \right\} \subset \mathfrak{E}
\end{equation}
is relatively compact for any $\eta \in \R$ and $\theta \in (0,\infty)$ and,
for each  compact $K \subset \mathfrak{E}$, 
there exist $\theta >0, \eta \in \R$ such that $K \cap (\R \times \R^d) \subset \HH^\theta_\eta$.
A suitable choice of $\mathfrak{E}$ is given in Appendix~\ref{s:compactification}.
Note that a Poisson point process in $\R \times \R^d$ with intensity $\texte^{-\lambda} \textd\lambda \otimes \textd z$ 
can be extended to $\mathfrak{E}$ as the latter measure is a Radon measure on $\mathfrak{E}$. 
Denote by $\scrMp = \scrMp(\mathfrak{E})$ the set of point measures (i.e., $\N_0$-valued Radon measures) on
$\mathfrak{E}$. We equip $\scrMp$ 
with the topology of vague convergence,
and let $\supp(\PP)$ denote the support of $\PP \in \scrMp$.

Let us denote 
\begin{equation}\label{e:defPPt}
\PP_t := \sum_{z \in \scrC} \delta_{\left( Y_t(z),\,\, z/t \right)}
\quad
\text{where}
\quad 
Y_t(z) := \frac{\lambda^{\scrC}(z) - a_t}{d_t}.
\end{equation}
Our convergence result for $\PP_t$ reads as follows.
\begin{proposition}\label{prop:PPPconvNNt}
The point process $\PP_t$ defined in~\eqref{e:defPPt} belongs almost surely to $\scrMp$, 
and converges in distribution as $t \to \infty$ with respect to the vague topology of $\scrMp$
to a Poisson point process supported in $\R \times \R^d \subset \mathfrak{E}$ 
with intensity measure $\texte^{-\lambda} \textd\lambda \otimes \textd z$.
\end{proposition}
The proof of the Proposition~\ref{prop:PPPconvNNt} relies on the following lemma.
\begin{lemma}
\label{l:PPPconvcompact}
Let $\mu$ be a Radon measure on $\R$ such that $\mu \otimes \text{d} z$ is a Radon measure on $\mathfrak{E}$.
Let $\widehat{N}_t \in \N_0$ such that $\widehat{N}_t \ll t$ as $t \to \infty$, 
and assume that, for each $t >0$,
$(\widehat{Y}_t(z))_{z \in (2 \widehat{N}_t+1) \Z^d}$ is a collection of i.i.d.\ 
real-valued random variables satisfying the following two conditions:
\begin{enumerate}
\item[(i)] For each $s \in \R$,
\begin{equation}\label{e:PPPconvcompact_cond1}
\lim_{t \to \infty} \frac{t^d}{(2 \widehat{N}_t)^d} \textnormal{Prob} \left( \widehat{Y}_t(0) > s \right) = \mu(s, \infty).
\end{equation}
\item[(ii)] For each $\theta > 0$, $\eta \in \R$,
\begin{equation}\label{e:PPPconvcompact_cond2}
\lim_{n \to \infty} \limsup_{t \to \infty} \sum_{x \in (2\widehat{N}_t+1)\Z^d \colon |x| \ge t n } \textnormal{Prob} \left( \widehat{Y}_t(0) > \frac{|x|}{\theta t} + \eta \right) = 0.
\end{equation}
\end{enumerate}
Then, for each $t>0$ large enough,  the point process 
\begin{equation}
\label{e:defhatNNt}
\widehat{\PP}_t := \sum_{x \in (2\widehat{N}_t+1)\Z^d} \delta_{\left(\widehat{Y}_t(x), \,\,x/t \right)}
\end{equation}
belongs almost surely to $\scrMp$, and converges in distribution as $t \to \infty$ 
with respect to the vague topology of $\scrMp$
to a Poisson point process in $\R \times \R^d \subset \mathfrak{E}$
with intensity measure $ \mu \otimes \textd z $.
\end{lemma}
\begin{proof}
Note first that, by~\eqref{e:PPPconvcompact_cond2}, 
when $t$ is large enough, the expected value of $\widehat{\PP}_t(\HH^\theta_\eta)$ is finite for all $\theta>0, \eta \in \R$, 
and hence $\widehat{\PP}_t \in \scrMp$.
The claimed convergence may be proved by
a straightforward generalization of Proposition 3.21 of \cite{R87},
with $[0,\infty)$ therein substituted by $\R^d$ and $E$ therein substituted by $\R$
(see also \cite[Lemma~2.4]{HMS08}).
Indeed, we only need to verify (3.20) and (3.21) in \cite{R87}.
For (3.21), we note that, for any compact $K \subset \mathfrak{E}$, 
there exists $\eta \in \R$ such that $K \cap (\R \times \R^d) \subset [\eta, \infty) \times \R^d$,
and thus (3.21) follows from~\eqref{e:PPPconvcompact_cond1}. 
For (3.20), it suffices to prove that
\begin{equation}\label{e:convmeasure}
\sum_{x \in (2\widehat{N}_t+1)\Z^d} \textnormal{Prob} \left( \widehat{Y}_t(0) \in \cdot \right) \otimes \delta_{x/t}(\textd z) \,\,\underset{t \to \infty}\longrightarrow\,\,
\mu \otimes \textd z \; \text{ vaguely in } \scrMp.
\end{equation}
Indeed, by~\eqref{e:PPPconvcompact_cond1}, 
the convergence in~\eqref{e:convmeasure} holds 
when evaluated on functions with support contained
in the closure of a set of the form $[-n, \infty) \times [-n, n]^d \subset \mathfrak{E}$ with $n \in \N$.
This is extended to functions compactly supported in $\mathfrak{E}$ by
applying~\eqref{e:PPPconvcompact_cond2} and the fact that, for any compact $K \subset \mathfrak{E}$,
there exists $\theta > 0$, $\eta \in \R$ such that $K \cap \R \times \R^d \subset \HH^\theta_\eta$.
\end{proof}

We can now proceed to:
\begin{proof}[Proof of Proposition~\ref{prop:PPPconvNNt}]
We will first apply Lemma~\ref{l:PPPconvcompact} to an auxiliary process. Let
\begin{equation}\label{e:defhatYt}
\widehat{Y}_t(x) := \frac{\lambda^{\ssup 1}_{B_{\widehat{N}_t}(x)}- a_t}{d_t}, \quad x \in (2 \widehat{N}_t+1) \Z^d,
\end{equation}
and let $\widehat{\PP}_t$ be defined as in~\eqref{e:defhatNNt}.
\RV
Note that $\widehat{Y}_t(x)$, $x \in (2 \widehat{N}_t+1) \Z^d$, are i.i.d.\
since the corresponding boxes are disjoint.
\eRV
We claim the following:
\begin{equation}\label{e:PPPconvhatNNt}
\text{The statement of Proposition~\ref{prop:PPPconvNNt} holds for } \widehat{\PP}_t \text{ in place of } \PP_t.
\end{equation}
Indeed, condition~\eqref{e:PPPconvcompact_cond1} follows from~\eqref{e:maxorderlambda},
while~\eqref{e:PPPconvcompact_cond2} is proved in Appendix~\ref{s:tailestimate}.

Arguing as in the proof of Proposition~\ref{prop:welldefined}, 
we see that, almost surely, $\PP_t \in \scrMp$ for all large enough $t$. 
By~\eqref{e:PPPconvhatNNt} and since both $\PP_t$ and $\widehat{\PP}_t$ are simple, 
it suffices to show that, for any $\theta \in (0,\infty)$ and $\eta \in \R$, 
with probability tending to $1$ as $t \to \infty$ there exists a bijection
\begin{equation}\label{e:prconvPPPNNt1}
T_t \colon \supp(\widehat{\PP}_t) \cap \HH^\theta_\eta \to \supp(\PP_t) \cap \HH^\theta_\eta
\end{equation}
such that
\begin{equation}\label{e:prconvPPPNNt2}
\sup_{\Xi \in \supp(\widehat{\PP}_t) \cap \HH^\theta_\eta} \dist\left(T_t(\Xi),\Xi \right) 
\,\,\underset{t \to \infty}\longrightarrow\,\,0 \;\; \text{ in probability}.
\end{equation}
To that end, pick $x \in (2\widehat{N}_t+1)\Z^d$ such that $(\widehat{Y}_t(x), x/t) \in \HH^\theta_\eta$. 
We first claim that, a.s.\ eventually as $t \to \infty$,
all such $x$ satisfy
\begin{equation}\label{e:prconvPPPNNt2.1}
B_{\widehat{N}_t}(x) \subset B_{L^*_t} \qquad \text{ and } \qquad \lambda^{\ssup 1}_{B_{\widehat{N}_t}(x)} > \widehat{a}_{L^*_t} - \chi + o(1).
\end{equation}
Indeed, the second claim above follows from~\eqref{e:prophata}.
If the first were violated,
then by~\eqref{e:monot_princev}, Lemma~\ref{l:maxpotential}
and the fact that $s \mapsto 2\rho (d_t)^{-1} \ln_2 s - s /( \theta t)$ 
is decreasing for $s \ge 2 d \theta t \ln t$,
we would have, a.s.\ eventually as $t \to \infty$,
\begin{equation}\label{e:prconvPPPNNt2.2}
\frac{\lambda^{\ssup 1}_{B_{\widehat{N}_t}(x)} - a_t}{d_t} - \frac{|x|}{\theta t} 
\le \frac{2 \rho \ln_2 |x|}{d_t} - \frac{|x|}{\theta t} 
\le \frac{2 \rho \ln_2 L^*_t}{d_t} - \frac{L^*_t - \widehat{N}_t}{\theta t} 
\,\underset{t \to \infty}\longrightarrow -\infty
\end{equation}
by~\eqref{e:growthL*t}, contradicting $(\widehat{Y}_t(x), x/t) \in \HH^\theta_\eta$.
This finishes the proof of~\eqref{e:prconvPPPNNt2.1}.
Now, since $\widehat{N}_t = N_{L^*_t}$, by Lemmas~\ref{l:properties_opt_comps} and~\ref{l:comparisoncapitalsislands}
there exists, with probability tending to $1$ as $t\to \infty$, a unique $z \in \scrC$ satisfying
\begin{equation}\label{e:prconvPPPNNt2.5}
B_{\varrho_{z}}(z) \subset B_{\widehat{N}_t}(x) \;\; \text{ and } \;\; 
\lambda^{\ssup 1}_{B_{\widehat{N}_t}(x)} - \lambda^{\scrC}(z) \le 2 \texte^{-c_1 (\ln L^*_t)^{\kappa/2}},
\end{equation}
which allows us to define an injective map
\begin{equation}\label{e:prconvPPPNNt3}
T_t\left(\widehat{Y}_t(x), \frac{x}{t} \right) := \left(Y_t(z), \frac{z}{t} \right) 
\in \supp(\PP_t). 
\end{equation}
Let us verify that $T_t$ satisfies the desired properties. Indeed, 
\eqref{e:prconvPPPNNt2} follows since
\begin{equation}
\label{e:prconvPPPNNt4}
\left|\widehat{Y}_t(x) - Y_t(z) \right| + \left|\frac{z - x}{\theta t} \right| 
\le \frac{2 \texte^{-c_1 (\ln L^*_t)^{\kappa/2}}}{d_t} + d \frac{\widehat{N}_t}{ \theta t} =: \varepsilon_t \to 0 \text{ as } t \to \infty,
\end{equation}
and thus we only need to show that, with probability tending to $1$ as $t \to \infty$, 
\eqref{e:prconvPPPNNt3} is in $\HH^\theta_\eta$ and $T_t$ is surjective.
Indeed, by~\eqref{e:PPPconvhatNNt}, with probability tending to $1$ as $t \to \infty$,
\begin{equation}
\label{e:prconvPPPNNt5}
\widehat{\PP}_t\left(\HH^\theta_{\eta-\varepsilon_t} \setminus \HH^\theta_{\eta+\varepsilon_t}\right) = 0,
\end{equation}
implying by~\eqref{e:prconvPPPNNt4} that~\eqref{e:prconvPPPNNt3} is in $\HH^{\theta}_\eta$.
Moreover, if $(Y_t(z), z/t) \in \HH^{\theta}_\eta$ for some $z \in \scrC$,
then as before $\lambda^{\scrC}(z) > \widehat{a}_{L^*_t} - \chi + o(1)$ 
and $B_{\varrho_z}(z) \subset B_{L^*_t}$.
Thus, by Lemmas~\ref{l:comparisoncapitalsislands} and~\ref{l:properties_opt_comps},
there exists $x \in (2 \widehat{N}_t + 1) \Z^d$ such that~\eqref{e:prconvPPPNNt2.5} and~\eqref{e:prconvPPPNNt4} hold, 
implying by~\eqref{e:prconvPPPNNt5} that $(Y_t(z), z/t)$ is the image by $T_t$ of a point in $\supp(\widehat{\PP}_t) \cap \HH^\theta_\eta$.
This finishes the proof.
\end{proof}


\subsection{Order statistics: proof of Propositions~\ref{prop:PPPconv} and~\ref{prop:stabilitygap} and Theorem~\ref{thm:aging_locus}}
\label{ss:order_stat_pen_func}\noindent
Our next task is to translate \twoeqref{e:defPsik}{e:defZk} (and generalizations thereof) in terms of maps defined on point measures.
We start with some necessary notation.

Denote by $\widehat{\scrM}_{\text{P}}$ the set of measures $\PP$ on $\R \times \R^d$
that can be represented as 
\begin{equation}\label{e:reprPP}
\PP = \sum_{i \in \II} \delta_{(\lambda_i, z_i)} \;\;\; \text{ for some } \; \II \subset \N \text{ and } (\lambda_i, z_i) \in \R \times \R^d,
\end{equation}
i.e., $\widehat{\scrM}_{\text{P}}$ is the set of
$\N_0$-valued $\sigma$-finite Borel measures on $\R \times \R^d$.

Fix a measurable map \RV $\vartheta\colon\R \times \R^d \to \R^d$. 
To prove our main results, we will only need to consider $\vartheta$ independent of the first coordinate, 
but we keep the setup here more general for possible future applications.
\eRV
For a measure $\PP \in \widehat{\scrM}_{\text{P}}$ as in~\eqref{e:reprPP}, 
we define
\begin{equation}\label{e:defPPvartheta}
\PP^\vartheta := \sum_{i \in \II} \delta_{(\lambda_i, \vartheta(\lambda_i, z_i))},
\end{equation}
and we set
\begin{equation}\label{e:defscrMPTT}
\scrM_{\textnormal{P}, \vartheta} := \{ \PP \in \widehat{\scrM}_{\textnormal{P}} \colon\, \PP^\vartheta \in \scrMp \}.
\end{equation}
Finally, we generalise~\eqref{e:defpsitheta} by setting, for $\theta > 0$,
\begin{equation}\label{e:defpsivarthetatheta}
\psi^\vartheta_\theta (\lambda, z) := \lambda - \frac{|\vartheta(\lambda, z)|}{\theta}, \qquad (\lambda, z) \in \R \times \R^d.
\end{equation}
Now, for $\PP \in \scrM_{\textnormal{P}, \vartheta}$ and $\theta > 0$,
we set, recursively for $i \in \N$, $i \le |\supp(\PP)|$,
\begin{equation}
\begin{aligned}
& \Psi^{\ssup i}_\vartheta(\PP)(\theta) := \\
& \quad \;\;\;\; \sup \left\{ \psi^\vartheta_\theta(\lambda, z) \colon\, (\lambda, z) \in \supp(\PP) \setminus \left\{\Xi^{\ssup 1}_\vartheta(\PP)(\theta), \ldots, \Xi^{\ssup {i-1}}_\vartheta(\PP)(\theta) \right\} \right\}, \label{e:defPsiNi}
\end{aligned}
\end{equation}
\begin{equation}
\begin{aligned}
& \mathfrak{S}^{\ssup i}_\vartheta(\PP)(\theta) := \\
& \quad \left\{ (\lambda, z) \in \supp(\PP) \setminus \left\{\Xi^{\ssup 1}_\vartheta(\PP)(\theta), \ldots, \Xi^{\ssup {i-1}}_\vartheta(\PP)(\theta) \right\} \colon\, \psi^\vartheta_\theta(\lambda, z) = \Psi^{\ssup i}_\vartheta(\PP) (\theta) \right\}
\end{aligned}
\end{equation}
and
\begin{equation}
\Xi^{\ssup i}_\vartheta(\PP)(\theta) \in \, \left\{ (\lambda, z) \in \mathfrak{S}^{\ssup i}_\vartheta(\PP)(\theta) \colon\, (\lambda, z) \succeq (\lambda', z') \,\forall\, (\lambda', z') \in \mathfrak{S}^{\ssup i}_\vartheta(\PP)(\theta) \right\}, \label{e:defXiNi}
\end{equation}
where $\succeq$ is the usual lexicographical order of $\R \times \R^d$ as introduced right before~\eqref{e:defPsik}.
Note that $\Xi^{\ssup i}_\vartheta(\PP)$ is well defined since the set in~\eqref{e:defXiNi} has cardinality $1$.
We put
\begin{equation}\label{e:defPhi}
\left( \Lambda^{\ssup i}_\vartheta(\PP), Z^{\ssup i}_\vartheta(\PP) \right) := \Xi^{\ssup i}_{\vartheta}(\PP) \quad \text{ and } \quad \Phi^{\ssup i}_\vartheta(\PP) := \left(\Psi^{\ssup i}_\vartheta(\PP), \Lambda^{\ssup i}_{\vartheta}(\PP), Z^{\ssup i}_\vartheta(\PP) \right).
\end{equation}
\RV
In the case $\vartheta(\lambda, z) = z$ for all $(\lambda,z) \in \R \times \R^d$, we omit $\vartheta$ from the notation.

\smallskip
As functions of $\theta$, the objects defined above enjoy the following properties: \eRV
\begin{lemma}
\label{l:propPhi}
For any $\vartheta:\R \times \R^d \to \R^d$ and any $\PP \in \scrM_{\textnormal{P}, \vartheta}$, the following hold:
\begin{enumerate}
\item[(i)] $\Psi^{\ssup 1}_{\vartheta}(\PP)$, $\Lambda^{\ssup 1}_\vartheta(\PP)$ and \RV $|\vartheta(\Xi^{\ssup 1}_\vartheta(\PP))|$ \eRV 
are non-decreasing \RV in $\theta$. \eRV
Moreover, if $\theta_0 < \theta_1$ and $\Xi^{\ssup 1}_\vartheta(\PP)(\theta_0) \neq \Xi^{\ssup 1}_\vartheta(\PP)(\theta_1)$,
then they are strictly smaller at $\theta_0$ than at $\theta_1$.

\item[(ii)] For any $a \in (0,\infty)$ and any $i \in \N$, $i \le |\supp(\PP)|$,
\begin{equation}\label{e:propPhi1}
\Psi^{\ssup i}_\vartheta(\PP) \in \CC([a,\infty), \R) \;\;\; \textnormal{ and } \;\;\; \Xi^{\ssup i}_\vartheta(\PP) \in \DD([a,\infty), \R \times \R^d).
\end{equation}
The set of discontinuities of $\Xi^{\ssup i}_\vartheta(\PP)$ is discrete and, 
if $\supp(\PP^{\vartheta}) \cap (\R \times \{0\}) = \emptyset$,
then $\Psi^{\ssup 1}_\vartheta(\PP)$ is strictly increasing.
\end{enumerate}
\end{lemma}
The proof of Lemma~\ref{l:propPhi} is postponed to Appendix~\ref{s:propertiescost}.
It already implies the properties claimed for $\Psi^{\ssup k}_t, Z^{\ssup k}_t$
at the end of Section~\ref{ss:concentration_locus}:
indeed, they follow from the representation
\begin{equation}\label{e:reprZt}
(\Psi^{\ssup k}_t, \lambda^{\scrC}(Z^{\ssup k}_t), Z^{\ssup k}_t) = \Phi^{\ssup k}_\vartheta(\mathcal{P}_{\scrC})(t), \;\;\;  \vartheta(\lambda, z) := z \ln_3^+ |z|, \; \mathcal{P}_{\scrC} := \sum_{z \in \scrC} \delta_{\left(\lambda^{\scrC}(z), \, z \right)}.
\end{equation}
Note that $\PP_\scrC \in \scrM_{\text{P}, \vartheta}$ a.s.\ by~\eqref{e:welldefined}, and that $|\vartheta(\lambda_1, z_1)| > |\vartheta(\lambda_0, z_0)|$ implies $|z_1| > |z_0|$.

Next we consider continuity of $\PP \mapsto \Phi^{\ssup i}(\PP)$ with respect to the Skorohod topology,
i.e., specializing to the case $\vartheta(\lambda, z) = z$.
To this end, we define the following subsets of $\scrMp$, indexed by $a \in (0, \infty)$:
\begin{equation}\label{e:defscrM*}
\begin{aligned}
\widetilde{\scrM}^{a}_\textnormal{P} := \Bigl\{ \PP \in \scrMp \colon\, & \supp(\PP) \subset \R \times \R^d \setminus \left( \R \times \{0\}\right), \\ 
& (\lambda, z) \mapsto \lambda \text{ is injective over } \supp(\PP), \\
& \PP(\partial \HH^\theta_\eta)\le 1 \; \forall  \theta \in \{a\} \cup \big( (0, \infty) \cap \Q \big), \eta \in \R, \\  
& \PP(\partial \HH^\theta_\eta) \le 2 \; \forall  \theta \in (0,\infty), \eta \in \R, \\
& |\{\eta \in \R \colon\, \PP(\partial \HH^\theta_\eta) = 2\}| \le 1 \;\forall\theta \in (0,\infty)\Bigr\}.
\end{aligned}
\end{equation}
Then we have:
\begin{lemma}\label{l:continuityPhiskorohod}
Fix $a \in (0,\infty)$ and $\PP \in \widetilde{\scrM}^{a}_{\textnormal{P}}$.
Let $\vartheta_t : \R \times \R^d \to \R^d$, $t>0$, satisfy
\begin{align}
\text{(i) } & \vartheta_t(\lambda, z) \underset{t \to \infty}\longrightarrow z \text{ locally uniformly for } (\lambda, z) \in \R \times \left(\R^d \setminus \{0\}\right), \text{ and } \label{e:propvarthetat1}\\
\text{(ii) } & \exists \,  c_* > 0 \text{ such that, for all } \eta \in \R \text{ and } \delta > 0, \; \liminf_{t \to \infty} \inf_{\lambda \ge \eta, |z|\ge \delta} \frac{|\vartheta_t(\lambda, z)|}{|z|} \ge c_*. \label{e:propvarthetat2}
\end{align}
Let $\PP_t \in \scrMp \cap \scrM_{\textnormal{P}, \vartheta_t}$ such that $\PP_t \underset{t \to \infty}\longrightarrow \PP$ vaguely in $\scrMp$.
Then also $\PP_t^{\vartheta_t} \to \PP$ vaguely and, 
for all $k \in \N$, $k \le |\supp(\PP)|$,
\begin{equation}\label{e:propPhi2}
\left(\Phi^{\ssup i}_{\vartheta_t}(\PP_t) \right)_{1 \le i \le k} \,\underset{t \to \infty}\longrightarrow\, 
\left(\Phi^{\ssup i}(\PP)\right)_{1 \le i \le k}
\end{equation}
in the Skorohod topology of $\DD([a, \infty), (\R \times \R \times \R^d)^k)$.
In particular, $(\Phi^{\ssup i})_{1 \le i \le k}$ is continuous at $\PP$ with respect to the Skorohod topology.
\end{lemma}
Lemma~\ref{l:continuityPhiskorohod} will be also proved in Appendix~\ref{s:propertiescost}.
We now use it to finish:
\begin{proof}[Proof of Proposition~\ref{prop:PPPconv}]
By Lemma~\ref{l:propPhi}, we may realise the processes in~\eqref{e:limitPPP} as
\begin{equation}\label{e:repr_barPsiZ}
\left(\overline{\Psi}^{\ssup i}_\theta, \overline{\Lambda}^{\ssup i}_\theta, \overline{Z}^{\ssup i}_\theta \right) 
= \Phi^{\ssup i}(\PP_\infty)(\theta)
\end{equation}
where 
$\PP_\infty$ is a Poisson point process on $\R \times \R^d$ with intensity $\texte^{-\lambda} \textd\lambda \otimes \textd z$.
Note that, for each~$a>0$, $\PP_\infty \in \widetilde{\scrM}^a_\textnormal{P}$ almost surely.
On the other hand, we also have the representation
\begin{equation}\label{e:reprPsiPP}
\left(\frac{\Psi^{\ssup i}_{\theta t,c}-a_{r_t}}{d_{r_t}}, \frac{\lambda^{\scrC}(Z^{\ssup i}_{\theta t, c}) - a_{r_t}}{d_{r_t}}, \frac{Z^{\ssup i}_{\theta t, c}}{r_t} \right) = 
\Phi^{\ssup i}_{\vartheta_t} \left( \PP_{r_t}\right)(\theta)
\end{equation}
where $\PP_t$ is as in~\eqref{e:defPPt} and
\begin{equation}
\label{e:defvarthetat}
\vartheta_t(\lambda, z) :=  z \left(\frac{\ln_3^+|r_t z| - c}{\ln_3 t} \right)^+ \frac{d_t}{d_{r_t}}.
\end{equation}
Note that, by~\eqref{e:welldefgenPsi}, $\PP_{r_t} \in \scrM_{\textnormal{P}, \vartheta_t}$ almost surely for all $t$ large enough.
The convergence claimed in Proposition~\ref{prop:PPPconv}
now follows by Proposition~\ref{prop:PPPconvNNt} and Lemma~\ref{l:continuityPhiskorohod} together with~\eqref{e:repr_barPsiZ},\twoeqref{e:reprPsiPP}{e:defvarthetat} and the Skorohod representation theorem; in fact,
\begin{equation}\label{e:jointconvergence}
\left(\PP_{r_t}, \left(\Phi^{\ssup i}_{\vartheta_t}(\PP_{r_t})(\theta) \right)_{\theta \in [a, \infty), 1 \le i \le k} \right) 
\underset{t \to \infty}{\overset{\text{law}}\longrightarrow}
\left(\PP_\infty, \left(\Phi^{\ssup i}(\PP_\infty)(\theta) \right)_{\theta \in [a, \infty), 1 \le i \le k} \right).
\end{equation}
The statement regarding $\EE^{\ssup k}_{a,b,c}$ follows from the distributional convergence since $d_{r_t} = d_t (1+o(1))$ and,
by the continuity properties of $\overline{\Psi}^{\ssup i}_\theta$ and $\overline{Z}^{\ssup i}_\theta$,
\begin{equation}\label{e:prPPPconvaux1}
\begin{aligned}
-\infty < & \inf_{ \theta \in [a,b]} \overline{\Psi}^{\ssup i}_\theta \le \sup_{ \theta \in [a,b]} \overline{\Psi}^{\ssup i}_\theta < \infty, \qquad \;\;
0 < \inf_{ \theta \in [a,b]} |\overline{Z}^{\ssup i}_\theta| \le \sup_{ \theta \in [a,b]} |\overline{Z}^{\ssup i}_\theta| < \infty \\
\text{ and } \;\; & \left( \overline{\Psi}^{\ssup i}_a - \overline{\Psi}^{\ssup {i+1}}_a \right) \wedge \left( \overline{\Psi}^{\ssup i}_b - \overline{\Psi}^{\ssup {i+1}}_b \right) > 0
\end{aligned}
\end{equation}
hold almost surely for each $i \in \N$.
The expression for the density in~\eqref{e:limitingdensity1ddist}
follows from an analogous computation as performed in the proof of Proposition~3.2 in \cite{ST14}.
\end{proof}

Next we interpret the event in Theorem~\ref{thm:aging_locus} in terms of the underlying point measure,
which is still kept rather general:
\begin{lemma}
\label{l:nojumpintermsofPP}
For any $\vartheta:\R \times \R^d\to\R^d$, any $\PP \in \scrM_{\textnormal{P}, \vartheta}$ and any $0 < a < b < \infty$,
the following statements are equivalent:
\begin{align}
&
\begin{aligned}
\text{(1)} \quad & \Lambda^{\ssup 1}_\vartheta(\PP)(a) = \Lambda^{\ssup 1}_\vartheta(\PP)(b); \\
\text{(2)} \quad & \Xi^{\ssup 1}_\vartheta(\PP)(\theta) = \Xi^{\ssup 1}_\vartheta(\PP)(a) \, \text{ for all } \theta \in [a, b]; \\
\text{(3)} \quad & 
\PP \left\{ (\lambda, z) \colon
\begin{array}{l}
\psi^\vartheta_b(\lambda,z) > \psi^\vartheta_b(\Xi^{\ssup 1}_\vartheta(\PP)(a)), \text{ or}\\
\psi^\vartheta_b(\lambda,z) = \psi^\vartheta_b(\Xi^{\ssup 1}_\vartheta(\PP)(a)) \text{ and }
\lambda > \Lambda^{\ssup 1}_\vartheta(\PP)(a)
\end{array} \right\} = 0.\\
\end{aligned}
\label{e:nojumpintermsofPP}\\
& \hspace{-28pt} \text{\RV If $\vartheta$ does not depend on $\lambda$, then (1)--(3) are also equivalent to: \eRV} \nonumber\\
& \, \text{(4)} \quad Z^{\ssup 1}_\vartheta(\PP)(a) = Z^{\ssup 1}_\vartheta(\PP)(b). \label{e:nojumpsPPspecial}
\end{align}
\end{lemma}
\begin{proof}
The implication $(1) \Rightarrow (2)$ follows from Lemma~\ref{l:propPhi}(i),
and $(2) \Rightarrow (3) \Rightarrow (1)$ 
are easily verified using the definition of $\Xi^{\ssup i}_\vartheta$.
It is clear that $(2) \Rightarrow (4)$ and,
when $\vartheta$ does not depend on $\lambda$,
 $(4) \Rightarrow (1)$ also follows from Lemma~\ref{l:propPhi}(i).
\end{proof}

\RV
The last equivalence in Lemma~\ref{l:nojumpintermsofPP} can be extended to the setup of Lemma~\ref{l:continuityPhiskorohod}:
\begin{lemma}\label{l:nojumpsPPclosetoID}
Let $a \in (0, \infty)$, $\PP$,  $\vartheta_t : \R \times \R^d \to \R^d$ and $\PP_t$ as in Lemma~\ref{l:continuityPhiskorohod}. 
Then, for all $b \in (a, \infty)$ and all large enough $t$,
\eqref{e:nojumpsPPspecial} is equivalent to $(1)$--$(3)$ in \eqref{e:nojumpintermsofPP}
with $\vartheta = \vartheta_t$, $\PP = \PP_t$.
\end{lemma}
Lemma~\ref{l:nojumpsPPclosetoID} will be proved in Appendix~\ref{s:propertiescost}.
\eRV

We study next continuity properties of the event in Lemma~\ref{l:nojumpintermsofPP}$(3)$.
To this end, we define, for $\vartheta:\R \times \R^d\to\R^d$, $\PP \in \scrM_{\textnormal{P}, \vartheta}$, 
$(\lambda, z) \in \R \times \R^d$ and $\theta>0$,
\begin{equation}\label{e:defnumbercandidatesjump}
\FF^{\vartheta}_\theta(\PP, \lambda, z) 
:=\, \PP \left\{ (\lambda', z') \colon 
\begin{array}{l}
\psi^{\vartheta}_\theta(\lambda', z') > \psi^{\vartheta}_\theta(\lambda, z), \text{ or}\\
\psi^{\vartheta}_\theta(\lambda', z') = \psi^{\vartheta}_\theta(\lambda, z) \text{ and } \lambda' > \lambda
\end{array} \right\} \in \N_0.
\end{equation}
When $\vartheta(\lambda, z) = z$, we again omit it from the notation.
Then we have:
\begin{lemma}\label{l:convergencenojumps}
Fix $a \in (0,\infty)$ and take $\PP$, $\vartheta_t$ and $\PP_t$ as in Lemma~\ref{l:continuityPhiskorohod}.
Assume that $(\lambda_*, z_*) \in \supp(\PP)$, $(\lambda_t, z_t) \in \supp(\PP_t)$ 
are such that $(\lambda_t, z_t) \to (\lambda_*, z_*)$ as $t\to\infty$.
Then
\begin{equation}\label{e:convergencenojumps}
\FF^{\vartheta_t}_a \left( \PP_t, \lambda_t, z_t \right) \underset{t\to\infty}\longrightarrow \FF_a(\PP, \lambda_*, z_*).
\end{equation}
\end{lemma}
\noindent
The proof of Lemma~\ref{l:convergencenojumps} is once more deferred to Appendix~\ref{s:propertiescost}.
Together with Lemma~\ref{l:nojumpintermsofPP}, it permits us to give:
\begin{proof}[Proof of Theorem~\ref{thm:aging_locus}]
Fix $0<a<b<\infty$ and use the representation \twoeqref{e:reprPsiPP}{e:defvarthetat} (with $c=0$), 
Lemma~\ref{l:nojumpintermsofPP} and~\eqref{e:defnumbercandidatesjump} to write \RV (note that $\vartheta_t$ in \eqref{e:defvarthetat} does not depend on $\lambda$) \eRV
\begin{align}\label{e:prthmaginglocus1}
Z_{at} = Z_{bt} & \Leftrightarrow Z_{\theta t} = Z_{at} \; \forall \; \theta \in [a,b] \nonumber\\
& \Leftrightarrow \FF^{\vartheta_t}_b \left( \PP_{r_t}, \Lambda^{\ssup 1}_{\vartheta_t}(\PP_{r_t})(a), Z^{\ssup 1}_{\vartheta_t}(\PP_{r_t})(a) \right) = 0.
\end{align}
Since $\PP_\infty \in \widetilde{\scrM}^a_{\textnormal{P}} \cap \widetilde{\scrM}^b_{\text{P}}$ a.s., 
the distributional convergence follows from Lemma~\ref{l:convergencenojumps},~\eqref{e:jointconvergence} and~\eqref{e:repr_barPsiZ}.
\RS To show \eqref{e:tailTheta}, 
fix $u>1$ and let \eRS
\begin{equation}\label{e:prtailTheta1}
D_u(\lambda, z) := \left\{ (\lambda', z') \in \R \times \R^d \colon\, \lambda' - \lambda < |z'| - |z| < u(\lambda'-\lambda)\right\}.
\end{equation}
Note that, by the definition of $\overline{\Xi}_1 = \Xi^{\ssup 1}(\PP_\infty)(1)$
and the fact that $\PP_\infty \in \widetilde{\mathscr{M}}^1_{\text{P}} \cap \widetilde{\mathscr{M}}^u_{\text{P}}$ almost surely,
$\FF_u(\PP_\infty, \overline{\Xi}_1) = \PP_\infty (D_u(\overline{\Xi}_1))$ almost surely.
Moreover, conditionally given $\overline{\Xi}_1 = \Xi$, $D_u(\Xi)$ is independent of $\overline{\Xi}_1$,
and thus by Lemma~\ref{l:nojumpintermsofPP},
\begin{equation}\label{e:prtailTheta2}
\Prob(\Theta > u-1) = E [\exp\{-\mu(D_u(\overline{\Xi}_1)\}]
\end{equation}
where $\mu = \texte^{-\lambda} \textd \lambda \otimes \textd z$ and $``E''$ denotes expectation under the law of $\PP_\infty$.
We identify
\begin{equation}\label{e:prtailTheta3}
\mu(D_u(\lambda, z)) = (2u)^d \texte^{-\lambda} G_u(|z|) \;\; \text{ where } \;\; G_u(r) := 1 - \frac{1}{u^d} + \sum_{k=1}^{d-1} \frac{r^k}{k!} \left(\frac{1}{u^i} - \frac{1}{u^d} \right)
\end{equation}
by a straightforward computation, as well as
\begin{equation}\label{e:prtailTheta4}
\int_{\R^d} \frac{\textd z}{\texte^{|z|}/u^d + G_u(|z|)} \sim \int_{|z| < d \ln u} \frac{\textd z}{\texte^{|z|}/u^d + G_u(|z|)} \sim \frac{(2d \ln u)^d}{d!} \;\; \text{ as } \; u \to \infty.
\end{equation}
Then \eqref{e:tailTheta} follows by a computation using \eqref{e:prtailTheta2}--\eqref{e:prtailTheta4} and \eqref{e:limitingdensity1ddist}.
\eRS
\end{proof}

The last objective of the section is to prove Proposition~\ref{prop:stabilitygap}.
Our next lemma shows that its statement holds in fact more generally:
\begin{lemma}\label{l:gapstability}
For any $\vartheta:\R \times \R^d \to \R^d$, any $\PP \in \scrM_{\textnormal{P}, \vartheta}$ and any $0 < a < b < \infty$,
if
\begin{equation}\label{e:hyp_gapstability}
\Xi^{\ssup 1}_\vartheta(\PP)(\theta) = \Xi^{\ssup 1}_\vartheta(\PP)(a) \quad \forall \; \theta \in [a, b]
\end{equation}
then
\begin{equation}\label{e:gapstability}
\inf_{\theta \in [a, b]} \left\{\Psi^{\ssup 1}_\vartheta(\PP)(\theta) - \Psi^{\ssup 2}_\vartheta(\PP)(\theta) \right\}
= \min_{\theta \in \{a, b\}} \left\{ \Psi^{\ssup 1}_\vartheta(\PP)(\theta) - \Psi^{\ssup 2}_\vartheta(\PP)(\theta) \right\}.
\end{equation}
\end{lemma}
\begin{proof}
For $\theta \in [a, b]$ and $i \in \{1,2\}$, 
put $(\hat{\lambda}^{\ssup i}_\theta, \hat{z}^{\ssup i}_\theta) := \Xi^{\ssup i}_\vartheta(\PP)(\theta)$
and write
\begin{equation}\label{e:prlgapstab1}
\Psi^{\ssup 1}_\vartheta(\PP)(\theta) - \Psi^{\ssup 2}_\vartheta(\PP)(\theta) = \hat{\lambda}^{\ssup 1}_\theta - \hat{\lambda}^{\ssup 2}_\theta - \frac{|\vartheta(\hat{\lambda}^{\ssup 1}_\theta, \hat{z}^{\ssup 1}_\theta)| - |\vartheta(\hat{\lambda}^{\ssup 2}_\theta, \hat{z}^{\ssup 2}_\theta)|}{\theta}.
\end{equation}
If $|\vartheta(\hat{\lambda}^{\ssup 1}_\theta, \hat{z}^{\ssup 1}_\theta)| \ge |\vartheta(\hat{\lambda}^{\ssup 2}_\theta, \hat{z}^{\ssup 2}_\theta)|$, 
use $\theta^{-1} \le a^{-1}$ and~\eqref{e:hyp_gapstability} to obtain
\begin{align}\label{e:prlgapstab2}
\Psi^{\ssup 1}_\vartheta(\PP)(\theta) - \Psi^{\ssup 2}_\vartheta(\PP)(\theta) 
& \ge \Psi^{\ssup 1}_\vartheta(\PP)(a) - \psi^{\vartheta}_{a}(\hat{\lambda}^{\ssup 2}_\theta, \hat{z}^{\ssup 2}_\theta) \ge \Psi^{\ssup 1}_\vartheta(\PP)(a) - \Psi^{\ssup 2}_\vartheta(\PP)(a).
\end{align}
If $|\vartheta(\hat{\lambda}^{\ssup 1}_\theta, \hat{z}^{\ssup 1}_\theta)| < |\vartheta(\hat{\lambda}^{\ssup 2}_\theta, \hat{z}^{\ssup 2}_\theta)|$,
using $\theta^{-1} \ge b^{-1}$ instead we analogously get
\begin{align}\label{e:prlgapstab3}
\Psi^{\ssup 1}_\vartheta(\PP)(\theta) - \Psi^{\ssup 2}_\vartheta(\PP)(\theta) 
& \ge \Psi^{\ssup 1}_\vartheta(\PP)(b) - \Psi^{\ssup 2}_\vartheta(\PP)(b).
\end{align}
Now \eqref{e:gapstability} follows from \twoeqref{e:prlgapstab2}{e:prlgapstab3}.
\end{proof}

We can finally conclude:
\begin{proof}[Proof of Proposition~\ref{prop:stabilitygap}]
Follows from Lemmas~\ref{l:nojumpintermsofPP} and \ref{l:gapstability} together with~\eqref{e:reprZt}.
\end{proof}

\section{Mass decomposition}
\label{s:massdecomp}\nopagebreak\noindent
Here we prove Proposition~\ref{prop:lowerbound_forproof} (subsection~\ref{ss:lower_bound}),
Proposition~\ref{prop:macroboxtruncation_forproof} (subsection~\ref{ss:macrobox_truncation}),
Propositions~\ref{prop:randomtruncation_forproof}--\ref{prop:pathsavoidingZ_forproof} (subsection~\ref{ss:negligcontrib}),
and finish the proof of Theorem~\ref{thm:locus} (subsection~\ref{ss:upperboundUt}).

\subsection{Lower bound for the total mass}
\label{ss:lower_bound}\noindent
We begin with a lower bound for the mass up to the hitting time of a point.
\begin{lemma}
\label{l:lbinitialmasscontrib}
Under Assumption~\ref{A:integlowtailxi}, 
there exists a constant $K > 1$ such that, a.s.\ eventually as $\theta \to \infty$,
for all $x \in \Z^d$ with $|x| > 4d \theta$,
\begin{equation}\label{e:lbinitialmasscontrib}
\E_0 \left[ \texte^{\int_0^{\tau_x} \xi(X_u) \textd u} \1_{\{\tau_x \le \theta\}} \right] \ge \exp \left\{- |x| \ln \frac{K |x|}{\theta} \right\}.
\end{equation}
\end{lemma}

\begin{proof}
We follow the proof of Lemma 4.3 of \cite{GM90} (case of $d=1$ therein).
Fix a path $\pi$ from $0$ to $x$ such that $|\pi|=|x|$.
Then the left-hand side of~\eqref{e:lbinitialmasscontrib} is at least
\begin{equation}\label{e:prinitialmasscontrib1}
(2d)^{-|x|}\E_0 \left[ \exp\left\{-\sum_{i=0}^{|x|-1} \sigma_i \xi^{-}(\pi_i) \right\} \1_{\{\sum_{i=0}^{|x|-1} \sigma_i \le \theta\}} \right]
\end{equation}
where $(\sigma_i)_{i = 0}^{\infty}$ are i.i.d.\ exponential random variables with parameter $2d$. 
We can further bound~\eqref{e:prinitialmasscontrib1} from below by
\begin{align}\label{e:prinitialmasscontrib2}
& (2d)^{-|x|}\texte^{-\theta}\P_0 \left( \sigma_i \le \frac{\theta/|x|}{1+\xi^-(\pi_i)} \;\forall\; i = 0, \ldots, |x|-1 \right)  \nonumber\\
\ge \, & (2d)^{-|x|}\texte^{-\theta} \prod_{i=0}^{|x|-1} \frac{d \theta/|x|}{1+\xi^-(\pi_i)}
= \exp \left\{- |x| \ln \frac{2|x|}{\theta} -\theta - \sum_{i = 0}^{|x|- 1} \ln(1+\xi^-(\pi_i)) \right\}
\end{align}
where we used $|x|> 4d \theta$ and $1-\texte^{-2y}\ge y$ when $0 < y < \tfrac14$.
By Theorem 1.1 of \cite{M02} and Assumption~\ref{A:integlowtailxi}, there exists a constant $c_0>0$ such that,
 a.s.\ eventually as $|x| \to \infty$,
\begin{equation}\label{e:prinitialmasscontrib3}
\sum_{i = 0}^{|x|- 1} \ln(1+\xi^-(\pi_i)) \le c_0 |x|.
\end{equation}
Now~\eqref{e:lbinitialmasscontrib} follows from \eqref{e:prinitialmasscontrib3} and $\theta < |x|/(4d)$.
\end{proof}

We can now prove Proposition~\ref{prop:lowerbound_forproof}.
\begin{proof}[Proof of Proposition~\ref{prop:lowerbound_forproof}]
For a finite connected subset $\Lambda \subset \Z^d$, let $\phi^{\ssup 1}_\Lambda$ be the $\ell^2$-normalised eigenfunction of $H_\Lambda$ 
corresponding to its largest eigenvalue $\lambda^{\ssup 1}_\Lambda$ as in Section~\ref{ss:specbounds}.
Let $x_0 \in \Lambda$ be a point where $\phi^{\ssup 1}_\Lambda$ attains its maximum and note that, 
since $\| \phi^{\ssup 1}_\Lambda \|_{\ell^2(\Z^d)} = 1$, $|\phi^{\ssup 1}_\Lambda(x_0)|^2 \ge |\Lambda|^{-1}$.
By Lemma~\ref{l:bounds_mass}, for any $s>0$,
\begin{equation}\label{e:lb2}
\E_{x_0} \left[ \texte^{\int_0^s \xi(X_u) \textd u} \1_{\{ \tau_{\Lambda^\cc} > s\}} \right] \ge \texte^{s\lambda^{\ssup 1}_\Lambda} |\phi^{\ssup 1}_\Lambda(x_0)|^2 \ge \texte^{s\lambda^{\ssup 1}_\Lambda - \ln |\Lambda|}.
\end{equation}
Using the Feynman-Kac formula and the strong Markov property, we get, for any $\theta < s$,
\begin{align}\label{e:lb3}
U(s) 
& \ge \E_0 \left[ \exp\left\{\int_0^{\tau_{x_0}} \xi(X_u) \textd u \right\} \1_{\{\tau_{x_0} \le \theta \}} \E_{x_0} \left[ \texte^{\int_0^{s-r} \xi(X_u) \textd u} \1_{\{\tau_{\Lambda^\cc} > s-r \}} \right]_{r = \tau_{x_0}} \right] \nonumber\\
& \ge \texte^{s\lambda^{\ssup 1}_\Lambda -\ln|\Lambda| - \theta  |\lambda^{\ssup 1}_\Lambda|} \, \E_0 \left[ \exp\left\{\int_0^{\tau_{x_0}} \xi(X_u) \textd u  \right\} \1_{\{\tau_{x_0} \le \theta \}} \right].
\end{align}
Specializing now to $\Lambda := B_{\varrho_{Z_s}}(Z_s)$, 
let $K >1$ as in Lemma~\ref{l:lbinitialmasscontrib} and set 
$\theta := K |x_0|/\lambda^{\scrC}(Z_s)$.
By Proposition~\eqref{prop:PPPconv},
we may assume that $\EE_{t,a,b}$ (cf.\ \eqref{e:defgoodevents})
occurs, and by Lemma~\ref{l:maxpotential} also that $\varrho_{Z_s} \le \ln t$.
Thus
\begin{equation}\label{e:lb5}
\frac{|x_0|}{s} \le \frac{|Z_s|+|x_0 - Z_s|}{a t} \le \frac{r_t g_t + d \ln t}{a t} = o(d_t b_t \epsilon_t),
\end{equation}
while $\lambda^{\scrC}(Z_s) \ge \Psi_s^{\ssup 1} \ge a_{r_t} -d_t g_t \to \infty$ as $t \to \infty$ since $d_t g_t = o(1)$.
Therefore, $\theta < |x_0|/(4d) < s $ for large enough $t$. 
On the other hand,
\begin{equation}\label{e:lb6}
\lambda^{\scrC}(Z_s) \le \xi(Z_s) \le 2 \rho \ln_2 |Z_s| \le 2 \rho \ln_2 t 
\end{equation}
for large enough $t$ since $r_t g_t = o(t)$. Hence
\begin{equation}
\theta \ge \frac{r_t f_t - 2d \ln t}{2 \rho \ln_2 t} \to \infty \text{ as } t \to \infty,
\end{equation}
and so we may apply Lemma~\ref{l:lbinitialmasscontrib} to~\eqref{e:lb3} obtaining
\begin{equation}\label{e:lb7}
\frac{\ln U(s)}{s} \ge \lambda^{\scrC}(Z_s) - \frac{|x_0|}{s} \ln \lambda^{\scrC}(Z_s) -K \frac{ |x_0|}{s} + o(d_t b_t \epsilon_t).
\end{equation}
Now, by~\eqref{e:lb6}, 
\begin{equation}\label{e:lb8}
\frac{\ln U(s)}{s} \ge \Psi^{\ssup 1}_s - \frac{|x_0-Z_s| \ln^+_3 |Z_s|}{s} - \left(|\ln 2\rho|+K \right) \frac{|x_0|}{s} + o(d_t b_t \epsilon_t),
\end{equation}
and to conclude we note that the second and third terms in~\eqref{e:lb8} are also $o(d_t b_t \epsilon_t)$.
\end{proof}

\subsection{Macrobox truncation}
\label{ss:macrobox_truncation}\noindent
Next we prove Proposition~\ref{prop:macroboxtruncation_forproof}, 
ensuring that the Feynman-Kac formula is not affected by restricting to random walk paths that 
do not leave a box of side~$L_t = \lfloor t \ln^+_2 t \rfloor$ around the starting point.
\begin{proof}[Proof of Proposition~\ref{prop:macroboxtruncation_forproof}]
We follow the proof of Proposition 2.1 in \cite{FM14}. First write
\begin{equation}\label{e:approx1} 
\E_0 \left[ \texte^{\int_0^s \xi(X_u) \textd u} \1_{\{ \sup_{\theta \in [0,s]} |X_\theta| \ge L_t \}} \right]
 \le \sum_{n=L_t}^\infty \exp \left\{ s \max_{x \in B_n} \xi(x) \right\} \P_0 \left( \sup_{\theta \in [0,s]} |X_\theta| = n \right).
\end{equation}
Denoting by $J_s$ the number
of jumps of $X$ up to time~$s$, the fact that $J_s$ is 
a Poisson random variable with parameter $2d s$ gives
\begin{equation}\label{e:approx2}
\P_0 \left( \sup_{\theta \in [0,s]} |X_\theta| = n \right) \le \P_0 \left( J_s \ge n \right) \le \frac{(2ds)^n}{n!}.
\end{equation}
By Lemma~\ref{l:maxpotential}, $\max_{x \in B_n} \xi(x) \le 2 \rho \ln_2 n$ a.s.\ for all $n$ large enough.
By Stirling's formula and $s \le bt$, the $n$-th summand in~\eqref{e:approx1} is at most
\begin{equation}\label{e:approx3}
\exp \Big\{ 2 \rho b t \ln_2 n - n (\ln n-\ln t - c)\Big\}
\end{equation}
for some deterministic constant $c>0$. 
Now, when $n \ge L_t$ and $t$ is large enough, $\ln n - \ln t - c \ge \frac12 \ln_3 t$.
Since the function $x \mapsto 2\rho b t \ln_2 x -  \frac{x}{4} \ln_3 t$ is
strictly decreasing on $[L_t,\infty)$ and negative at $x = L_t$, a.s.\ for all $t$ large enough, 
\eqref{e:approx1} is smaller than
\begin{equation}\label{e:approx4}
\sum_{n=L_t}^\infty \texte^{-\frac{n}{4}\ln_3 t } \le 2 \texte^{-\frac{L_t}{4}\ln_3 t}.
\end{equation}
Plugging in the definition of~$L_t$ now yields~\eqref{e:macroboxtruncation_forproof}.
\end{proof}

\subsection{Negligible contributions}
\label{ss:negligcontrib}\noindent
In this subsection we prove Propositions~\ref{prop:randomtruncation_forproof} and \ref{prop:pathsavoidingZ_forproof}.
Here and in the next subsection we will work with $R_L$ satisfying \twoeqref{e:propertiesR_L}{e:addpropR_L}.
It will be useful to introduce yet another family of 
auxiliary cost functionals $\widetilde{\Psi}_{t,s,c}$, indexed by $t,s\ge0$, $c \in \R$, 
 and defined on the elements of $\mathfrak{C}_{L_t, A}$ as follows:
\begin{equation}\label{e:deftildePsi}
\widetilde{\Psi}_{t,s,c}(\CC) := \lambda^{\ssup 1}_\CC - \frac{(\ln^+_3 |z_\CC| - c)^+}{s}|z_\CC|, \;\;\; \CC \in \mathfrak{C}_{L_t, A}.
\end{equation}
These functionals will be convenient to express bounds to the Feynman-Kac formula obtained via Proposition~\ref{prop:massclass}.
In order to compare $\widetilde{\Psi}_{t,s,c}$ and $\Psi_t$, we will need the following.

\begin{lemma}\label{l:existoptisl}
Almost surely for all $t,s >0$,
there exists a component $\CC_{t,s} \in \mathfrak{C}_{L_t, A}$
such that, for all $0 < a \le b < \infty$, the following holds with probability tending to $1$ as $t \to \infty$:
\begin{equation}\label{e:existoptisl}
z_{\CC_{t,s}} = Z_s \quad \forall \, s \in [at, bt].
\end{equation}
\end{lemma}
\begin{proof}
By Lemma~\ref{l:comparisoncapitalsislands},
there exists a $\delta>0$ such that,
with probability tending to $1$ as $t \to \infty$,
whenever $|Z_s| + 2d \varrho_{Z_s} < L_t$ and $\lambda^{\scrC}(Z_s) > \widehat{a}_{L_t} - \chi - \delta$ 
we can find a unique $\CC_{t,s} \in \mathfrak{C}_{L_t, A}$ with $z_{\CC_{t,s}} = Z_s$.
Fixing $\CC^*_t \in \mathfrak{C}_{L_t, A}$ in an arbitrary (measurable) fashion,
we define $\CC_{t,s} = \CC^*_t$ when either the conclusion of Lemma~\ref{l:comparisoncapitalsislands} does not hold,
or when $Z_s$ does not satisfy the properties above.
By Proposition~\ref{prop:PPPconv}, $\CC_{t,s}$ satisfies~\eqref{e:existoptisl}
with probability tending to $1$ as $t \to \infty$.
\end{proof}
When $t = s$, we write $\CC_t$ instead of $\CC_{t,s}$.
The following lemma relates $\widetilde{\Psi}_{t,s,c}$ to $\Psi_t$.
\begin{lemma}
\label{l:compPsitildePsi}
For all $A>0$ large enough and any $0<a \le b < \infty$, $\delta > 0$ and $c \in \R$,
\begin{align}\label{e:compPsitildePsi}
\CC_{t,s} \in \mathfrak{C}^{\delta}_{L_t, A} \;\;\; \text{ and } \;\;\;
\left|\widetilde{\Psi}_{t,s,c} (\CC_{t,s}) - \Psi^{\ssup 1}_s \right| \vee \left| \max_{\CC \neq \CC_{t,s}}\widetilde{\Psi}_{t,s,c}(\CC) - \Psi^{\ssup 2}_s \right| \le  o(d_t b_t \epsilon_t)
\end{align}
hold for all $s \in [at, bt]$ with probability tending to $1$ as $t \to \infty$. 
\end{lemma}
\begin{proof}
Fix $A, \delta >0$ as in Lemma~\ref{l:comparisoncapitalsislands}.
By this lemma and Proposition~\ref{prop:PPPconv},
if $\CC \notin \mathfrak{C}^\delta_{L_t, A}$ then
$\widetilde{\Psi}_{t,s,c}(\CC) \le \lambda^{\ssup 1}_\CC < \Psi^{\ssup 2}_s$,
while, if $z \in \scrC$ and $\CC \in \mathfrak{C}^\delta_{L_t,A}$ 
are related as in Lemma~\ref{l:comparisoncapitalsislands}, then
\begin{equation}\label{e:prcompPsitildePsi1}
\widetilde{\Psi}_{t,s,c}(\CC) = \Psi_{s,c}(z) + o(d_t b_t \epsilon_t).
\end{equation}
By Proposition~\ref{prop:PPPconv} and \eqref{e:monot_princev}, $Z_s$, $Z^{\ssup 1}_{s,c}$ and $Z^{\ssup 2}_{s,c}$ 
all satisfy the conditions of Lemma~\ref{l:comparisoncapitalsislands}(ii) with $L=L_t$,
and thus \eqref{e:prcompPsitildePsi1} and Lemma~\ref{l:bound_Psi_tc} together imply \eqref{e:compPsitildePsi}.
\end{proof}

\RS We proceed to the proofs of Propositions~\ref{prop:randomtruncation_forproof}--\ref{prop:pathsavoidingZ_forproof}. \eRS
Recall~\eqref{e:deflambdaLApi} and consider the following classes of paths: First set
\begin{equation}
\mathcal{N}^{\ssup 0}_{t,s} := \bigl\{\pi \in \scrP(0,\Z^d) \colon\, \supp(\pi) \subset B_{L_t}, \supp(\pi)\cap (D^\circ_{t,s})^\cc \neq \emptyset \bigr\}
\end{equation}
and then let
\begin{equation}
\mathcal{N}^{\ssup 1}_{t,s}:=\bigl\{\pi\in \mathcal{N}^{\ssup 0}_{t,s}\colon
\lambda_{L_t,A}(\pi) \le \lambda^{\ssup 1}_{\CC_{t,s}}\bigr\}
\quad\text{and}\quad
\mathcal{N}^{\ssup 2}_{t,s}:=\mathcal{N}^{\ssup 0}_{t,s}\smallsetminus \mathcal{N}^{\ssup 1}_{t,s},
\end{equation}
where $\CC_{t,s}$ is as in Lemma~\ref{l:existoptisl}.
Note that, if $\tau_{(D^\circ_{t,s})^\cc} \le s < \tau_{B^\cc_{L_t}}$, then $\pi(X_{0,s}) \in \mathcal{N}^{\ssup 1}_{t,s} \cup \mathcal{N}^{\ssup 2}_{t,s}$
and hence we may bound the contribution of each class of paths separately.
This is carried out in the following lemma, using Proposition~\ref{prop:massclass}.
\begin{lemma}
\label{l:reduce_to_inner_box}
For all $A>0$ large enough, there exists $c>0$
such that, for all $0<a \le b < \infty$,
\begin{equation}
\label{e:lemma_reducetoinnerbox1}
\ln \E_0 \left[ \texte^{\int_0^s \xi(X_u) \textd u} \1_{\{\pi(X_{0,s}) \in \mathcal{N}^{\ssup 1}_{t,s}\}} \right] 
\le s \widetilde{\Psi}_{t,s,c}(\CC_{t,s}) - (\ln_3(d L_t)-c) \RV h_t \eRV |Z_s| + o(t d_t b_t)
\end{equation}
and
\begin{equation}
\ln \E_0 \left[ \texte^{\int_0^s \xi(X_u) \textd u} \1_{\{\pi(X_{0,s}) \in \mathcal{N}^{\ssup 2}_{t,s}\}} \right] \le s \max_{\CC \neq \CC_{t,s}} \widetilde{\Psi}_{t,s,c}(\CC) + o(t d_t b_t) \label{e:lemma_reducetoinnerbox2}
\end{equation}
hold for all $s \in [at,bt]$ with probability tending to $1$ as $t\to\infty$.
\end{lemma}
\begin{proof}
On $\EE_{t,a,b}$ (cf.~\eqref{e:defgoodevents}), $\inf_{s \in [at, bt]}|Z_s| \gg \ln L_t$ and so we may apply Proposition~\ref{prop:massclass}
to $\mathcal{N}^{\ssup 1}_{t,s}$ and $\mathcal{N}^{\ssup 2}_{t,s}$.
Choose $\gamma_\pi$, $z_\pi$ as follows.
For $\pi \in \mathcal{N}^{\ssup 1}_{t,s}$,
let $\gamma_\pi = \lambda^{\ssup 1}_{\CC_{t,s}} + d_t/ \ln_3 t$
and take $z_\pi$ arbitrarily in $\supp(\pi) \cap (D^\circ_{t,s})^\cc \neq \emptyset$.
If $\pi \in \mathcal{N}^{\ssup 2}_{t,s}$, then $\supp(\pi)\cap\Pi_{L_t, A}\neq \emptyset$
and we may set $\gamma_\pi = \lambda_{L_t,A}(\pi) + d_t/\ln_3 t$, 
$z_\pi = z_{\CC_\pi}$ where $\CC_\pi \in \mathfrak{C}_{L_t,A}$ is such that $\lambda_{L_t,A}(\pi) = \lambda^{\ssup 1}_{\CC_\pi}$.
Note that, by Lemma~\ref{l:compPsitildePsi}, we may assume that $\lambda^{\ssup 1}_{\CC_{t,s}} > \widehat{a}_{L_t} - A$.
Then \twoeqref{e:lemma_reducetoinnerbox1}{e:lemma_reducetoinnerbox2} 
follow by substituting our choice of $\gamma_\pi, z_\pi$ in~\eqref{e:mass_class}, 
using the definition of $\widetilde{\Psi}_{t,s,c}$, the fact that $|z_\pi|>|Z_s|(1+h_t)$ 
for $\pi \in \NN^{\ssup 1}_{t,s}$ and noting that $d_t/\ln_3 t = o(d_t b_t)$ by~\eqref{e:relation_scales}.
\end{proof}

\begin{proof}[Proof of Proposition~\ref{prop:randomtruncation_forproof}]
This now follows from Lemmas~\ref{l:compPsitildePsi}--\ref{l:reduce_to_inner_box}, Proposition~\ref{prop:PPPconv}, the definition of~$d_t$ and~$r_t$ in~\eqref{e:def_fundam_scales} and the relations between the various error scales in~\eqref{e:relation_scales}.
\end{proof}

Next we turn to Proposition~\ref{prop:pathsavoidingZ_forproof}.
Note that paths avoiding $B_\nu(Z_s)$ do not necessarily exit an $\ell^1$-ball of radius $\ln L_t$,
so we may not directly use Proposition~\ref{prop:massclass}.
As points in $\Pi_{L,A}$ are typically far away from the origin,
this can be remedied by considering
\begin{equation}\label{e:def_neglig_paths2}
\begin{aligned}
\mathcal{N}^{\ssup 3}_{t} & := \left\{ \pi \in \scrP(0,\Z^d) \colon\, \supp(\pi) \subset B_{L_t} \setminus \Pi_{L_t, A_1} \right\}, \\
\mathcal{N}^{\ssup 4}_{t,s} & := \left\{ \pi \in \scrP(0,\Z^d) \colon\, \supp(\pi) \subset B_{L_t} \setminus B_{\nu}(Z_s), \supp(\pi) \cap \Pi_{L_t, A_1} \neq \emptyset\right\},
\end{aligned}
\end{equation}
\RS where $A_1>4d$ is fixed as in Lemma~\ref{l:potential_rel_islands_upbd}. \eRS
Since $\tau_{B_{\nu}(Z_s)} \wedge \tau_{B_{L_t}^\cc} > s\,$ implies $\,\pi(X_{0,s}) \in \mathcal{N}^{\ssup 3}_{t} \cup \NN^{\ssup 4}_{t,s}$,
we may again control the contribution of each set separately. 
For $\NN^{\ssup 3}_{t}$ this is an easy task since, for any $A, s > 0$,
\begin{equation}\label{e:pathsoutsidePi}
\ln \E_0 \left[ \texte^{\int_0^s \xi(X_u) \textd u} \1\{\tau_{B^\cc_{L_t}} \wedge \tau_{\Pi_{L_t, A}} > s\}\right] \le s (\widehat{a}_{L_t} - 2 A)
\end{equation}
by the definition of $\Pi_{L_t, A}$.
For $\mathcal{N}^{\ssup 4}_{t,s}$, we may again apply Proposition~\ref{prop:massclass}:

\begin{lemma}
\label{l:reduc_finite_box}
There exist $\nu_1 \in \N$ and $c > 0$ such that,
for all $A>0$ large enough and all $0<a \le b< \infty$,
the following holds with probability tending to $1$ as $t \to \infty$:
For all $\nu \ge \nu_1$, $s \in [at,bt]$ and $\theta > 0$,
\begin{equation}\label{e:reduc_finite_box}
\ln \E_0 \left[\texte^{\int_0^\theta \xi(X_u) \textd u} \1_{\{\pi(X_{0,\theta}) \in \mathcal{N}^{\ssup 4}_{t,s} \}} \right] 
\le \theta \left( \max_{\CC \neq \CC_{t,s}} \widetilde{\Psi}_{t,\theta,c}(\CC) \vee (\widehat{a}_{L_t} - 4d ) + o(d_t b_t) \right),
\end{equation}
where $o(d_t b_t)$ does not depend on $\theta$.
\end{lemma}
\begin{proof}
Let $\delta$, $A_1 > 4d$ and $\nu_1$ be as in Lemma~\ref{l:potential_rel_islands_upbd},
and assume that $t$ is large enough for the conclusions of this lemma to hold with $L=L_t$.
We may assume $A>A_1$.

We will apply Proposition~\ref{prop:massclass} using the islands of $\mathfrak{C}_{L_t,A_1}$.
We are justified to do so because, by Lemma~\ref{l:maxpotential},
$\Pi_{L_t, A_1} \cap B_{\ln L_t} = \emptyset$ almost surely when $t$ is large, and thus all 
$\pi \in \mathcal{N}^{\ssup 4}_{t,s}$ exit a box of radius $\ln L_t$.
Let $c=c_{A_1}$ be as in~\eqref{e:mass_class}.
Since $A>A_1$,
\begin{equation}\label{e:reduc_fin_box1}
\forall  \CC \in \mathfrak{C}_{L_t,A_1}, \, \exists \, \CC' \in \mathfrak{C}_{L_t,A} \text{ s.t.\ } \CC \subset \CC'.
\end{equation}
Recall the definition of $\lambda_{L,A}(\pi)$ in~\eqref{e:deflambdaLApi}.
For $\pi \in \mathcal{N}^{\ssup 4}_{t,s}$, let $z_\pi := z_{\CC_\pi}$ where $\CC_\pi \in \mathfrak{C}_{L_t,A_1}$ is such that
$\pi \cap \CC \cap \Pi_{L, A_1} \neq \emptyset$ and $\lambda_{L_t,A_1}(\pi) = \lambda^{\ssup 1}_{\CC_\pi}$. 
Note that $z_\pi = z_{\CC'_\pi}$ where $\CC_\pi \subset \CC'_\pi \in \mathfrak{C}_{L_t,A}$.
When $t$ is large enough, $\CC_{t,s} \in \mathfrak{C}_{L_t,A}^\delta$ by Lemma~\ref{l:compPsitildePsi}; 
hence, by Lemma~\ref{l:potential_rel_islands_upbd} and the definition of $\NN^{\ssup 4}_{t,s}$,
$\CC'_\pi \neq \CC_{t,s}  = \emptyset$.
From this we conclude that
\begin{equation}
\begin{aligned}
\label{e:reduc_fin_box2}
\theta \lambda_{L_t,A_1}(\pi) - & (\ln_3 (d L_t) - c )|z_\pi| = \theta \lambda^{\ssup 1}_{\CC_\pi} - (\ln_3 (d L_t) -c )|z_{\CC_{\pi}}| \\
& \le \theta \sup \left\{\lambda^{\ssup 1}_{\CC'} - (\ln^+_3 |z_{\CC'}| - c)^+ \frac{|z_{\CC'}|}{\theta} \colon\, \CC' \in \mathfrak{C}_{L_t,A} \setminus \{\CC_{t,s}\} \right\}.
\end{aligned}
\end{equation}
Choosing now $\gamma_\pi = \lambda_{L_t,A_1}(\pi) \vee (\widehat{a}_{L_t}-4d) + d_t/\ln_3 t$,
\eqref{e:reduc_finite_box} follows from~\eqref{e:mass_class},~\eqref{e:reduc_fin_box2}
and~\eqref{e:relation_scales}.
\end{proof}

\begin{proof}[Proof of Proposition~~\ref{prop:pathsavoidingZ_forproof}]
Proposition~\ref{prop:pathsavoidingZ_forproof} now follows
from~\eqref{e:pathsoutsidePi} with $A=A_1$ together with Lemma~\ref{l:reduc_finite_box} applied to $\theta=s$, 
Lemma~\ref{l:compPsitildePsi} and the fact that, by Proposition~\ref{prop:PPPconv},
$\Psi^{\ssup 2}_s >(\widehat{a}_{L_t} - 4d)$ for all $s \in [at, bt]$ with probability tending to~$1$ as~$t \to \infty$.
\end{proof}

\subsection{Upper bound for the total mass and proof of Theorem~\ref{thm:locus}}
\label{ss:upperboundUt}\noindent
We will prove Theorem~\ref{thm:locus} by comparing $\tfrac1t \ln U(t)$ to $\Psi^{\ssup 1}_t$ and then applying Proposition~\ref{prop:PPPconv}.
The last missing ingredient is the following upper bound for $U(t)$.
Recall that we assume \twoeqref{e:propertiesR_L}{e:addpropR_L}.

\begin{lemma}[Upper bound for the total mass]
\label{l:upperboundUt}
For any $0<a\le b < \infty$,
\begin{equation}\label{e:upperboundUt}
\sup_{s \in [at, bt]} \Bigl\{ \ln U(s) - s \Psi^{\ssup 1}_s \Bigr\} \le o(t d_t b_t)
\end{equation}
holds with probability tending to~$1$ as~$t \to \infty$.
\end{lemma}
\begin{proof}
Applying Proposition~\ref{prop:massclass} to the set of paths
\begin{equation}\label{e:defN5}
\NN^{\ssup 5}_{t} := \left\{ \pi \in \scrP(0,\Z^d) \colon\,  \supp(\pi) \subset B_{L_t}, \supp(\pi) \cap \Pi_{L_t, A} \neq \emptyset \right\}
\end{equation}
with $\gamma_\pi := \lambda_{L_t, A}(\pi) \vee (\widehat{a}_{L_t} - A) + d_t/ \ln_3 t$ and 
$z_\pi := z_{\CC_{\pi}}$ where $\CC_{\pi} \in \mathfrak{C}_{L_t, A}$ satisfies $\lambda_{L_t, A}(\pi) = \lambda^{\ssup 1}_{\CC_\pi}$, we obtain
\begin{equation}
\label{e:prupperbdUt1}
\ln \E_0 \left[\texte^{\int_0^s \xi(X_u) \textd u} \1_{\{\pi_{0,s}(X) \in \NN^{\ssup 5}_t \}} \right] 
\le
s \max_{\CC \in \mathfrak{C}_{L_t, A}} \widetilde{\Psi}_{t,s,c}(\CC) + o(t d_t b_t) 
\le s \Psi^{\ssup 1}_s + o(t d_t b_t)
\end{equation}
with probability tending to~$1$ as~$t \to \infty$ by \eqref{e:mass_class}, \eqref{e:deftildePsi},  Lemma~\ref{l:compPsitildePsi}, \eqref{e:def_fundam_scales} and \eqref{e:relation_scales}.
Now \eqref{e:upperboundUt} follows by \eqref{e:prupperbdUt1} together with \eqref{e:pathsoutsidePi} and Propositions~\ref{prop:macroboxtruncation_forproof} and \ref{prop:PPPconv}.
\end{proof}

\begin{proof}[Proof of Theorem~\ref{thm:locus}]
Proposition~\ref{prop:lowerbound_forproof} and Lemma~\ref{l:upperboundUt} imply that, for any $0<a\le b<\infty$,
\begin{equation}\label{e:prthmlocus1}
\lim_{t \to \infty} \sup_{s \in [at, bt]}\frac{\left| \tfrac1s \ln U(s) - \Psi^{\ssup 1}_s\right|}{d_t} = 0 \;\;\; \text{ in probability,}
\end{equation}
and thus the theorem follows from Proposition~\ref{prop:PPPconv} and $d_{r_t} = d_t (1+o(1))$. 
\end{proof}

\section{Localization}
\label{s:localization}
\nopagebreak\noindent
In this section we prove Propositions~\ref{prop:boundbyprincipalef_forproof}--\ref{prop:localizationinnerdomain}, dealing with 
localization of the solution to the PAM as well as the eigenfunction~$\phi_{t,s}^\circ$. The proof of the former proposition is actually quite short:

\begin{proof}[Proof of Proposition~\ref{prop:boundbyprincipalef_forproof}]
By~\eqref{e:relation_scales} and \eqref{e:boundsZ},
$B_\nu(Z_s) \subset D^{\circ}_{t, s}$ for all $s \in [at,bt]$
with probability tending to $1$ as $t \to \infty$,
and thus we may apply Lemma~\ref{l:bound_top_ef} to $\Lambda = D^\circ_{t,s}$, $z=0$, $\Gamma = B_{\nu}(Z_s)$.
\end{proof}

We now turn to the proof of Proposition~\ref{prop:localizationinnerdomain}.
The first step is to obtain a spectral gap in the inner domain $D^\circ_{t,s}$,
which is a consequence of our choice of the scale $h_t$ in~\eqref{e:relation_scales}.
Recall the following useful formulas for the second largest eigenvalue of the Anderson Hamiltonian
in a subset of $\Z^d$:
For $\Lambda \subset \Z^d$, let $\lambda^{\ssup k}_\Lambda$, $\phi^{\ssup k}_\Lambda$ be the eigenvalues and eigenvectors
of $H_\Lambda$ as in Section~\ref{ss:specbounds}.
Then we may write
\begin{equation}\label{e:minmaxlambda2}
\lambda^{\ssup 2}_\Lambda = \sup \left\{ \langle (\Delta+\xi) \phi, \phi \rangle \colon\, \phi \in \R^{\Z^d}, \supp \phi \subset \Lambda, \|\phi\|_{\ell^2(\Z^d)} = 1, \phi \perp \phi^{\ssup 1}_\Lambda \right\}.
\end{equation}
A consequence of~\eqref{e:minmaxlambda2} and~\eqref{e:RRformula} is that, 
if $\Lambda_1, \Lambda_2 \subset \Z^d$ satisfy $\dist(\Lambda_1, \Lambda_2) \ge 2$, then
\begin{equation}\label{e:2ndev_separatedsets}
\lambda^{\ssup 1}_{\Lambda_1} \ge \lambda^{\ssup 1}_{\Lambda_2}
\quad\Rightarrow\quad
\lambda^{\ssup 2}_{\Lambda_1 \cup \Lambda_2} = \max \left\{ \lambda^{\ssup 2}_{\Lambda_1}, \lambda^{\ssup 1}_{\Lambda_2} \right\}.
\end{equation}

In the following, we assume that the scale sequence~$R_L$ obeys \twoeqref{e:propertiesR_L}{e:addpropR_L}.
Recall the component $\CC_{t,s} \in \mathfrak{C}_{L_t, A}$ from Lemma~\ref{l:existoptisl}, and the notation  
$\GG_{t,s}:= \{\Psi^{\ssup 1}_{s} - \Psi^{\ssup 2}_{s} > e_t d_t \}$.
We then have: 
\begin{lemma}[Spectral gap]
\label{l:spectralgap}
For any $A>0$ large enough and any $0<a\le b < \infty$, it holds with probability tending to $1$ as $t \to \infty$ that, for all $s \in [at,bt]$, on $\GG_{t,s}$,
\begin{equation}
\lambda^{\ssup 1}_{\CC_{t,s}}  > \sup_{\substack{\CC \in \mathfrak{C}_{L_t,A} \setminus \{\CC_{t,s}\} \colon \\
\dist(\CC, D^{\circ}_{t,s}) \le (\ln t)^2}} \lambda^{\ssup 1}_\CC + d_t e_t + o(d_t e_t)
\label{e:spectralgap1}
\end{equation}
and
\begin{equation}
\lambda^{\ssup 1}_{D^\circ_{t,s}}  > \lambda^{\ssup 2}_{D^\circ_{t,s}} +d_t e_t + o(d_t e_t). \label{e:spectralgap2}
\end{equation}
\end{lemma}
\begin{proof}
Let $t$ be large enough such that the conclusion of Lemma~\ref{l:size_comps} is in place with $L=L_t$. 
Then, for any $\CC \in \mathfrak{C}_{L_t,A} \setminus \{ \CC_{t,s} \}$, 
by~\eqref{e:deftildePsi} and Lemma~\ref{l:compPsitildePsi}, 
on $\GG_{t,s}$ we have
\begin{equation}\label{e:specgap0}
\lambda^{\ssup 1}_{\CC_{t,s}} - \lambda^{\ssup 1}_\CC 
\ge d_t e_t + o(d_t b_t) - \frac{|z_\CC| \ln^+_3|z_\CC| -|Z_s| \ln^+_3 |Z_s|}{s}
\end{equation}
with probability tending to~$1$ as $t \to \infty$.
By Proposition~\ref{prop:PPPconv} and Lemma~\ref{l:size_comps}, 
we may assume that $|Z_s| \ge t^{1/2}$ and that, for all $\CC \in \mathfrak{C}_{L_t, A}$ such that
 $\dist(\CC,D^\circ_{t,s}) \le (\ln t)^2$,
$|z_\CC| \le |Z_s|(1+h_t)+ (\ln t)^2 + n_A R_{L_t} < t$. 
With the help of~\eqref{e:def_fundam_scales},~\eqref{e:relation_scales} and~\eqref{e:propertiesR_L}, we can see that 
the right-hand side of \eqref{e:specgap0} is at least
\begin{align}
\label{e:specgap1}
d_t e_t + o(d_t b_t) - 2 (\ln_3 t) \,\frac{|Z_s| h_t + (\ln t)^2}{s}
& \ge d_t e_t + o(d_t b_t) - 2 (\ln_3 t) \,\frac{r_t g_t h_t + (\ln t)^2 }{a t} \nonumber\\
& = d_t e_t + o(d_t e_t),
\end{align}
thus proving~\eqref{e:spectralgap1}.

To show~\eqref{e:spectralgap2}, we may assume $\lambda^{\ssup 2}_{D^\circ_{t,s}} > \lambda^{\ssup 1}_{D^\circ_{t,s}} - A/4$
since otherwise~\eqref{e:spectralgap2} is trivially satisfied.
For $A> \chi+1$ large enough, take $\delta \in (0,1)$ as in Lemma~\ref{l:properties_opt_comps}.
By Lemma~\ref{l:size_comps}, Proposition~\ref{prop:goodevent_forproof} and Lemma~\ref{l:compPsitildePsi}, 
we may assume that $\CC_{t,s} \subset D^\circ_{t,s}$ and $\CC_{t,s} \in \mathfrak{C}^\delta_{L_t, A}$.
Thus, by~\eqref{e:spectralgap1}, $\lambda^{\ssup 1}_{D^\circ_{t,s}} -A \ge \lambda^{\ssup 1}_{\CC_{t,s}} - A \ge \widehat{a}_{L_t} - 2A$.
Applying Theorem 2.1 of \cite{BK16} to $D := D^\circ_{t,s}$ together with \eqref{e:ev_separatedsets} and \eqref{e:2ndev_separatedsets}, we obtain
\begin{equation}\label{e:specgap2}
\lambda^{\ssup 2}_{D^\circ_{t,s}} < \left(\sup_{\CC \neq \CC_{t,s} \colon \CC \cap D^\circ_{t,s} \neq \emptyset } \lambda^{\ssup 1}_\CC \right) \vee \lambda^{\ssup 2}_{\CC_{t,s}} + 2d (\eta_A)^{R_{L_t}}, \;\;\; \text{ where } \; \eta_A := \left(1+\frac{A}{4d}\right)^{-1}.
\end{equation}
Now, by Lemma~\ref{l:properties_opt_comps}(i),~\eqref{e:spectralgap1} and~\eqref{e:specgap2},
\begin{equation}\label{e:specgap3}
\lambda^{\ssup 1}_{D^\circ_{t,s}}-\lambda^{\ssup 2}_{D^\circ_{t,s}}  > \left\{d_t e_t + o(d_t e_t) \right\} \wedge \tfrac12 \rho \ln 2 - 2d (\eta_A)^{R_t},
\end{equation}
which proves~\eqref{e:spectralgap2} since $(\eta_A)^{R_t} = o(d_t e_t)$ by~\eqref{e:def_fundam_scales},~\eqref{e:relation_scales} and~\eqref{e:propertiesR_L}.
\end{proof}

We are now in position to finish the proof.
\begin{proof}[Proof of Proposition~\ref{prop:localizationinnerdomain}(i)]
We can use the proof of Theorem 1.4 in~\cite{BK16} with the following three main modifications:
\begin{enumerate}
\item 
In the part of the proof dealing with large distances, Theorem 2.5 of~\cite{BK16} is invoked, with the generic component $\CC$ appearing in its statement now set to~$\CC_{t,s}$ (which we may and do assume to be contained in $D^{\circ}_{t,s}$). 
For that we need to show that, with probability tending to $1$ as $t \to \infty$,
\begin{equation}\label{e:localis1}
\left\| \phi^\circ_{t,s} \1_{\CC_{t,s}}\right\|_2 > \frac12  \;\;\;\;\; \forall  s \in [at, bt].
\end{equation}
The proof of Theorem 2.5 then shows that this inequality characterizes~$\CC$.

\item Still in the part dealing with large distances, we use~\eqref{e:spectralgap2} instead of Lemma 8.1 of~\cite{BK16}.

\item In the second part of the proof dealing with short distances, use~\eqref{e:potential_rel_islands_uppbd}
instead of Lemma 4.8 of~\cite{BK16}.
\end{enumerate}
With these modifications, the proof goes through in our case.

In order to complete the proof, it thus remains to establish~\eqref{e:localis1}. Let $D := D^\circ_{t,s} \setminus \CC_{t,s}$. 
We first claim that, with probability tending to $1$ as $t\to \infty$,
\begin{align}\label{e:localis2}
\lambda^{\ssup 1}_D \le \lambda^{\ssup 1}_{\CC_{t,s}} - d_t e_t + o(d_t e_t).
\end{align}
Indeed, take $A>\chi+\delta$.
By Lemma~\ref{l:compPsitildePsi}, we may assume that $\CC_{t,s} \in \mathfrak{C}^\delta_{L_t, A}$, 
and thus we may also assume that $\lambda^{\ssup 1}_D > \widehat{a}_{L_t} - A$
since otherwise~\eqref{e:localis2} is satisfied.
By Theorem 2.1 of~\cite{BK16} and~\eqref{e:ev_separatedsets},
\begin{align}\label{e:localis2.5}
\lambda^{\ssup 1}_D \le \sup \bigl\{\lambda^{\ssup 1}_\CC \colon\, \CC \in \mathfrak{C}_{L_t,A}\setminus\{\CC_{t,s}\}, \CC \cap D^\circ_{t,s} \neq \emptyset\bigr\} + 2d (\eta_A)^{R_{L_t}}
\end{align}
where $\eta_A :=(1+A/(4d))^{-1}$,
so~\eqref{e:localis2} follows by Lemma~\ref{l:spectralgap},~\eqref{e:def_fundam_scales},~\eqref{e:relation_scales} and~\eqref{e:propertiesR_L}.
Now, for $x \in D$, the eigenfunction $\phi^\circ_{t,s}$ satisfies the equation
\begin{equation}\label{e:localis3}
\left(-H_D - \lambda^{\ssup 1}_{D^\circ_{t,s}}\right) \phi^\circ_{t,s} (x) = \sum_{y \in \partial D, |y-x|=1} \phi^{\circ}_{t,s}(y)
\end{equation}
where $H_D$ is the Anderson operator in $D$ with Dirichlet boundary conditions
and $\partial D := \{x \in D^\circ_{t,s} \setminus D \colon\, \exists \; y \in D, |y-x|=1\}$.
By Lemma 4.2 of \cite{BK16},
\begin{equation}\label{e:localis4}
\left\|\phi^{\circ}_{t,s} \1_{\partial D}\right\|_{\ell^2(\Z^d)} \le 
 \{1+ A/(2d) \}^{-2 R_{L_t}} \le (\eta_A)^{R_{L_t}}. 
\end{equation}
Using \twoeqref{e:localis3}{e:localis4} together with the operator norm of the resolvent of $-H_D$
and the Cauchy-Schwarz inequality, we obtain
\begin{align}\label{e:localis5}
\left\| \phi^{\circ}_{t,s} \1_D \right\|_{\ell^2(\Z^d)}
& \le \dist(\lambda^{\ssup 1}_{D^\circ_{t,s}}, \textnormal{Spec}(-H_D))^{-1} 2d (\eta_A)^{R_{L_t}} \nonumber\\
& \le (\ln t)^2 (\eta_A)^{R_{L_t}} = o(1),
\end{align}
where the last line holds by~\eqref{e:localis2}, 
$\lambda^{\ssup 1}_{D^\circ_{t,s}} \ge \lambda^{\ssup 1}_{\CC_{t,s}}$,~\eqref{e:def_fundam_scales},~\eqref{e:relation_scales} and~\eqref{e:propertiesR_L}.
As $\| \phi^{\circ}_{t,s}\|_{\ell^2(\Z^d)}=1$, this implies~\eqref{e:localis1} as desired.
\end{proof}

\begin{proof}[Proof of Proposition~\ref{prop:localizationinnerdomain}(ii)]
To prove~\eqref{e:localizationinnerdomain2}, 
we use~\eqref{e:localizationinnerdomain1}, the representation~\eqref{e:FKrepr_ef} 
and Lemma~\ref{l:potential_rel_islands_lwbd}.
Let $c_1, c_2$ as in~\eqref{e:localizationinnerdomain1}. 
Since $\phi_{t,s}^\circ$ is normalized in $\ell^2(\Z^d)$, 
there exists $\nu_0 = \nu_0(c_1, c_2)$ such that, for all $\nu \ge \nu_0$,
\begin{equation}\label{e:lwbdef1}
\max_{y \in B_{\nu}(Z_s)} \phi_{t,s}^\circ(y) \ge \max_{y \in B_{\nu_0}(Z_s)} \phi_{t,s}^\circ(y) \ge (2 |B_{\nu_0}|)^{-\frac12} =: \varepsilon_0 > 0.
\end{equation}
Fix $\nu \ge \nu_0$ and let~$A^*, \delta$ and~$A$ be as in Lemma~\ref{l:potential_rel_islands_lwbd}.
When~$t$ is large, the conclusion of this lemma holds with $L:=L_t$.
By Lemma~\ref{l:compPsitildePsi}, we may assume that $\CC_{t,s} \in \mathfrak{C}^\delta_{L_t, A}$,
and thus~\eqref{e:potential_rel_islands_lwbd} holds for $\CC_{t,s}$.
On the other hand, by~\eqref{e:monot_princev}, \eqref{e:boundsZ} and Lemma~\ref{l:maxpotential}, we have
\begin{equation}\label{e:lwbdef2}
\lambda^{\ssup 1}_{D^\circ_{t,s}} \le \max_{x \in D^\circ_{t,s}} \xi(x) \le \max_{x \in B_{L_t}} \xi(x) \le \widehat{a}_{L_t}+1,
\end{equation}
with probability tending to $1$ as $t \to \infty$.
Since $Z_s = z_{\CC_{t,s}}$, for any $z \in B_{\nu}(Z_s)$,
\begin{equation}\label{e:lwbdef3}
\lambda^{\ssup 1}_{D^\circ_{t,s}} -\xi(z) \le 2A^* +1 =: A'.
\end{equation}
Let $\bar{x} \in B_\nu(Z_s)$ with $\phi^\circ_{t,s}(\bar{x}) = \max_{y \in B_\nu(Z_s)} \phi^\circ_{t,s}(y)$.
For $y \in B_\nu(Z_s)$, fix a shortest-distance path $\pi$ from $y$ to $\bar{x}$ inside $B_\nu(Z_s)$.
Then
\begin{equation}
\begin{aligned}
\label{e:lwbdef4}
\E_{y} \biggl[ \exp \Bigl\{\int_0^{\tau_{\bar{x}}} \bigl(\xi(X_s) &- \lambda^{\ssup 1}_{D^\circ_{t,s}} \bigr) \textd s \Bigr\}\1\{\tau_{\bar{x}} < \tau_{(D^\circ_{t,s})^\cc} \}\biggr] 
\\
\ge & \;\E_{y} \left[ \exp \left\{\int_0^{T_{|\pi|}} \left(\xi(X_s) - \lambda^{\ssup 1}_{D^\circ_{t,s}} \right) \textd s \right\}\1\{\pi^{\ssup {|\pi|}}(X) = \pi \}\right] 
\\
= & \prod_{i=0}^{|\pi|-1} \frac{1}{2d+\lambda^{\ssup 1}_{D^\circ_{t,s}} -\xi(\pi_i) } \ge (2d+A')^{-2d\nu} =: \varepsilon_1 > 0
\end{aligned}
\end{equation}
by Lemma~\ref{l:path_eval} and~\eqref{e:lwbdef3}. 
To conclude, invoke~\eqref{e:FKrepr_ef} to write
\begin{equation}\label{e:lwbd5}
\phi^\circ_{t,s}(y) = \phi^\circ_{t,s}(\bar{x}) 
\E_{y} \left[\exp\left\{ \int_0^{\tau_{\bar{x}}} \left(\xi(X_s) - \lambda^{\ssup 1}_{D^\circ_{t,s}} \right) \textd s \right\} \1\{ \tau_{\bar{x}} < \tau_{(D^\circ_{t,s})^\cc} \} \right] \ge \varepsilon_0 \varepsilon_1
\end{equation}
by~\eqref{e:lwbdef1} and~\eqref{e:lwbdef4}.
The claim now follows with $\varepsilon_\nu := \varepsilon_0 \varepsilon_1>0$.
\end{proof}

\section{Path localization}
\label{s:pathconc}\nopagebreak\noindent
In this section, we prove Propositions~\ref{prop:closetoZinshorttime} and~\ref{prop:boundforpathconc};
these proofs come in Sections~\ref{ss:fastapproachlocus} and \ref{ss:pathconclocus}, respectively.
We assume throughout that $A>0$ and $\nu \in \N$ have been fixed at sufficiently
large values to satisfy the hypotheses of all previous results.
We also assume that $R_L$ obeys \twoeqref{e:propertiesR_L}{e:addpropR_L}.
In order to avoid repetition, 
statements inside proofs are tacitly assumed to hold with probability tending to $1$ as $t \to \infty$.

\subsection{Fast approach to the localization center}
\label{ss:fastapproachlocus}\noindent
Recall the component $\CC_t = \CC_{t,t} \in \mathfrak{C}_{L_t,A}$ from Lemma~\ref{l:existoptisl}. 
We first show that, under $Q^{\ssup \xi}_t$, the random walk exits a box of radius $\ln L_t$ by time $\epsilon_t t$,
at least on the event that a neighborhood of the localization center $Z_t$ is hit by time~$t$.
\begin{lemma}
\label{l:exitsmallboxfast}
In probability under the law of $\xi$,
\begin{equation}\label{e:exitsmallboxfast}
\frac{1}{U(t)} \E_0 \left[\texte^{\int_0^t \xi(X_u) \textd u} \1\{\tau_{(D^\circ_{t,t})^\cc} > t \ge \tau_{B_\nu(Z_t)}, \tau_{B^\cc_{\lfloor \ln L_t \rfloor}} > \epsilon_t t \} \right] \,\,\underset{t \to \infty}\longrightarrow\,\,0.
\end{equation}
\end{lemma}
\begin{proof}
\RS Note that, by Proposition~\ref{prop:PPPconv}, $B_\nu(Z_t) \subset B^\cc_{\lfloor \ln L_t \rfloor}(x)$ for any $x \in B_{\lfloor \ln L_t \rfloor}$. \eRS
For such $x$, we may apply Proposition~\ref{prop:massclass} to the set of paths
\begin{equation}\label{e:prexitsmallboxfast1}
\NN^{\ssup 6}_{t,x} := \left\{ \pi \in \scrP(x,\Z^d) \colon \supp(\pi) \subset D^\circ_{t,t}, \, \supp(\pi) \cap B_{\nu}(Z_t) \neq \emptyset  \right\}
\end{equation}
with $\gamma_\pi = \lambda^{\ssup 1}_{\CC_t} + d_t / \ln_3 t$ and $z_\pi \in B_{\nu}(Z_t)$ arbitrary,
which is justified by Lemma~\ref{l:spectralgap}, Lemma~\ref{l:compPsitildePsi}, 
\eqref{e:propertiesR_L} and \eqref{e:def_fundam_scales}.
Since $|z_\pi-x| \ge |Z_t| - d \nu -d \lfloor \ln L_t \rfloor$, we obtain
\begin{multline}
\qquad
\label{e:prexitsmallboxfast2}
\ln \E_x \left[ \texte^{\int_0^{(1-\epsilon_t)t} \xi(X_u) \textd u } \1{\{\tau_{(D^\circ_{t,t})^\cc}> (1-\epsilon_t)t \ge \tau_{B_{\nu}(Z_t)}\}} \right] 
\\
\le (1-\epsilon_t)t \lambda^{\ssup 1}_{\CC_t} - |Z_t|\ln_3 |Z_t|  + o(t d_t b_t)
\qquad
\end{multline}
by \eqref{e:def_fundam_scales} and \eqref{e:relation_scales}.
On the other hand,  by Lemma~\ref{l:maxpotential}, a.s.\ eventually as $t \to \infty$,
\begin{equation}\label{e:smallpaths}
\ln \E_0 \left[\texte^{\int_0^s \xi(X_u) \textd u} \1\{\tau_{B^\cc_{\lfloor \ln L_t\rfloor}}>s\} \right] \le s \max_{x \in B_{\lfloor \ln L_t \rfloor}} \xi(x) \le s \, 2 \rho \ln_3 t \;\;\;\; \forall  s \ge 0.
\end{equation}
Now use the Markov property at time $\epsilon_t t$, \twoeqref{e:prexitsmallboxfast2}{e:smallpaths} and
Proposition~\ref{prop:lowerbound_forproof} to obtain
\begin{multline}
\qquad
\label{e:prexitsmallboxfast3}
\frac{1}{U(t)}\E_0 \left[\texte^{\int_0^t \xi(X_u) \textd u} \1{\{\tau_{(D^\circ_{t,t})^\cc} > t \ge \tau_{B_\nu(Z_t)}, \tau_{B^\cc_{\lfloor \ln L_t \rfloor}} > \epsilon_t t \}} \right]
\\
\le \exp \left\{ t(\widetilde{\Psi}_t^{\ssup 1}-\Psi_t^{\ssup 1})  -\epsilon_t t(\lambda^{\ssup 1}_{\CC_t} - 2 \rho \ln_3 t) + o(t d_t b_t)     \right\}
\qquad
\end{multline}
which goes to $0$ as $t \to \infty$ by Lemma~\ref{l:compPsitildePsi}, \eqref{e:relscales1} and $\epsilon_t \gg (\ln_3 t)^{-1}$.
\end{proof}

The following result can be seen as an alternative version of Lemma~\ref{l:reduc_finite_box}.
\begin{lemma}
\label{l:uptotimeepsilont}
There exists a constant $c > 0$ such that,
with probability tending to $1$ as $t \to \infty$,
\begin{multline}
\qquad
\label{e:uptotimeepsilont}
\ln \E_0 \left[\texte^{ \int_0^{\epsilon_t t} \xi(X_u) \textd u } \1 \{\tau_{B_{\nu}(Z_t)} \wedge \tau_{(D^\circ_{t,t})^\cc}> \epsilon_t t \ge \tau_{B^\cc_{\lfloor \ln L_t \rfloor}}, X_{\epsilon_t t} = x \} \right]
\\
\le \epsilon_t t \max_{\CC \neq \CC_t} \lambda^{\ssup 1}_\CC - \left(\ln_3 (d L_t) - c \right)|x| + o(\epsilon_t t d_t b_t)
\qquad
\end{multline}
for all $x \in \Z^d$, and $o(\epsilon_t t d_t b_t)$ in~\eqref{e:uptotimeepsilont} does not depend on $x$.
\end{lemma}
\begin{proof}
Let $A>A_1$ where $A_1>4d$ is as in Lemma~\ref{l:potential_rel_islands_upbd}, 
and define the set of paths
\begin{equation}\label{e:pruptimeepsilont1}
\NN^{\ssup 7}_{t,x} := \left\{ \pi \in \scrP(0,x) \colon\; D^\circ_{t,t} \supset \supp(\pi) \not \subset B_{\lfloor \ln L_t \rfloor}, \, \supp(\pi) \cap B_{\nu}(Z_t) = \emptyset \right\}.
\end{equation}
We wish to apply Proposition~\ref{prop:massclass} to $\NN^{\ssup 7}_{t,x}$
using the islands of $\mathfrak{C}_{L_t, A_1}$ (i.e., with $L = L_t$, $A=A_1$ therein),
similarly as in the proof of Lemma~\ref{l:reduc_finite_box}.
To this end, we take, 
for $\pi \in \NN^{\ssup 7}_{t,s}$, 
$\gamma_\pi := \max_{\CC \neq \CC_t}\lambda^{\ssup 1}_\CC + d_t / \ln_3 t$ 
(where the maximum is taken over $\CC \in \mathfrak{C}_{L_t, A} \setminus \CC_t$), 
and $z_\pi := x$. 
Let us check that $\gamma_\pi$ satisfies~\eqref{e:cond_massclass1}.
Indeed, by Lemma~\ref{l:compPsitildePsi} we may assume that $\max_{\CC \neq \CC_t} \lambda^{\ssup 1}_\CC > \widehat{a}_{L_t} - A_1$.
Reasoning as in the arguments leading to \twoeqref{e:reduc_fin_box1}{e:reduc_fin_box2},
we obtain $\lambda_{L_t, A_1}(\pi) \le \max_{\CC \neq \CC_t} \lambda^{\ssup 1}_\CC$ 
for all $\pi \in \NN^{\ssup 7}_{t,x}$, so~\eqref{e:cond_massclass1} follows.
Inserting our choice of $\gamma_\pi, z_\pi$ in~\eqref{e:mass_class} 
and using~\eqref{e:relation_scales},
we obtain~\eqref{e:uptotimeepsilont} with $c = c_{A_1}$.
\end{proof}

We can now finish the proof of Proposition~\ref{prop:closetoZinshorttime}.
\begin{proof}[Proof of Proposition~\ref{prop:closetoZinshorttime}]
The key point is to show that, for some constant $c>0$ and uniformly in $x \in \Z^d$,
\begin{equation}\label{e:prfastappr1}
\begin{aligned}
\E_0 \biggl[ &\texte^{\int_0^t \xi(X_u) \textd u} \1 \Big\{ \tau_{(D^\circ_{t,t})^\cc} > t \ge \tau_{B_\nu(Z_t)} > \epsilon_t t \ge \tau_{B^\cc_{\lfloor \ln L_t \rfloor}}, X_{\epsilon_t t} =x \Big\} \biggr]  \\
& \; \le \exp \left\{\epsilon_t t \sup_{\CC \neq \CC_t} \lambda^{\ssup 1}_\CC + (1-\epsilon_t)t \lambda^{\ssup 1}_{\CC_t} - (\ln_3 (d L_t) - c)|Z_t| + o(\epsilon_t t d_t b_t) \right\}.
\end{aligned}
\end{equation}
Indeed, assuming~\eqref{e:prfastappr1},  Proposition~\ref{prop:lowerbound_forproof}, Lemma~\ref{l:compPsitildePsi} and \eqref{e:boundsZ} allow us to write 
\begin{equation}
\begin{aligned}
\label{e:prfastappr6}
 \frac{1}{U(t)} \E_0 \biggl[ &\texte^{\int_0^t \xi(X_u) \textd u} \1_{\{ \tau_{(D^\circ_{t,t})^\cc} > t \ge \tau_{B_\nu(Z_t)} > \epsilon_t t \ge \tau_{B^\cc_{\lfloor \ln L_t \rfloor}} \}} \biggr] \\
&\le  \frac{|D^\circ_{t,t}|}{U(t)} \sup_{x \in \Z^d} 
\E_0 \left[ \texte^{\int_0^t \xi(X_u) \textd u} \1_{\{ \tau_{(D^\circ_{t,t})^\cc} > t \ge \tau_{B_\nu(Z_t)} > \epsilon_t t \ge \tau_{B^\cc_{\lfloor \ln L_t \rfloor}}, X_{\epsilon_t t} =x \}} \right] \\
&\le \exp \left\{ -\epsilon_t t (\lambda^{\ssup 1}_{\CC_t} - \max_{\CC \neq \CC_t} \lambda^{\ssup 1}_\CC) + o(\epsilon_t t d_t b_t)\right\} 
\underset{t \to \infty}\longrightarrow 0 \;\; \text{ in probability}
\end{aligned}
\end{equation}
by Lemma~\ref{l:spectralgap} and~\eqref{e:relation_scales}.
This and Lemma~\ref{l:exitsmallboxfast} yield~\eqref{e:closetoZinshorttime}.

In order to prove~\eqref{e:prfastappr1}, suppose first that $\dist(x, B_{\nu}(Z_t)) \ge \ln L_t$.
Then we may apply Proposition~\ref{prop:massclass} to the set of paths
\begin{equation}\label{e:prfastappr2}
\NN^{\ssup 8}_{t,x} := \left\{ \pi \in \scrP(x, \Z^d) \colon\, \supp(\pi) \subset D^\circ_{t,t}, \supp(\pi) \cap B_{\nu}(Z_t) \neq \emptyset \right\}
\end{equation}
with $\gamma_\pi = \lambda^{\ssup 1}_{\CC_t} + d_t / \ln_3 t$ and $z_\pi \in B_{\nu}(Z_t) \cap \supp(\pi)$ arbitrary, obtaining 
\begin{multline}
\qquad
\label{e:prfastappr3}
\ln \E_x \left[ \texte^{\int_0^{(1-\epsilon_t)t} \xi(X_u) \textd u} \1_{\{\tau_{(D^\circ_{t,t})^\cc}>(1-\epsilon_t)t \ge \tau_{B_{\nu}(Z_t)}\}}\right]
\\
\le (1-\epsilon_t)t \lambda^{\ssup 1}_{\CC_t} - (\ln_3 (d L_t) -c_A)|Z_t - x| + o(\epsilon_t t d_t b_t)
\qquad
\end{multline}
since $|z_\pi-x| \ge |Z_t -x| - d \nu$.
Noting that both~\eqref{e:prfastappr3} and~\eqref{e:uptotimeepsilont}
remain true if we substitute $c$ and $c_A$ by $c \vee c_A$,
\eqref{e:prfastappr1} follows by applying the Markov property at time $\epsilon_t t$ and then using
\eqref{e:prfastappr3}, Lemma~\ref{l:uptotimeepsilont} and the triangle inequality.

If instead $\dist(x, B_{\nu}(Z_t)) < \ln L_t$, we may bound using Lemma~\ref{l:bounds_mass}
\begin{align}
\label{e:prfastappr4}
\E_x \left[ \texte^{\int_0^{(1-\epsilon_t)t} \xi(X_u) \textd u} \1_{\{\tau_{(D^\circ_{t,t})^\cc}> (1-\epsilon_t)t \ge \tau_{B_{\nu}(Z_t)}\}}\right]
& \le \texte^{(1-\epsilon_t)t \lambda^{\ssup 1}_{D^\circ_{t,t}} } |D^\circ_{t,t}|^{\tfrac32} 
\nonumber\\
& \le \exp \left\{(1-\epsilon_t)t \lambda^{\ssup 1}_{D^\circ_{t,t}} + o(\epsilon_t t d_t b_t) \right\}
\end{align}
by \eqref{e:boundsZ} and \eqref{e:relscales1}.
By Theorem~2.1 of \cite{BK16} together with Lemma~\ref{l:spectralgap}, \eqref{e:ev_separatedsets} and \eqref{e:propertiesR_L},
\begin{equation}\label{e:prfastappr5}
\lambda^{\ssup 1}_{D^\circ_{t,t}} < \lambda^{\ssup 1}_{\CC_t} + o(\epsilon_t d_t b_t).
\end{equation}
Since $|x| > |Z_t| - d \nu - \ln L_t$,
\eqref{e:prfastappr1} again follows using the Markov property 
together with \twoeqref{e:prfastappr4}{e:prfastappr5} and Lemma~\ref{l:uptotimeepsilont}.
\end{proof}

\subsection{Local concentration}
\label{ss:pathconclocus}\noindent
In this section, we address the principal ingredient needed for the proof of path localization,
culminating in the proof of Proposition~\ref{prop:boundforpathconc}.

For $L \in \N$, let $\widetilde{\epsilon}_L := \inf \{\epsilon_s \colon\, s > 0, L_s = L \}$.
\RS
Note that $\widetilde{\epsilon}_{L_t} \le \epsilon_t$.
Using \eqref{e:defLt} and $\lim_{t \to \infty} \epsilon_t \ln_3 t = \infty$, it is straightforward to show that also $\lim_{L \to \infty} \widetilde{\epsilon}_L \ln_3 L = \infty$.
\eRS
Let
\begin{equation}\label{e:deftildeR}
\widetilde{R}_L := \left\lfloor \frac{\widetilde{\epsilon}_L \ln L}{2(n_A+1)} \right\rfloor.
\end{equation}
Note that $\widetilde{R}_L$ satisfies~\eqref{e:propertiesR_L} but \emph{not}~\eqref{e:addpropR_L}.
Furthermore, $(n_A + 1) \widetilde{R}_{L_t} \le \tfrac12 \epsilon_t \ln t$.

Let $\widetilde{\mathfrak{C}}_{L, A}$ be the analogue of $\mathfrak{C}_{L,A}$ using the radius $\widetilde{R}_L$,
and let $\widetilde{\CC}_t \in \widetilde{\mathfrak{C}}_{L, A}$ such that $Z_t \in \widetilde{\CC}_t \cap \Pi_{L_t, A}$.
This is well-defined with probability tending to $1$ as $t \to \infty$ since, by~\eqref{e:monot_princev}
and Proposition~\ref{prop:PPPconv}, we may assume that $Z_t \in \Pi_{L_t, A}$.
Note that, without assuming~\eqref{e:addpropR_L}, we cannot use Lemma~\ref{l:comparisoncapitalsislands};
in particular, it may be that $Z_t \neq z_{\widetilde{\CC}_t}$.
Nonetheless, we still have the following.

\begin{lemma}
\label{l:specgaplnL}
With probability tending to $1$ as $t \to \infty$, 
\RS
\begin{equation}\label{e:containspecgap}
\widetilde{\CC}_t \subset D^\circ_{t,t}, \qquad \lambda^{\ssup 1}_{D^\circ_{t,t}} \ge \lambda^{\ssup 1}_{\widetilde{\CC}_t} > \widehat{a}_{L_t} - \chi + o(1)
\end{equation}
and
\eRS
\begin{align}
\lambda^{\ssup 1}_{\widetilde{\CC}_t} & > \sup_{\widetilde{\CC} \in \widetilde{\mathfrak{C}}_{L_t, A} \setminus \{\widetilde{\CC}_t\} \colon \widetilde{\CC} \cap D^\circ_{t,t} \neq \emptyset} \lambda^{\ssup 1}_{\widetilde{\CC}} + d_t e_t + o(d_t e_t). \label{e:specgaplnL}
\end{align}
In particular, $\lambda^{\ssup 1}_{\widetilde{\CC}_t} = \max \{ \lambda^{\ssup 1}_{\widetilde{\CC}} \colon\, \widetilde{\CC} \in \widetilde{\mathfrak{C}}_{L_t, A}, \widetilde{\CC} \cap D^\circ_{t,t} \neq \emptyset\}$.
\end{lemma}
\begin{proof}
Let us start with \eqref{e:containspecgap}. Note that, by \eqref{e:relation_scales} and \eqref{e:boundsZ}, 
$h_t |Z_t| > h_t f_t r_t \gg \widetilde{R}_{L_t}$, implying the containment;
the inequality between eigenvalues then follows by \eqref{e:monot_princev}. 
Now fix $R_L \le \widetilde{R}_L$ satisfying \twoeqref{e:propertiesR_L}{e:addpropR_L}
and let $\CC_t = \CC_{t,t} \in \mathfrak{C}_{L_t, A}$ as in Lemma~\ref{l:existoptisl}.
Then $\CC_t \subset \widetilde{\CC}_t$ and thus $\lambda^{\ssup 1}_{\widetilde{\CC}_t} \ge \lambda^{\ssup 1}_{\CC_t}$.
In particular, the remaining inequality in \eqref{e:containspecgap} follows by Lemma~\ref{l:compPsitildePsi}.
Moving to \eqref{e:specgaplnL}, 
fix $\widetilde{\CC} \in \widetilde{\mathfrak{C}}_{L_t, A} \setminus \{\widetilde{\CC}_t\}$, $\widetilde{\CC} \cap D^{\circ}_{t,t} \neq \emptyset$. 
Applying Theorem~2.1 of \cite{BK16} to $D:=\widetilde{\CC}$ and then~\eqref{e:ev_separatedsets} and Lemma~\ref{l:size_comps}, we get
\begin{align}\label{e:prspecgaplnL1}
\lambda^{\ssup 1}_{\widetilde{\CC}} 
 \le \sup_{\CC \in \mathfrak{C}_{L_t, A} \colon\, \CC \cap \widetilde{\CC} \neq \emptyset} \lambda^{\ssup 1}_{\CC} + 2d (\eta_A)^{R_{L_t}} \le \sup_{\substack{\CC \in \mathfrak{C}_{L_t, A} \setminus \{\CC_t\} \colon \\ \dist(\CC, D^\circ_{t,t}) \le (\ln t)^2 }} \lambda^{\ssup 1}_{\CC} + 2d (\eta_A)^{R_{L_t}}
\end{align}
where $\eta_A := (1+A/(4d))^{-1}$. Hence~\eqref{e:specgaplnL} follows from Lemma~\ref{l:spectralgap}.
\end{proof}

We can now give the proof of Proposition~\ref{prop:boundforpathconc}.
\begin{proof}[Proof of Proposition~\ref{prop:boundforpathconc}]
Let $n_A \in \N$ be as in Lemma~\ref{l:size_comps}. Fix $x \in B_\nu(Z_t)$ and define
\begin{equation}\label{e:prboundpathconc1}
\NN^{\ssup 9}_{t,x} := \left\{ \pi \in \scrP(x, \Z^d) \colon\, \supp(\pi) \subset D^\circ_{t,t}, \max_{1 \le \ell \le |\pi|} |\pi_\ell -x| > (n_A+1) \widetilde{R}_{L_t} \right\}.
\end{equation}
Let $\vartheta_L := 3(n_A+1) \lfloor \widetilde{\epsilon}^{-1}_L \rfloor$
and note that
\begin{equation}\label{e:propsvarthetaL}
\vartheta_L \ll \ln_3 L \; \text{ as } L \to \infty \qquad \text{ and } \qquad \vartheta_L \widetilde{R}_L \ge \ln L \; \text{ for all } L \text{ large enough.}
\end{equation}
Choosing $\gamma_\pi := \lambda^{\ssup 1}_{\widetilde{\CC}_t} + 2/t$, 
by Lemma~\ref{l:specgaplnL} and~\eqref{e:propsvarthetaL},
we may apply Proposition~\ref{prop:massclasslargeR}
(using the islands of $\widetilde{\mathfrak{C}}_{L_t, A}$)
to $\NN^{\ssup 9}_{t,x}$, obtaining, for all $0 \le s \le t$,
\begin{align}\label{e:prboundpathconc2}
\E_x \left[\texte^{\int_0^s \xi(X_u) \textd u} \1_{\{ \tau_{(D^\circ_{t,t})^\cc}>s, \, \sup_{0 \le u \le s} |X_u-x|> \tfrac12 \epsilon_t \ln t  \}} \right]
\le \texte^2 \exp \left\{ s \lambda^{\ssup 1}_{\widetilde{\CC}_t} - \tfrac12 \widetilde{R}_{L_t} \ln_3 L_t \right\}
\end{align}
since $\tfrac12 \epsilon_t \ln t \ge (n_A+1) \widetilde{R}_{L_t}$.
Now note that,
by Lemma~\ref{l:bounds_mass} and Proposition~\ref{prop:localizationinnerdomain}(ii),
\begin{equation}\label{e:prboundpathconc3}
\E_x \left[ \texte^{\int_0^s \xi(X_u) \textd u} \right] \ge \E_x \left[ \texte^{\int_0^s \xi(X_u) \textd u} \1 \{\tau_{D^\circ_{t,t}} > s, X_s = x\}\right] \ge 
\varepsilon_\nu^2 \exp \left\{ s \lambda^{\ssup 1}_{D^\circ_{t,t}} \right\}.
\end{equation}
\RV
Noting that $\widetilde{R}_L \ln_3 L \gg \ln L$, 
\eqref{e:boundforpathconc} follows from \twoeqref{e:prboundpathconc2}{e:prboundpathconc3}, \eqref{e:containspecgap} and $L_t > t$.
\eRV
\end{proof}


\section{Local profiles}
\label{s:localprofiles}\nopagebreak\noindent
In this section, we prove Propositions~\ref{prop:improvmassconc}--\ref{prop:shapes} 
dealing with the local ``shapes'' of the solution to the PAM  and of
the potential configuration in the vicinity of the localization center,
starting with the latter.
In the following we will assume that $A > 0$ and $\nu \in \N$ 
have been taken large enough so as to satisfy the hypotheses of all previous results.

\begin{proof}[Proof of Proposition~\ref{prop:shapes}]
Fix $0<a \le b< \infty$.
Let $\text{d}(\cdot, \cdot)$ be a metric under which $[-\infty, 0]^{\Z^d}$ is compact and has the topology of pointwise convergence. 
Since for each $R \in \N$ the principal Dirichlet eigenvalue of $\Delta + V_\rho$ in $B_R$ is simple,
there exists $\varepsilon_R >0$ such that
\begin{equation}\label{e:prapproxlocef1}
\text{d}(V,V_\rho) < \varepsilon_R \;\; \Rightarrow \;\; \sup_{x \in B_R} \left|V(x) - V_\rho(x) \right| \; \vee \; \left\| v^R_V - v^R_\rho \right\|_{\ell^1} < \frac{1}{R},
\end{equation}
where $v^R_V$, resp.,~$v^R_\rho$ are the principal Dirichlet eigenfunctions of $ \Delta + V$, resp., $\Delta + V_\rho$ in~$B_R$, both normalised in $\ell^1$.
Under Assumption~\ref{A:uniqueness}, Lemma~3.2(i) in \cite{GKM07} shows that the quantity
\begin{equation*}\label{e:prapproxlocef2}
\FF(\varepsilon) := -\chi - \sup \left\{ \lambda^{\ssup 1}(V) \colon\, V \in [-\infty,0]^{\Z^d}, \LL(V) \le 1, 0 \in \textnormal{argmax}(V), \text{d}(V,V_\rho) \ge \varepsilon \right\}
\end{equation*}
is strictly positive for $\varepsilon > 0$.
By Lemmas~\ref{l:maxpotential}, \ref{l:uppbddcurlyL} and \ref{l:compPsitildePsi},
there exists a deterministic non-increasing function $\delta_t>0$ such that  $\delta_t \to 0$ as $t \to \infty$ and the following holds
with probability tending to $1$ as $t \to \infty$: 
\begin{equation}
\begin{aligned}
\label{e:prapproxlocef3}
\max_{x \in B_{L_t}} \xi(x) < \widehat{a}_{L_t} + \delta_t, \qquad
\inf_{s \in [at, bt]} \lambda^{\ssup 1}_{\CC_{t,s}} > \widehat{a}_{L_t} - \chi - \delta_t
\end{aligned}
\end{equation}
and
\begin{equation}
\label{e:prapproxlocef3.5}
\sup_{s \in [at, bt]} \LL_{\CC_{t,s}}(\xi - \widehat{a}_{L_t} - \delta_t) \le 1.
\end{equation}
Letting $t_R>0$ with $t_R \to \infty$ be
such that $\delta_{t}< \tfrac12 \FF(\varepsilon_R)$ for all $t \ge t_R$, we define 
\begin{equation}\label{e:defmut}
\mu_t := \inf \{R \in \N \colon\, t_{R+1} > t \}.
\end{equation}
Then $\mu_t \to \infty$, and we can take $\mu_t \ll (\ln t)^\kappa$ by making $t_R$ grow sufficiently fast with~$R$.
\RS By \eqref{e:propertiesR_L}, \eqref{e:boundsZ} and Lemma~\ref{l:size_comps}, we may assume that $B_{\mu_t}(Z_s) \subset \CC_{t,s} \subset B_{L_t}$. \eRS
Defining
\begin{equation}\label{e:prapproxlocef4}
V^*(x) := \left\{ \begin{array}{ll} \xi(x+Z_s) - \widehat{a}_{L_t} - \delta_t & \text{ if } x +Z_s \in \CC_{t,s},\\
                   - \infty & \text{ otherwise,}\\
                   \end{array}\right.
\end{equation}
we have $V^* \in [-\infty,0]^{\Z^d}$, $\LL(V^*) = \LL_{\CC_{t,s}}(\xi-\widehat{a}_{L_t} -\delta_t) \le 1$ and $0 \in \textnormal{argmax}(V^*)$.
Furthermore,  $\lambda^{\ssup 1}(V^*) = \lambda^{\ssup 1}_{\CC_{t,s}}-\widehat{a}_{L_t} - \delta_t > -\chi - \FF(\varepsilon_{\mu_t})$.
Since $v^{\mu_t}_{V^*}(\cdot)=\phi^\bullet_{t,s}(\cdot+Z_s)$,
\begin{equation}\label{e:prapproxlocef5}
\sup_{x \in \mu_t} \left| \xi(x+Z_s) - \widehat{a}_{L_t} - V_\rho(x) \right| \; \vee \; \|\phi^\bullet_{t,s}(Z_s+\cdot)-v^{\mu_t}_\rho(\cdot) \|_{\ell^1} < \frac{1}{\mu_t} + \delta_t
\end{equation}
by~\eqref{e:prapproxlocef1} and the definition of $\FF(\varepsilon)$.
To conclude, we observe that $\widehat{a}_{L_t} = \widehat{a}_t + o(1)$ and that, by Lemma 3.3(iii) of \cite{GKM07}, 
$\lim_{t \to \infty}\|v^{\mu_t}_\rho-v_\rho\|_{\ell^1} = 0$.
\end{proof}

Next we prove Proposition~\ref{prop:improvmassconc} by adapting the strategy of Section 8.2 of \cite{GKM07}.
The proof is based on two lemmas whose proofs will be postponed to subsequent subsections. 
Fix $\mu_t \in \N$, $1 \ll \mu_t \ll R_t$, which is enough by~\eqref{e:propertiesR_L}. 
We will again decompose the solution with the help of the Feynman-Kac representation,
which states that, for a function $f:\Z^d\to [0,\infty)$, $f \not \equiv 0$, the function
\begin{equation}\label{FKformulawithfunction}
(x,t) \mapsto \E_x \left[\texte^{\int_0^t \xi(X_s) \textd s} f(X_t) \right]
\end{equation}
is the unique positive solution of the equation~\eqref{PAM} with initial condition $f$.

Fix an auxiliary function $t\mapsto T_t \in \N$ such that $\sqrt{\mu_t} \ll T_t \ll \mu_t$. 
For notational convenience we set $B_{t,s}:=B_{\mu_t}(Z_s)$.
Using~\eqref{FKformulawithfunction},
we may write $u(x,s)=u^{\ssup 1}(x,s;t)+u^{\ssup 2}(x,s;t)$ where
\begin{equation}\label{e:defuI}
u^{\ssup 1}(x,s;t):= \E_x \left[ \texte^{\int_0^s \xi(X_u) \textd u} \1_{\{ X_s = 0, \tau_{B^\cc_{t,s}} > T_t \}}\right]
\end{equation}
and $u^{\ssup 2}$ is defined by replacing $\tau_{B^\cc_{t,s}} > T_t$ by the complementary inequality.
The first lemma shows that the contribution of $u^{\ssup 2}$ is negligible.

\begin{lemma}
\label{l:doesnotexitlocalboxveryfast}
For any $0 < a \le b < \infty$,
\begin{equation}\label{e:doesnotexitlocalboxveryfast}
\lim_{t \to \infty} \sup_{s \in [at,bt]} \1_{\GG_{t,s}} \sum_{x \in \Z^d} \frac{u^{\ssup 2}(x,s;t)}{U(s)} =0 \;\; \textnormal{ in probability.}
\end{equation}
\end{lemma}

Finally, the second lemma controls the distance between $u^{\ssup 1}$ and $\phi^\bullet_{t,s}$.

\begin{lemma}
\label{l:localcomparison}
For any $0< a \le b < \infty$,
\begin{equation}\label{e:localcomparison}
\lim_{t \to \infty} \sup_{s \in [at, bt]} \1_{\GG_{t,s}} \sum_{x \in \Z^d}  \left| \frac{u^{\ssup 1}(x,s;t)}{U(s)} - \phi^\bullet_{t,s}(x) \right| = 0 \;\; \textnormal{ in probability.}
\end{equation}
\end{lemma}

\begin{proof}[Proof of Proposition~\ref{prop:improvmassconc}]
Follows directly from Lemmas~\ref{l:doesnotexitlocalboxveryfast}--\ref{l:localcomparison}.
\end{proof}

The remainder of this section is devoted to the proofs of Lemmas~\ref{l:doesnotexitlocalboxveryfast}--\ref{l:localcomparison}.
In order to avoid repetition, we fix here $0<a \le b < \infty$, and all statements made in what follows are assumed to hold 
for all $s \in [at,bt]$ with probability tending to $1$ as $t \to \infty$.

\subsection{Contribution of $u^{\ssup 2}$}
\label{ss:prooflemma_doesnotexitlocalboxveryfast}
\vspace{-20pt}
\begin{proof}[Proof of Lemma~\ref{l:doesnotexitlocalboxveryfast}]
Recall that $B_{t,s} = B_{\mu_t}(Z_s)$. Since $u^{\ssup 2}(x,s;t) \le u(x,s)$, \eqref{e:massconcwithgap} implies 
\begin{equation}\label{e:prnegluII1}
\lim_{t \to \infty} \sup_{s \in [at,bt]} \1_{\GG_{t,s}} \sum_{x \notin B_{t,s}} \frac{u^{\ssup 2}(x,s;t)}{U(s)} = 0 \;\; \textnormal{ in probability,}
\end{equation}
so we only need to control the sum over $x \in B_{t,s}$. By the strong Markov property,
\begin{align}\label{e:expruII}
u^{\ssup 2}(x,s;t) = \E_x \left[ \exp \left\{\int_0^{\tau_{B^\cc_{t,s}}} \xi(X_\theta) \textd \theta \right\} u(X_{\tau_{B^\cc_{t,s}}}, s-\tau_{B^\cc_{t,s}}) \1_{\{X_s = 0, \tau_{B^\cc_{t,s}} \le T_t \}} \right].
\end{align}
Consider the event
\begin{equation}
\RR^{\nu}_{t,s,\theta} := \left\{\tau_{(D^\circ_{t,s})^\cc} > \theta \ge \tau_{B_{\nu}(Z_s)} \right\}, \label{defRtstheta}
\end{equation}
introduce the functions
\begin{align}
u_1(x,\theta)  & := \E_x \left[ \texte^{\int_0^\theta \xi(X_u) \textd u} \1_{\{X_\theta=0\} \cap \RR^\nu_{t,s,\theta}}\right],
\label{e:defu1}\\
u_2(x,\theta)  & := \E_x \left[ \texte^{\int_0^\theta \xi(X_u) \textd u} \1_{\{X_\theta=0\} \cap (\RR^\nu_{t,s,\theta})^\cc}\right] = u(x,\theta) - u_1(x,\theta),\label{e:defu2}
\end{align}
and define $u^{\ssup 2}_i(x,s;t)$, $i=1,2$, by substituting~$u_i$ for~$u$ in~\eqref{e:expruII}. Then, clearly, we have
$u^{\ssup 2}(x,s;t) = u^{\ssup 2}_1(x,s;t) + u^{\ssup 2}_2(x,s;t)$.  Our strategy is to separately estimate the contribution of~$u^{\ssup 2}_1$ and $u^{\ssup 2}_2$.
Starting with $u^{\ssup 2}_2$, we claim that, for all $\theta <s$,
\begin{equation}\label{e:prnegluII2}
u_2(x,s-\theta) \le \texte^{\theta(2d - \xi(0))}u_2(x,s).
\end{equation}
Indeed,~\eqref{e:prnegluII2} can be obtained from~\eqref{e:defu2} with $\theta=s$
by intersecting with the event $(R^{\nu}_{t,s, s-\theta})^\cc \cap \{ X_u=0 \,\forall\, u \in [s-\theta, s]\}$
and applying the Markov property. The inequality~\eqref{e:prnegluII2} in turn shows
\begin{equation}\label{e:prnegluII3}
\sum_{x \in B_{t,s}} \frac{u^{\ssup 2}_2(x,s;t)}{U(s)} \le |B_{\mu_t}| \, \texte^{T_t(2d+|\xi(0)|+2 \rho \ln_2 t)} \sum_{x \in \Z^d}\frac{u_2(x,s)}{U(s)},
\end{equation}
where we bound $\xi(X_\theta)\le 2 \rho \ln_2 t$ by Lemma~\ref{l:maxpotential} noting that $B_{t,s} \subset B_{t}$.
By \twoeqref{e:proofthm_massconc1}{e:proofthm_massconc2} (and invariance under time-reversal of the law of $X$), 
on $\GG_{t,s}$ we can bound~\eqref{e:prnegluII3} by
\begin{equation}\label{e:prnegluII4}
|B_{\mu_t}| \, \exp \left\{ - t (\ln t)^{-2} + T_t(2d +|\xi(0)|+2\rho \ln_2 t) \right\},
\end{equation}
which tends to $0$ as $t \to \infty$.

Thus we are left with controlling $u^{\ssup 2}_1$.
\RV
To this end, recall the setup of Lemma~\ref{l:bound_top_ef} and note that,
taking $z=0$, $\Lambda = D^{\circ}_{t,s}$ and $\Gamma = B_\nu(Z_s)$,
we have $u_1(x,\theta) = w(x,\theta)$ with $w$ as defined in \eqref{e:boundtopef1}.
Then Corollary~\ref{cor:bdEFwithtimetrans} and Proposition~\ref{prop:localizationinnerdomain} give,
\eRV
on $\GG_{t,s}$,
\begin{align}\label{e:prnegluII5}
u_1(x,s-\theta) \le \texte^{-\theta \lambda^\circ_{t,s}} \left( \inf_{y \in \Gamma} \phi^{\circ}_{t,s}(y) \right)^{-5} \phi^\circ_{t,s}(x) \sum_{y \in \Gamma} u_1(y,s) \le \texte^{-\theta \lambda^\circ_{t,s}} \varepsilon_\nu^{-5} \phi^\circ_{t,s}(x) U(s),
\end{align}
where $\lambda^\circ_{t,s}$ is the largest Dirichlet eigenvalue of $H_{D^\circ_{t,s}}$
and $\varepsilon_\nu$ is as in Proposition~\ref{prop:localizationinnerdomain}(ii).
Inserting~\eqref{e:prnegluII5} in the definition of $u^{\ssup 2}_1$, we obtain, for some constant $c_0>0$,
\begin{equation}\label{e:prnegluII6}
\sum_{x \in B_{t,s}} \frac{u^{\ssup 2}_1(x,s;t)}{U(s)} \le c_0 \mu_t^d \sup_{x \notin B_{t,s}} \phi^\circ_{t,s}(x) \sup_{x \in B_{t,s}} \E_x \left[\texte^{\int_0^{\tau_{B^\cc_{t,s}}} (\xi(X_u) - \lambda^\circ_{t,s}) \textd u} \1_{\{\tau_{B^\cc_{t,s}}\le T_t \}} \right].
\end{equation}
Since $B_{t,s} \subset D^\circ_{t,s}$, \eqref{e:monot_princev} shows that $\max_{x \in B_{t,s}} \xi(x)-\lambda^\circ_{t,s} \le 2d$.
Applying Proposition~\ref{prop:localizationinnerdomain}(i), on $\GG_{t,s}$ we may further bound~\eqref{e:prnegluII6} by
\begin{equation}\label{e:prnegluII7}
c_0 c_1 \mu_t^d \texte^{-c_2 \mu_t+ 2 d T_t}. 
\end{equation}
By our choice of $T_t$, \eqref{e:prnegluII7} tends to $0$ as $t \to \infty$, concluding the proof of Lemma~\ref{l:doesnotexitlocalboxveryfast}.
\end{proof}

\subsection{Contribution of $u^{\ssup 1}$}
\label{ss:prooflemma_localcomparison}\noindent
Let $\lambda^{\ssup k}_{t,s}$, $\phi^{\ssup k}_{t,s}$ be the ordered Dirichlet eigenvalues and respective orthonormal eigenfunctions of the Anderson
operator in $B_{t,s}$. We extend the eigenfunctions to be $0$ outside of $B_{t,s} = B_{\mu_t}(Z_s)$.
In our previous notation, 
$\lambda^\bullet_{t,s} = \lambda^{\ssup 1}_{t,s}$ and $\phi^\bullet_{t,s}= \phi^{\ssup 1}_{t,s}/\|\phi^{\ssup 1}_{t,s}\|_{\ell^1(\Z^d)}$.
We start with the following important fact.

\begin{lemma}
\label{l:specgap_localbox}
For any $0 < a \le b < \infty$, with probability tending to $1$ as $t \to \infty$,
\begin{equation}
\inf_{s \in [at,bt]} \lambda^{\ssup 1}_{t,s}  > \widehat{a}_{L_t} - \chi +o(1), \label{e:eiglocboxlarge}
\end{equation}
and
\begin{equation}
\inf_{s \in [at,bt]} \lambda^{\ssup 1}_{t,s}-\lambda^{\ssup 2}_{t,s}  \ge \tfrac13 \rho \ln 2 \label{e:specgap_localbox}.
\end{equation}
\end{lemma}
\begin{proof}
By Lemma~\ref{l:compPsitildePsi}, may assume that $\lambda^{\ssup 1}_{\CC_{t,s}} > \widehat{a}_{L_t}-\chi +o(1)$. 
Thus, by Lemma~\ref{l:properties_opt_comps}(i),
\begin{equation}\label{e:prspecgaplocbox1}
\lambda^{\ssup 1}_{\CC_{t,s}} - \lambda^{\ssup 2}_{\CC_{t,s}} > \tfrac12 \rho \ln 2.
\end{equation}
Since $B_{t,s} \subset \CC_{t,s}$, $\lambda^{\ssup 2}_{t,s} \le \lambda^{\ssup 2}_{\CC_{t,s}}$ by the minimax formula 
(see e.g.\ the proof of Lemma 4.3 in \cite{BK16}).
Furthermore, by Lemma~\ref{l:potential_rel_islands_upbd} together with Theorem 2.1 of \cite{BK16} 
(note that $\lambda^{\ssup 1}_{\CC_{t,s}}-A_1>\widehat{a}_{L_t} -2A_1$),
\begin{equation}\label{e:prspecgaplocbox2}
\lambda^{\ssup 1}_{t,s} > \lambda^{\ssup 1}_{\CC_{t,s}} - 2d \left(1 + \frac{A_1}{4d} \right)^{1-2(\mu_t-\nu_1)}.
\end{equation}
Now \twoeqref{e:eiglocboxlarge}{e:specgap_localbox} follow from \twoeqref{e:prspecgaplocbox1}{e:prspecgaplocbox2}.
\end{proof}

Lemma~\ref{l:specgap_localbox} will allow us to prove the following localization property for $\phi^{\ssup 1}_{t,s}$.
\begin{lemma}
\label{l:localizationlocalbox}
There exist $c_1, c_2 \in (0,\infty)$ and, for $R \in \N$, $\varepsilon^\bullet_R > 0$ 
such that, for all $0<a\le b <\infty$, the following holds with probability tending to $1$ as $t \to \infty$: For all $s \in [at,bt]$,
\begin{equation}
\phi^{\ssup 1}_{t,s}(x)  \le c_1 \texte^{-c_2 |x -Z_s|} \qquad \forall  x \in \Z^d, \label{e:localizationlocalbox}
\end{equation}
and
\begin{equation}
\phi^{\ssup 1}_{t,s}(y)  \ge \varepsilon^\bullet_R \qquad \forall  y \in B_R(Z_s). \label{e:lbeflocalbox}
\end{equation}
\end{lemma}
\begin{proof}
Fix $A_1$, $\nu_1$ as in Lemma~\ref{l:potential_rel_islands_upbd} and take $r > \nu_1$.
By Lemma 4.2 of \cite{BK16} and~\eqref{e:eiglocboxlarge},
\begin{equation}\label{e:prlocallocalbox1}
\sum_{x \in B_{t,s}\setminus B_r(Z_s)} |\phi^{\ssup 1}_{t,s}(x)|^2 \le \left(1 + \frac{A_1}{2d} \right)^{-2(r-\nu_1)},
\end{equation}
proving~\eqref{e:localizationlocalbox}. The bound \eqref{e:lbeflocalbox} is obtained using 
\eqref{e:localizationlocalbox} and Lemma~\ref{l:potential_rel_islands_lwbd}
as in the proof of Proposition~\ref{prop:localizationinnerdomain}(ii).
\end{proof}

We can now finish the proof of Lemma~\ref{l:localcomparison}.
\begin{proof}[Proof of Lemma~\ref{l:localcomparison}]

Using the Markov property, we can write
\begin{equation}\label{e:expruI}
u^{\ssup 1}(x,s;t) = \E_x \left[ \texte^{\int_0^{T_t} \xi(X_u) \textd u} u(X_{T_t}, s-T_t) \1_{\{\tau_{B^\cc_{t,s}}>T_t\}}\right].
\end{equation}
Since 
\begin{equation}
(x,T) \mapsto \E_x \left[ \texte^{\int_0^{T} \xi(X_u) \textd u} u(X_{T}, s-T_t) \1_{\{\tau_{B^\cc_{t,s}}>T\}}\right]
\end{equation}
solves the parabolic equation~\eqref{e:PAM} with $\Lambda := B_{t,s}$ and initial condition $u(\cdot,s-T_t) \1_{B_{t,s}}$,
an eigenvalue expansion as~\eqref{e:specrepr} gives
\begin{equation}\label{e:specrepuI}
u^{\ssup 1}(x,s;t) = \sum_{k=1}^{|B_{t,s}|} \texte^{T_t \lambda^{\ssup k}_{t,s}} \phi^{\ssup k}_{t,s}(x) \langle \phi^{\ssup k}_{t,s}, u(\cdot, s-T_t)\rangle,
\end{equation}
where $\langle \cdot, \cdot \rangle$ is the canonical inner product in $\ell^2(\Z^d)$.

Set $U^{\ssup 1}(s;t):= \sum_{x \in \Z^d} u^{\ssup 1}(x,s;t)$ and note that, by Lemma~\ref{l:doesnotexitlocalboxveryfast},
\begin{equation}\label{e:prloccomp1}
\lim_{t \to \infty} \sup_{s \in [at, bt]} \1_{\GG_{t,s}} \, \left| \frac{U^{\ssup 1}(s;t)}{U(s)} - 1 \right| =0 \;\; \textnormal{ in probability.} 
\end{equation}
It is thus enough to show~\eqref{e:localcomparison} with $U(s)$ substituted by $U^{\ssup 1}(s;t)$. Using~\eqref{e:specrepuI}, write
\begin{equation}\label{e:prloccomp2}
\frac{u^{\ssup 1}(x,s;t)}{U^{\ssup1}(s;t)} = \frac{\phi^{\ssup 1}_{t,s}(x) + E_{t,s}(x)}{\|\phi^{\ssup 1}_{t,s}\|_{\ell^1(\Z^d)} + \sum_{x \in \Z^d} E_{t,s}(x)}
\end{equation}
where
\begin{equation}\label{e:defEts}
E_{t,s}(x) := \sum_{k=2}^{|B_{t,s}|} \texte^{-T_t(\lambda^{\ssup 1}_{t,s} - \lambda^{\ssup k}_{t,s})} \phi^{\ssup k}_{t,s}(x) \frac{\langle \phi^{\ssup k}_{t,s}, u(\cdot, s-T_t)\rangle}{\langle \phi^{\ssup 1}_{t,s}, u(\cdot, s-T_t)\rangle}.
\end{equation}
Once we show that
\begin{equation}\label{e:prloccomp3}
\lim_{t \to \infty} \sup_{s \in [at,bt]} \1_{\GG_{t,s}} \, \|E_{t,s}\|_{\ell^1(\Z^d)} =0 \;\; \textnormal{ in probability,}
\end{equation}
the desired conclusion will follow by the bound (recall $\phi^\bullet_{t,s} = \phi^{\ssup 1}_{t,s}/\|\phi^{\ssup 1}_{t,s}\|_{\ell^1}$)
\begin{equation}\label{e:prloccomp2.5}
\left\| \frac{u^{\ssup 1}(\cdot,s;t)}{U^{\ssup1}(s;t)} - \phi^\bullet_{t,s}(\cdot)\right\|_{\ell^1(\Z^d)} 
\le \frac{2 \|E_{t,s}\|_{\ell^1(\Z^d)}}{1 - \|E_{t,s}\|_{\ell^1(\Z^d)}},
\end{equation}
where we used that $\|\phi^{\ssup 1}_{t,s}\|_{\ell^1(\Z^d)} \ge \|\phi^{\ssup 1}_{t,s}\|^2_{\ell^2(\Z^d)} =1$.
To prove \eqref{e:prloccomp3}, we first use the Cauchy-Schwarz inequality and Parseval's identity to obtain
\begin{align}\label{e:prloccomp4}
|E_{t,s}(x)| & \le \frac{\texte^{-T_t(\lambda^{\ssup 1}_{t,s} - \lambda^{\ssup 2}_{t,s})}}{\langle \phi^{\ssup 1}_{t,s}, u(\cdot, s-T_t)\rangle} 
\left(\sum_{k=1}^{|B_{t,s}|} \langle \phi^{\ssup k}_{t,s}, \1_x \rangle^2 \right)^{\tfrac12}\left(\sum_{k=1}^{|B_{t,s}|} \langle \phi^{\ssup k}_{t,s}, u(\cdot, s-T_t)\rangle^2 \right)^{\tfrac12}\nonumber\\
& = \texte^{-T_t(\lambda^{\ssup 1}_{t,s} - \lambda^{\ssup 2}_{t,s})} \frac{\|u(\cdot, s-T_t) \|_{\ell^2(B_{t,s})}}{\langle \phi^{\ssup 1}_{t,s}, u(\cdot, s-T_t)\rangle} \1_{B_{t,s}}(x).
\end{align}
Now it suffices to show that, for some positive constants $c_0, c_1$, on $\GG_{t,s}$
\begin{equation}
\|u(\cdot, s-T_t) \|_{\ell^2(\Z^d)} \le c_0 \, \texte^{-T_t \lambda^{\bullet}_{t,s}} \, U(s), \label{e:prloccomp5}
\end{equation}
and
\begin{equation}
\langle \phi^{\ssup 1}_{t,s}, u(\cdot, s-T_t)\rangle \ge c_1 \, \texte^{-T_t \lambda^{\bullet}_{t,s}} \, U(s); \label{e:prloccomp6}
\end{equation}
indeed, using \twoeqref{e:prloccomp4}{e:prloccomp6} and~\eqref{e:specgap_localbox}, we can bound
\begin{equation}\label{e:prloccomp7}
\sup_{s \in [at,bt]} \1_{\GG_{t,s}} \, \|E_{t,s}\|_{\ell^1(\Z^d)} \le \frac{c_0}{c_1} (2\mu_t+1)^{d} \texte^{-\frac{\rho \ln 2}{3} T_t}
\end{equation}
which tends to $0$ as $t \to \infty$ by our choice of $T_t$.
Thus it only remains to prove \twoeqref{e:prloccomp5}{e:prloccomp6}.
We start with \eqref{e:prloccomp5}. By the triangle inequality,
\begin{equation}\label{e:prloccomp8}
\|u(\cdot, s-T_t)\|_{\ell^2(\Z^d)} \le \|u_1(\cdot, s-T_t)\|_{\ell^2(\Z^d)} + \|u_2(\cdot, s-T_t)\|_{\ell^2(\Z^d)}
\end{equation}
where $u_1$, $u_2$ are defined as in \twoeqref{e:defu1}{e:defu2}.
Reasoning as in \twoeqref{e:prnegluII2}{e:prnegluII4},
we can see that, on $\GG_{t,s}$,
\begin{multline}
\qquad
\label{e:prloccomp9}
\frac{\|u_2(\cdot, s-T_t)\|_{\ell^2(\Z^d)}}{U(s)} \le \frac{\|u_2(\cdot, s-T_t)\|_{\ell^1(\Z^d)}}{U(s)} 
\\
\le \exp \left\{ T_t(2d + |\xi(0)|) - t (\ln t)^{-2} \right\}
\ll \texte^{-T_t \lambda^\bullet_{t,s} }
\qquad
\end{multline}
since $\lambda^{\bullet}_{t,s} \le \max_{x \in B_{t,s}} \xi(x) \le 2\rho \ln_2 t$ by Lemma~\ref{l:maxpotential}.
Using~\eqref{e:prnegluII5} we get, on $\GG_{t,s}$,
\begin{equation}\label{e:prloccomp10}
\frac{\|u_1(\cdot, s-T_t)\|_{\ell^2(\Z^d)}}{U(s)} \le \varepsilon_\nu^{-5} \texte^{-T_t \lambda^{\circ}_{t,s}} \le \varepsilon_\nu^{-5} \texte^{- T_t \lambda^{\bullet}_{t,s}  }
\end{equation}
since $\lambda^{\circ}_{t,s} \ge \lambda^{\bullet}_{t,s}$. This shows~\eqref{e:prloccomp5}.
For~\eqref{e:prloccomp6}, let $u^{\ssup 1}$, $u^{\ssup 2}$ be as in~\eqref{e:defuI} and write
\begin{align}\label{e:prloccomp11}
\langle u(\cdot, s), \phi^{\ssup 1}_{t,s} \rangle 
& = \langle u^{\ssup 1}(\cdot, s;t), \phi^{\ssup 1}_{t,s} \rangle + \langle u^{\ssup 2}(\cdot, s;t), \phi^{\ssup 1}_{t,s} \rangle \nonumber\\
& = \texte^{T_t \lambda^{\bullet}_{t,s}} \langle u(\cdot, s-T_t), \phi^{\ssup 1}_{t,s} \rangle + \langle u^{\ssup 2}(\cdot, s;t), \phi^{\ssup 1}_{t,s} \rangle
\end{align}
(where we used the spectral representation \eqref{e:specrepuI}) to obtain
\begin{equation}\label{e:prloccomp12}
\langle u(\cdot, s-T_t), \phi^{\ssup 1}_{t,s} \rangle = \texte^{-T_t \lambda^\bullet_{t,s}} \left\{ \langle u(\cdot, s), \phi^{\ssup 1}_{t,s} \rangle  - \langle u^{\ssup 2}(\cdot, s;t), \phi^{\ssup 1}_{t,s} \rangle \right\}.
\end{equation}
Fix $R \in \N$ such that~\eqref{e:massconcwithgap} holds with $\delta < \tfrac12$ and, 
for this $R$, take $\varepsilon^\bullet_R > 0$ as in~\eqref{e:lbeflocalbox}. 
Then on $\GG_{t,s}$ we can estimate
\begin{equation}\label{e:prloccomp13}
\langle u(\cdot, s), \phi^{\ssup 1}_{t,s} \rangle \ge \sum_{x \in B_R(Z_s)} \phi^{\ssup 1}_{t,s}(x) u(x,s)
\ge \varepsilon^\bullet_R (1-\delta)U(s) > \tfrac 12 \varepsilon^\bullet_R U(s).
\end{equation}
On the other hand, by Lemma~\ref{l:doesnotexitlocalboxveryfast}, the second term inside the brackets in~\eqref{e:prloccomp12} 
multiplied by $\1_{\GG_{t,s}}$ is smaller than $\varepsilon^\bullet_R U(s)/4$ with probability tending to $1$, 
proving~\eqref{e:prloccomp6} with $c_1 = \tfrac14 \varepsilon^\bullet_R$. 
This concludes the proof of Lemma~\ref{l:localcomparison}.
\end{proof}

\appendix

\section{A tail estimate}
\label{s:tailestimate}\nopagebreak\noindent
In this section, 
we prove~\eqref{e:PPPconvcompact_cond2} for $\widehat{Y}_t$ given by~\eqref{e:defhatYt} 
using an approach from \cite{BK16}.
We will strongly rely on Assumption~\ref{A:dexp}.
The first step concerns the tail of $\xi$.

\begin{lemma}
\label{l:bdddevpotential}
For any $\varepsilon>0$, there exists $t_0>0$ such that, for all $t\ge t_0$,
\begin{equation}\label{e:bdddevpotential}
t^d \,\textnormal{Prob} \left( \xi(0) > \widehat{a}_t + s d_t \right) \le \texte^{-s (1- \varepsilon)} \qquad \forall  s \ge 0.
\end{equation}
\end{lemma}

\begin{proof}
Recall the definition of $F$ in~\eqref{e:defF}. 
Note that $t^d = \exp(\texte^{F(\widehat{a}_t)})$ to write
\begin{multline}
\qquad
\label{e:prbdddevpot1}
-\ln \left\{ t^d \textnormal{Prob} \left( \xi(0) > \widehat{a}_t + s d_t \right) \right\} 
\\
= \texte^{F(\widehat{a}_t)} \left( \texte^{F(\widehat{a}_t + s d_t)- F(\widehat{a}_t)} - 1 \right) \ge \texte^{F(\widehat{a}_t)} \left\{ F(\widehat{a}_t + s d_t)- F(\widehat{a}_t) \right\}
\qquad
\end{multline}
where in the last inequality we used $\texte^x-1 \ge x$.
Using~\eqref{e:assumpF'} and the Mean Value Theorem, we obtain $F(\widehat{a}_t + sd_t)-F(\widehat{a}_t) \ge  s d_t (1-\varepsilon) / \rho $
for all $s \ge 0$ if $t$ is large enough. Since $d_t = \rho \texte^{-F(\widehat{a}_t)}$,~\eqref{e:bdddevpotential} follows from~\eqref{e:prbdddevpot1}.
\end{proof}
Lemma~\ref{l:bdddevpotential} will allow us to reduce the sum in~\eqref{e:PPPconvcompact_cond2} to $|x| \le 6 d \theta t / d_t$. 

\begin{corollary}\label{cor:uppersumneglig}
For any $\eta \in \R$, $\theta \in (0,\infty)$, 

\begin{equation}\label{e:uppersumneglig}
\lim_{t \to \infty} \sum_{\begin{subarray}{c}
x \in (2 \widehat{N}_t+1)\Z^d \\
|x|>  6 d \theta t / d_t
\end{subarray}}
 \textnormal{Prob} \left(\widehat{Y}_t(0) > \frac{|x|}{\theta t} + \eta \right) = 0.
\end{equation}
\end{corollary}
\begin{proof}
Recall that $\max_{x \in B_{\widehat{N}_t}} \xi(x) \ge \lambda^{\ssup 1}_{B_{\widehat{N}_t}}$ by~\eqref{e:monot_princev}.
Using $a_t = \widehat{a}_t - \chi + o(1)$ and $\chi \le 2d$, we obtain, for each $L \in \N$,

\begin{equation}
\begin{aligned}
\label{e:pruppersumneg1}
\limsup_{t \to \infty} &\sum_{\begin{subarray}{c}
x \in (2 \widehat{N}_t+1)\Z^d \\
|x|>  6 d \theta t / d_t
\end{subarray}} 
\textnormal{Prob} \left(\widehat{Y}_t(0) > \frac{|x|}{\theta t} + \eta \right) 
\\
\le & \limsup_{t \to \infty} \sum_{\begin{subarray}{c}
x \in (2 \widehat{N}_t+1)\Z^d 
\\
|x|>  6 d \theta t / d_t
\end{subarray}}  |B_{\widehat{N}_t}|\textnormal{Prob} \left(\xi(0) > \widehat{a}_t + \frac{d_t}{2} \left( \frac{|x|}{\theta t} + 2\eta \right) \right) 
\\
\le & \limsup_{t \to \infty} \sum_{\begin{subarray}{c}
x \in \Z^d \\ |x|>  L t /(2\widehat{N}_t + 1)
\end{subarray}} \frac{|B_{\widehat{N}_t}|}{t^d} \exp \left\{ - \tfrac14 \left(\frac{|x|(2\widehat{N}_t+1)}{\theta t} + 2 \eta \right) \right\} \\
= & \int_{|z| \ge L} \texte^{- \tfrac14 \left(\frac{|z|}{\theta} + 2 \eta \right)} \textd z
\end{aligned}
\end{equation}
by Lemma~\ref{l:bdddevpotential} and~\eqref{e:def_fundam_scales}. Since the integral 
converges to $0$ as $L \to \infty$,~\eqref{e:uppersumneglig} follows.
\end{proof}

To control the sum in~\eqref{e:PPPconvcompact_cond2} with $|x| \le t 6 d \theta / d_t$, we will use the following lemma.
\begin{lemma}
\label{l:controltaillocev}
There exist $c_0, \varepsilon>0$ such that, for all large enough $t$ and all $s \ge 0$,
\begin{equation}\label{e:controltaillocev}
\frac{t^d}{(2\widehat{N}_t)^d} \textnormal{Prob} \left( \widehat{Y}_t(0) > s \right) \le 4 \, \texte^{-c_0 s} + t^{-\varepsilon}.
\end{equation}
\end{lemma}

Before we prove Lemma~\ref{l:controltaillocev}, let us finish the proof of~\eqref{e:PPPconvcompact_cond2}.

\begin{proof}[Proof of~\eqref{e:PPPconvcompact_cond2}]
By Corollary~\ref{cor:uppersumneglig}, we may restrict the sum over $|x| \le t 6 d \theta / d_t$.
Fix $\eta \in \R$. Taking $n\ge \theta |\eta|$, if $|x| \ge nt$ then $|x|/(\theta t) + \eta \ge 0$. 
Thus we may bound, by Lemma~\ref{l:controltaillocev},
\begin{multline}
\label{e:proof_prconvmeasure11}
\sum_{\begin{subarray}{c}
x \in (2 \widehat{N}_t+1)\Z^d \\ n t \le |x| \le t 6d \theta/d_t
\end{subarray}}
\textnormal{Prob} \left( \widehat{Y}_t(0) > \frac{|x|}{\theta t} + \eta\right)\\
\le \frac{c_2 (\ln t)^d}{t^{\varepsilon}} \,\,+ \!\!\!\!\!\!\sum_{\begin{subarray}{c}
x \in \Z^d \\ n t \le |x|(2 \widehat{N}_t +1) \le t 6d \theta /d_t
\end{subarray}}
\frac{(2 \widehat{N}_t)^d}{t^d} 4 \exp \left\{ - c_0 \left( \frac{|x|(2\widehat{N}_t+1)}{\theta t}+ \eta \right) \right\}
\end{multline}
for a constant $c_2>0$ and all large enough $t$. 
To conclude~\eqref{e:PPPconvcompact_cond1}, note that the right-hand side of \eqref{e:proof_prconvmeasure11} converges as $t \to \infty$ to
\begin{equation}\label{e:proof_prconvmeasure12}
4 \int_{|z| \ge n} \texte^{-c_0 \left( \frac{|z|}{\theta} + \eta\right)} \textd z,
\end{equation}
which converges itself to $0$ as $n \to \infty$.
\end{proof}

The remainder of this section is dedicated to the proof of Lemma~\ref{l:controltaillocev}.
Note that, by Assumption~\ref{A:dexp}, 
$\xi(0)$ has a density $f$ with respect to Lebesgue measure given by
\begin{equation}\label{e:exprf}
f(r) = \left\{ \begin{array}{ll} F'(r) \exp \left\{ F(r)-\texte^{F(r)} \right\}, & r > \essinf \xi(0),\\
0 & \text{otherwise.}
\end{array}\right.
\end{equation}
The following bound holds for $f$.
\begin{lemma}
\label{l:bounddensityxi}
Fix a finite $\Lambda \subset \Z^d$ and two functions $\alpha, \varphi \colon \Lambda \to \R$.
Then, as $t \to \infty$,
\begin{equation}\label{e:bounddensityxi}
\prod_{x \in \Lambda} \frac{f(\widehat{a}_t + \varphi(x) + \alpha(x) d_t)}{f(\widehat{a}_t + \varphi(x))}
\le \exp \left \{ -(1+ o(1))\sum_{x \in \Lambda} \alpha(x) \texte^{\frac{\varphi(x)}{\rho}} + o(1) \LL_\Lambda(\varphi)\right\}
\end{equation}
where $\LL_\Lambda(\varphi)$ is as in~\eqref{e:defcurlyL}.
If $\alpha(x) \ge0$ and $|\varphi(x)| \le M$, 
then $o(1)$ only depends on $M$.
If $|\alpha(x)| \vee |\varphi(x)| \le M$, 
then equality holds in \eqref{e:bounddensityxi}
with $o(1)$ only depending on $M$.
\end{lemma}
\begin{proof}
One can follow the reasoning leading to the proof of Lemma 7.5 in \cite{BK16}.
\end{proof}

Fix now $c_0 := \tfrac14 \texte^{-2(d+1)/\rho}$; this will the constant appearing in~\eqref{l:controltaillocev}.
The following corollary is a convenient rephrasing of~\eqref{e:bounddensityxi}.
\begin{corollary}\label{cor:bdddensity}
There exists $t_0 >0$ such that, for all $t \ge t_0$,
$s \ge 0$, $\Lambda \subset \Z^d$ 
and all $\alpha, \varphi \colon \Lambda \to \R$
with $\alpha(x) \ge 0$, $-2(d+1) \leq \varphi(x) \leq 1$,
\begin{equation}\label{e:corbdddensity}
\prod_{x \in \Lambda} \frac{f(\widehat{a}_t + \varphi(x) + s \alpha(x)d_t)}{f(\widehat{a}_t + \varphi(x))} \le \exp \left\{ - 2 c_0 s \sum_{x \in \Lambda} \alpha(x) + \LL_{\Lambda}(\varphi) \right\}.
\end{equation}
\end{corollary}

We can now prove Lemma~\ref{l:controltaillocev}.
\begin{proof}[Proof of Lemma~\ref{l:controltaillocev}]
For $t >0$ such that $a_t > \essinf \xi(0)+1$, define the continuous map
\begin{equation}\label{e:defFF}
\FF_{t,s}(r) := \left\{ \begin{array}{ll}
					r &\quad\text{ if } r \le a_t-1,\\
					r - s d_t &\quad \text{ if } r \ge a_t + s d_t,\\
					\text{linear, } &\quad \text{ otherwise.}
					\end{array}\right.
\end{equation}
Then $\FF_{t,s}$ is bijective with inverse
\begin{equation}\label{e:defFF-1}
\FF^{-1}_{t,s}(r) := \left\{ \begin{array}{ll}
					r &\quad \text{ if } r \le a_t-1,\\
					r + s d_t &\quad \text{ if } r \ge a_t,\\
					\text{linear, } &\quad \text{ otherwise.}
					\end{array}\right.
\end{equation}
Let $\xi_{t,s}(x) := \FF_{t,s}(\xi(x))$.
Then $\xi_{t,s}(x)$ has a density with respect to $\xi(x)$ given by
\begin{equation}\label{e:densityxis}
\frac{\textd\xi_{t,s}(x)}{\textd\xi(x)}(r) = 
\left\{ \begin{array}{ll} 1 & \text{ if } r \le a_t -1, \\
 \left(1 + s d_t \right)^{\1\{r < a_t\}} \frac{f(\FF^{-1}_{t,s}(r))}{f(r)} & \text{ otherwise.}
 \end{array}\right.
\end{equation}
Recalling that $\lambda^{\ssup 1}_{B_R}(\xi)$ denotes the principal Dirichlet eigenvalue of $\Delta + \xi$ in $B_R$, define
\begin{align}\label{e:defeventG}
G_{t,s} := \left\{\xi \colon\, \lambda^{\ssup 1}_{B_{R_t}}(\xi) > a_t + s d_t, \, \LL_{B_{R_t}}(\xi -\widehat{a}_t) \le \ln 2, \, \max_{x \in B_{R_t}} \xi(x) \le \widehat{a}_t + 1  \right\}.
\end{align}
Since $\xi(x)-s d_t \le \xi_{t,s}(x) \le \xi(x)$, $\xi \in G_{t,s}$ implies $\xi_{t,s} \in G_{t,0}$. Write
\begin{align}\label{e:exprdevxis}
 \text{Prob} \left( \xi_{t,s} \in G_{t,0} \right) =  E \left[\1_{G_{t,0}}(\xi) \left(1 + s d_t \right)^{|\{x \in B_{R_t} \colon\, a_t-1 < \xi(x) < a_t|\}}
\!\!\!\!\!\prod_{\begin{subarray}{c}
x \in B_{R_t} \\ \xi(x) > a_t -1
\end{subarray}}\!\!\!\!\!
 \frac{f(\FF^{-1}_{t,s}(\xi(x))) }{f(\xi(x))} \right]
\end{align}
where $E$ denotes expectation with respect to $\textnormal{Prob}$.
Bound the middle term in~\eqref{e:exprdevxis} by
\begin{equation}\label{e:prbddconvmeas2}
\left(1 + s d_t \right)^{|B_{R_t}|} \le \texte^{s d_t (2R_t+1)^d} \le \texte^{s c_0}
\end{equation}
for large $t$ by~\eqref{e:propertiesR_L}.
For the product term, define $\varphi(x) := \xi(x) - \widehat{a}_t \le 1$ on $\GG_{t,0}$,
and $\alpha(x) \in [0,1]$ by the equation $\xi(x) + s d_t \alpha(x)  = \FF^{-1}_{t,s}(\xi(x))$.
Note that, if $\alpha(x) \neq 0$, then $\varphi(x) > a_t - 1 -\widehat{a}_t \geq -2(d+1)$ 
for large $t$; thus, by Corollary~\ref{cor:bdddensity},
\begin{equation}\label{e:prlemmacontroltailev1}
\prod_{x \in B_{R_t} \colon\, \xi(x) > a_t -1} \frac{f(\FF^{-1}_{t,s}(\xi(x))) }{f(\xi(x))} \le 2 \exp \left\{ -2 c_0 s \sum_{x \in B_{R_t} \colon\, \xi(x) > a_t -1} \alpha(x) \right\}
\end{equation}
since $\LL_{B_{R_t}}( \varphi) \le \ln 2$ on $G_{t,0}$.
Moreover, by \eqref{e:monot_princev}, on $\GG_{t,0}$ we have $\xi(x) > a_t$ for some $x \in B_{R_t}$ and thus also $\alpha(x) = 1 $.
Noting now that, by~\eqref{e:bdddevpotential} and Lemma~6.4 of \cite{BK16},
\begin{equation}\label{e:prlemmacontroltailev2}
\textnormal{Prob} \left( \lambda^{\ssup 1}_{B_{R_t}}(\xi) >  a_t  +s d_t \right) \le \textnormal{Prob} \left( \xi \in G_{t,s} \right) + o(t^{-(d+\varepsilon_0)})
\end{equation}
for some $\varepsilon_0>0$, we obtain by \twoeqref{e:defeventG}{e:prlemmacontroltailev2}
\begin{equation}\label{e:prlemmacontroltailev3}
\text{Prob} \left( \lambda^{\ssup 1}_{B_{R_t}}(\xi) \ge a_t + s d_t \right) 
\le 2 \texte^{-c_0 s} \text{Prob} \left( \lambda^{\ssup 1}_{B_{R_t}}(\xi) \ge a_t \right) + o(t^{-(d+\varepsilon_0)}).
\end{equation}
To pass the estimate to $\lambda^{\ssup 1}_{B_{\widehat{N}_t}}(\xi)$, note first that, by Lemma 7.6 of \cite{BK16},
\begin{equation}\label{e:prlemmacontroltailev4}
\limsup_{t \to \infty} \frac{t^d}{(2 R_t)^d}\text{Prob} \left( \lambda^{\ssup 1}_{B_{R_t}}(\xi) \ge a_t \right) \le 1,
\end{equation}
and thus for large $t$ the right-hand side of~\eqref{e:prlemmacontroltailev3} is at most $3 \, \texte^{-c_0 s} (2 R_t/t)^d + o(t^{-(d+\varepsilon_0)})$.
Moreover, by Lemma 7.7 of \cite{BK16} applied to $t_L :  = a_L-\widehat{a}_L +s d_L$ and $R'_L := (\ln_2 L)^2$,
\begin{equation}\label{e:prlemmacontroltailev5}
\frac{t^d}{(2 \widehat{N}_t)^d} \text{Prob} \left( \lambda^{\ssup 1}_{B_{\widehat{N}_t}}(\xi) \ge a_t + sd_t\right) \le \widehat{N}^{-d}_t + 4 \, \texte^{-c_0 s} + o(t^{-\varepsilon_0})
\end{equation}
for $t$ large enough, noting that $o(L^{-d})$ and $o(1)$ in equation (7.27) of \cite{BK16} are uniform on the sequence $t_L$. 
Note that the factor $2$ multiplying $R_t$ and $\widehat{N}_t$ here and not in \cite{BK16} appears
since our boxes have side-length $2R+1$ while theirs $R$.
Recalling that $\widehat{N}_t \gg t^{\beta}$ for some $\beta > 0$ and taking 
$\varepsilon := \varepsilon_0 \wedge (\beta d)$, the lemma is proved. 
\end{proof}

\section{Compactification}
\label{s:compactification}
\nopagebreak\noindent
Let $\mathfrak{E} := (\R \times \R^d) \cup [0,\infty)$
be equipped with a metric $\mathbf{d}$ defined by setting,
for $\theta, \theta' \in [0,\infty)$ and $(\lambda, z), (\lambda', z') \in \R \times \R^d$,
\begin{equation}\label{e:defmetric}
\begin{aligned}
\mathbf{d}(\theta, \theta') & := \left|\theta - \theta'\right|, \qquad \; \mathbf{d}(\theta, (\lambda, z)) := \texte^{-\lambda} + \left|\frac{|z|}{1 \vee \lambda}-\theta \right|,\\
\mathbf{d}((\lambda, z), (\lambda', z')) & := \texte^{-\lambda \wedge \lambda'} \left(1 - \texte^{-|\lambda - \lambda'| - |z - z'|} \right) + \left|\frac{|z|}{1 \vee \lambda}-\frac{|z'|}{1 \vee \lambda'}\right|.\\
\end{aligned}
\end{equation}
One may verify that $\mathbf{d}$ is indeed a metric under which $\mathfrak{E}$ is separable, complete and locally compact. Moreover:
\begin{lemma}
\label{l:setsarecompact}
For any $(\theta, \eta) \in (0,\infty) \times \R$, the set $\HH^\theta_\eta \subset \mathfrak{E}$ defined in~\eqref{e:defHH} 
is relatively compact.
\end{lemma}
\begin{proof}
Note that the closure of $\HH^\theta_\eta$ in $\mathfrak{E}$ is given by
\begin{equation}\label{e:prsetsarecomp1}
\overline{\HH^\theta_\eta} = \left\{ (\lambda, z) \in \R\times \R^d \colon\, \lambda - \frac{|z|}{\theta} \ge \eta\right\} \cup [0,\theta].
\end{equation}
Fix a sequence $(\Xi_n)_{n \in \N}$ in $\overline{\HH^\theta_\eta}$ and consider the following three cases:
\begin{enumerate}
\item $\Xi_n \in [0, \theta]$ for infinitely many $n$;

\item There is an infinite subsequence $\Xi_{n_j} = (\lambda_j, z_j) \in \R \times \R^d$ and $(\lambda_j)_{j\in \N}$ is bounded,
implying that $\{\Xi_{n_j} \colon\, j \in \N\}$ is contained in a compact subset of $\R \times \R^d$;

\item There is an infinite subsequence $\Xi_{n_j} = (\lambda_j, z_j) \in \R \times \R^d$ and $\lim_{j \to \infty}\lambda_j =\infty$.
Note that $\limsup_{j \to \infty} |z_j|/\lambda_j \le \theta$.
\end{enumerate}
As is directly checked, in each case there exists a subsequence converging in $\mathfrak{E}$ to a point of $\overline{\HH^\theta_\eta}$, thus proving the claim.
\end{proof}

We finish the section with the following important property of $\mathfrak{E}$.
\begin{lemma}
\label{l:compactsetscontained}
For any compact set $K \subset \mathfrak{E}$, there exist $\theta \in (0,\infty)$ and  $\eta \in \R$ such that $K \cap (\R \times \R^d) \subset \HH^{\theta}_\eta$.
\end{lemma}
\begin{proof}
Cover each $x \in K$ with an open set $\HH^{\theta_x}_{\eta_x} \cup [0, \theta_x)$ for some $\theta_x >0, \eta_x \in \R$.
Use compactness to extract a finite subcover corresponding to $x_1, \ldots, x_N$ and set $\theta:= \max_{i=1}^N \theta_{x_i}$, $\eta := \min_{i=1}^N \eta_{x_i}$
to obtain the result.
\end{proof}


\section{Properties of the cost functional}
\label{s:propertiescost}\noindent
In this section we prove Lemmas~\ref{l:propPhi}, \ref{l:continuityPhiskorohod}, \ref{l:nojumpsPPclosetoID} and \ref{l:convergencenojumps}.
\begin{proof}[Proof of Lemma~\ref{l:propPhi}(i)]
Fix $\theta_0 < \theta_1$ and set $(\lambda_i, z_i)=\Xi^{\ssup 1}_\vartheta(\PP)(\theta_i)$, $i=0, 1$.
Then
\begin{equation}\label{e:prmonotLambda1}
\theta_0(\lambda_1 - \lambda_0) \le |\vartheta(\lambda_1, z_1)| - |\vartheta(\lambda_0, z_0)| \le \theta_1 (\lambda_1 - \lambda_0)
\end{equation}
by the definition of $\Psi^{\ssup 1}_\vartheta(\PP)$,
so that all three functions are non-decreasing.
Now, if  $(\lambda_0, z_0)\neq (\lambda_1, z_1)$, then one of the inequalities above is strict,
since otherwise $\lambda_1 = \lambda_0$, $|\vartheta(\lambda_1, z_1)|= |\vartheta(\lambda_0, z_0)|$ 
and we would have $(\lambda_i, z_i) \in \mathfrak{S}^{\ssup 1}_{\vartheta}(\PP)(\theta_j)$ for all $i, j \in \{0,1\}$,
implying that $(\lambda_1, z_1) = (\lambda_0,z_0)$ by the definition of $\Xi^{\ssup 1}_\vartheta(\PP)$.
This concludes the proof.
\end{proof}

\begin{proof}[Proof of Lemma~\ref{l:propPhi}(ii)]
We will first consider the case $|\supp(\PP)|<\infty$. We may assume $|\supp(\PP)| \ge 2$ since otherwise there is nothing to prove.

Consider first the case $i=1$.
$\Psi^{\ssup 1}_\vartheta(\PP)$ is continuous as the pointwise maximum of finitely many continuous functions.
Lemma~\ref{l:propPhi}(i) implies that $\Xi^{\ssup 1}_\vartheta(\PP)$ jumps finitely many times, and thus has left limits; 
let us to show that it is c\`adl\`ag.
Fix $\theta_0 > 0$ and let $(\lambda_0, z_0) := \Xi^{\ssup 1}_\vartheta(\PP)(\theta_0)$.
Note first that, if $(\lambda, z) \in \mathfrak{S}^{\ssup 1}_\vartheta(\PP)(\theta_0)$,
then $\psi^\vartheta_\theta(\lambda,z) \le \psi^\vartheta_\theta(\lambda_0,z_0)$ for all $\theta \ge \theta_0$ 
because $\lambda \le \lambda_0$ by definition.
On the other hand, if $(\lambda, z) \notin \mathfrak{S}^{\ssup 1}_\vartheta(\PP)(\theta_0)$,
then there exists $\delta_{\lambda, z}>0$ such that $\psi^\vartheta_\theta(\lambda, z) < \psi^\vartheta_\theta(\lambda_0, z_0)$ for all $ \theta \in [\theta_0, \theta_0+\delta_{\lambda, z}]$.
Setting $\delta>0$ to be the smallest among these, we can see that
\begin{equation}\label{e:prpropPhi1}
(\lambda_0, z_0) \in \mathfrak{S}^{\ssup 1}_\vartheta(\PP)(\theta) \subset \mathfrak{S}^{\ssup 1}_\vartheta(\PP)(\theta_0) \;\; \forall\,  \theta \in [\theta_0, \theta_0+\delta]
\end{equation}
implying $\Xi^{\ssup 1}_\vartheta(\PP)(\theta) = \Xi^{\ssup 1}_\vartheta(\PP)(\theta_0)$ for all $\theta \in [\theta_0, \theta_0+\delta]$,
i.e., $\Xi^{\ssup 1}_\vartheta(\PP)$ is right-continuous.

Assume now by induction that the statement of Lemma~\ref{l:propPhi}(ii) has been proved in the case $|\supp(\PP)|<\infty$
for all $i \le k-1$, $k \ge 2$.
Note that, by the definition of $\Phi^{\ssup k}_\vartheta$,
\begin{equation}\label{e:prpropPhi2}
\Phi^{\ssup k}_\vartheta(\PP)(\theta)= \sum_{\Xi \in \supp(\PP)} \1_{\left\{\Xi^{\ssup 1}_\vartheta(\PP)(\theta)= \Xi \right\}} \Phi^{\ssup {k-1}}_\vartheta(\PP_{\Xi})(\theta)
\end{equation}
where $\PP_{\Xi}(\cdot) := \PP( \cdot \setminus \{\Xi\})$.
Since $\Xi^{\ssup 1}_\vartheta(\PP)$ is c\`adl\`ag, it follows from the induction hypothesis that $\Phi^{\ssup k}_\vartheta(\PP)$ is also c\`adl\`ag.
To prove in addition that $\Psi^{\ssup k}_\vartheta(\PP)$ is continuous, 
we only need to show that, if $\Xi_0:=\Xi^{\ssup 1}_\vartheta(\PP)(\theta-)\neq \Xi^{\ssup 1}_\vartheta(\PP)(\theta)=:\Xi$,
then $\Psi^{\ssup {k-1}}_\vartheta(\PP_{\Xi_0})(\theta) = \Psi^{\ssup {k-1}}_\vartheta(\PP_{\Xi})(\theta)$;
but this follows from the definition of $\Psi^{\ssup {k-1}}_\vartheta$ since, 
by the continuity of $\Psi^{\ssup 1}_\vartheta(\PP)$,
$\psi^\vartheta_\theta(\Xi_0) = \psi^\vartheta_\theta(\Xi)$.
This finishes the proof in the case $|\supp(\PP)| < \infty$.

The case $|\supp(\PP)|=\infty$ can be reduced to the previous one as follows.
First note that we may substitute $(0,\infty)$ by $[a,b]$ with $0 < a < b <\infty$ arbitrary.
Fix $i \in \N$. 
Since $\HH^a_\eta \uparrow \R \times \R^d$ as $\eta \to -\infty$, $\HH^b_{\eta}$ is relatively compact and $\PP^\vartheta \in \scrMp$, 
there exists an $\eta \in \R$ such that $i \le |\supp(\PP^\vartheta) \cap \HH^a_{\eta}| \le \PP^\vartheta(\HH^b_{\eta}) < \infty$. 
Noting that, on $[a,b]$,
$\Phi^{\ssup i}_\vartheta(\PP) = \Phi^{\ssup i}_\vartheta(\PP')$ 
where $\PP'(\cdot):=\PP(\cdot \cap \{(\lambda, z) \colon (\lambda, \vartheta(\lambda, z)) \in \HH^b_\eta\})$, 
we fall into the previous case.

For the last statements,
note that the proof above shows that $\Xi^{\ssup i}_\vartheta(\PP)$ jumps finitely
many times in each compact interval $[\theta_1, \theta_2] \subset (0,\infty)$.
Moreover, if $\vartheta(\Xi^{\ssup 1}_\vartheta(\PP)(\theta_1)) \neq 0$
and $\Xi^{\ssup 1}_\vartheta(\PP)$ is constant in $[\theta_1, \theta_2]$,
then $\Psi^{\ssup 1}_\vartheta(\PP)$ is strictly increasing in $[\theta_1, \theta_2]$.
\end{proof}

\begin{proof}[Proof of Lemma~\ref{l:continuityPhiskorohod}]
We first consider the case $1 \le |\supp(\PP)|<\infty$.
By Proposition~3.13 of \cite{R87}, 
for $t$ large enough there exist bijections $T_t : \supp(\PP) \to \supp(\PP_t)$ such that
\begin{equation}\label{e:mapsPPtPP}
\lim_{t\to\infty}\sup_{\Xi \in \supp(\PP)} \dist(T_t(\Xi),\Xi) = 0.
\end{equation}
Letting $\TT_t(\lambda, z):=(\lambda, \vartheta_t(\lambda, z))$, by~\eqref{e:propvarthetat1} and $\supp(\PP) \cap \R \times \{0\}=\emptyset$ 
we also have
\begin{equation}\label{e:convTT}
\lim_{t \to \infty} \sup_{\Xi \in \supp(\PP)} \dist( \TT_t \circ T_t (\Xi), \Xi) = 0,
\end{equation}
and $\TT_t \circ T_t$ is a bijection onto $\supp(\PP_t^{\vartheta_t})$. 
In particular, $\PP_t^{\vartheta_t} \to \PP$.

\RS
To characterize the jump times of our processes,
the following definition will be useful: 
For $\vartheta: \R \times \R^d$, $\Xi_i = (\lambda_i, z_i) \in \R \times \R^d$, $i=0,1$, and $\theta > 0$, let
\begin{align}\label{e:defFFjump}
\FF^{\vartheta}_\theta(\Xi_1, \Xi_0) := 
\left\{ \begin{array}{ll} \frac{|\vartheta(\Xi_1)|- |\vartheta(\Xi_0)| }{\lambda_1 - \lambda_0} & \text{ if } \lambda_1 > \lambda_0 \text{ and } \psi^\vartheta_\theta(\Xi_1) < \psi^\vartheta_\theta(\Xi_0), \\
\infty & \text{ otherwise.}\end{array} \right.
\end{align}
When $\vartheta(\lambda, z)=z$, we omit it from the notation.
\eRS

We now proceed with the proof. 
Let $a_0 := a$ and, recursively for $\ell \in \N$,
\begin{equation}\label{e:prcontPhi1}
a_{\ell} := \inf\{ \theta > a_{\ell-1} \colon\, \exists \, 1 \le i \le |\supp(\PP)|, \Xi^{\ssup i}_\vartheta(\PP)(\theta) \neq \Xi^{\ssup i}_\vartheta(\PP)(a_{\ell-1})\}.
\end{equation}
Note that $\Xi^{\ssup i}(\PP)$ jumps finitely many times:
for $i=1$ this follows by Lemma~\ref{l:propPhi}(i),
and for $i \ge 2$, by induction using~\eqref{e:prpropPhi2}.
Thus $\ell_* = \ell_*(a, \PP):= \inf\{\ell \ge 0 \colon\, a_{\ell+1}=\infty\} < \infty$.

We proceed by induction on $\ell_*$, starting with $\ell_* = 0$.
Since $\PP \in \widetilde{\scrM}^a_\text{P}$, the values $i \mapsto \psi_a(\Xi^{\ssup i}(\PP)(a))$ are all distinct,
which together with~\eqref{e:mapsPPtPP}--\eqref{e:convTT} implies that
$\Xi^{\ssup i}_{\vartheta_t}(\PP_t)(a) = T_t(\Xi^{\ssup i}(\PP)(a))$ for all $i$ when $t$ is large enough.
In particular,~\eqref{e:mapsPPtPP} implies the result in the case $\ell_* =0$.
Assume by induction that, for some $L \in \N$, the statement has been proved for all $a' \in (0, \infty)$ and $\PP' \in \widetilde{\scrM}^{a'}_{\text{P}}$ satisfying $|\supp(\PP')| < \infty$ and $\ell_*(a', \PP') \le L-1$, 
and suppose that $\ell_*=\ell_*(a, \PP)=L$ (in which case necessarily $|\supp(\PP)|\ge 2$).

Note now that, because $\PP \in \widetilde{\scrM}^{a}_{\text{P}}$,
there exists a unique $i_1$ such that both $\Xi^{\ssup {i_1}}(\PP)$ and $\Xi^{\ssup {i_1+1}}(\PP)$ jump at $a_1$
while $\Xi^{\ssup i}(\PP)$ is continuous at $a_1$ for all $i \notin \{i_1, i_1+1\}$.
Moreover, $\Xi^{\ssup {i_1}}(\PP)(a_1)$ is the point $\Xi \in \supp(\PP)$ minimizing $\FF_a(\Xi, \Xi^{\ssup {i_1}}(\PP)(a))$
(cf.\ \eqref{e:defFFjump}), $a_1 = \FF_a(\Xi^{\ssup {i_1}}(\PP)(a_1), \Xi^{\ssup {i_1}}(\PP)(a))$ and $\Xi^{\ssup {i_1+1}}(\PP)(a_1) = \Xi^{\ssup {i_1}}(\PP)(a)$.

Let $a^t_\ell$, $\ell^t_*$ be \RV the analogues of \eRV $a_\ell$, $\ell_*$ for~$\Xi^{\ssup i}_{\vartheta_t}(\PP_t)$,
and fix $a' \in (a_1, a_2) \cap \Q$.
By~\eqref{e:mapsPPtPP}--\eqref{e:convTT} and the previous discussion, when~$t$ is large enough, 
$\Xi^{\ssup i}_{\vartheta_t}(\PP_t)$ does not jump in $[a, a']$ for all $i \notin \{i_1, i_1+1\}$,
$\Xi^{\ssup {i_1}}_{\vartheta_t}(\PP_t)(a^t_1) = T_t(\Xi^{\ssup {i_1}}(\PP)(a_1))$, 
and $\Xi^{\ssup {i_1+1}}_{\vartheta_t}(\PP_t)(a^t_1) = \Xi^{\ssup {i_1}}_{\vartheta_t}(\PP_t)(a) = T_t(\Xi^{\ssup {i_1}}(\PP)(a))$.
\RS Moreover,  $a^t_2 > a' > a^t_1$ and
\[
a^t_1 = \FF_a^{\vartheta_t}( \Xi^{\ssup {i_1}}_{\vartheta_t}(\PP_t)(a^t_1), \Xi^{\ssup {i_1}}_{\vartheta_t})(a) )
=\FF_a(\TT_t \circ T_t(\Xi^{\ssup {i_1}}(\PP)(a_1)), \TT_t \circ T_t(\Xi^{\ssup {i_1}}(\PP)(a))),
\]
allowing us to conclude, by~\eqref{e:convTT},
\eRS
\begin{equation}\label{e:prcontPhi3}
|a_1 - a^t_1| 
\le \,\,\,\max_{\begin{subarray}{c}
\Xi_1, \Xi_2 \in \supp(\PP) \\ \FF_a(\Xi_1, \Xi_2) < \infty
\end{subarray}}
\left|\FF_a(\Xi_1, \Xi_2) - \FF_a(\TT_t \circ T_t(\Xi_1), \TT_t \circ T_t(\Xi_ 2))\right|
\underset{t \to \infty} \longrightarrow 0.
\end{equation}
Define now a time change $\sigma_t:[a, a']\to [a, a']$ by setting
\begin{equation}\label{e:prcontPhi4}
\sigma_t(a)=a, \;\;\; \sigma_t(a_1)=a^t_1, \;\;\; \sigma_t(a')=a' \;\;\; \text{ and linear otherwise.}
\end{equation}
Then, by the previous discussion together with \eqref{e:mapsPPtPP}, \eqref{e:convTT} and~\eqref{e:prcontPhi3},
\begin{equation}\label{e:prcontPhi5}
\lim_{t \to \infty} \, \sup_{1 \le i \le |\supp(\PP)|}  \, \sup_{\theta \in [a, a']} \, \left| \sigma_t(\theta) - \theta \right|  \vee \left| \Phi^{\ssup i}_{\vartheta_t}(\PP_t)(\sigma_t(\theta)) - \Phi^{\ssup i}(\PP)(\theta)\right| = 0.
\end{equation}
Since $\ell_*(a', \PP) = L-1$ and $\PP \in \widetilde{\scrM}^{a'}_{\text{P}}$,
by the induction hypothesis we can extend $\sigma_t$ to $[a, \infty)$ in such a way that~\eqref{e:prcontPhi5}
holds with $[a, a']$ substituted by $[a, \infty)$, finishing the proof in the case $|\supp(\PP)|<\infty$.

Consider now the case $|\supp(\PP)| = \infty$.
\RV
We may assume without loss of generality that $c_*$ in~\eqref{e:propvarthetat2} is not larger than $1$.
\eRV
Let us first show~\eqref{e:propPhi2}.
Fix $k \in \N$ and a point $b \in (a, \infty) \cap \Q$.
Note that, since $\PP \in \widetilde{\scrM}^{a}_{\text{P}}$, $
b$ is a continuity point of $\Phi^{\ssup i}(\PP)$ for all $1 \le i \le k$.
Let $\eta \in \R$ be negative enough such that, for all $t$ large enough,
\begin{equation}\label{e:prcontPhi6}
k \le |\supp(\PP) \cap \HH^{a}_\eta| = |\supp(\PP_t) \cap \HH^{a}_\eta| \le \PP_t(\HH^{2b/c_*}_\eta) = \PP(\HH^{2b/c_*}_\eta) < \infty,
\end{equation}
which is possible because $\PP \in \scrMp$ and $\PP_t \to \PP$.
Moreover, since $\supp(\PP) \cap \R \times \{0\} = \emptyset$, by~\eqref{e:propvarthetat1}--\eqref{e:propvarthetat2} 
we may also assume that
\begin{equation}\label{e:prcontPhi7}
k \le |\supp(\PP_t^{\vartheta_t}) \cap \HH^a_\eta| \quad \text{ and } \quad \supp(\PP_t^{\vartheta_t}) \cap \HH^b_\eta \subset \TT_t\left( \supp(\PP_t) \cap \HH^{2b/c_*}_\eta\right),
\end{equation}
where $\TT_t$ is defined right before~\eqref{e:convTT}.
Now~\eqref{e:prcontPhi6}--\eqref{e:prcontPhi7} imply that, on $[a,b]$,
$\Phi^{\ssup i}(\PP) = \Phi^{\ssup i}(\PP')$ and $\Phi^{\ssup i}_{\vartheta_t}(\PP_t) = \Phi^{\ssup i}_{\vartheta_t}(\PP'_t)$
for all $1 \le i \le k$, where $\PP'(\cdot) := \PP(\cdot \cap \HH^{2b/c_*}_\eta)$ and analogously for $\PP'_t$.
Since $\PP'_t \to \PP'$,~\eqref{e:propPhi2} follows by the previous case and Theorem~16.2 of \cite{B99}.
The convergence $\PP_t^{\vartheta_t} \to \PP$ follows from~\eqref{e:prcontPhi7},~\eqref{e:propvarthetat1} and $\PP_t \to \PP$ (note that $b$, $\eta$ above can be taken arbitrarily large, respec.\ negative).
\end{proof}

\RV
\begin{proof}[Proof of Lemma~\ref{l:nojumpsPPclosetoID}]
By Lemma~\ref{l:nojumpintermsofPP}, it is enough to show that $(4) \Rightarrow (1)$.
Arguing as at the end of the proof of Lemma~\ref{l:continuityPhiskorohod},
we reduce to the case $|\supp(\PP)| < \infty$.
Denote by $\pi: \R \times \R^d \to \R^d$ the projection on the second coordinate, i.e., 
$\pi(\lambda, z) = z$. 
For $z \in \pi(\supp(\PP))$, set $\lambda_z := \max \{\lambda \colon (\lambda, z) \in \supp(\PP)\}$
and define
$\widehat{\PP} := \sum_{z \in \pi(\supp(\PP))} \delta_{(\lambda_z, z)}$.
Note that $\pi$ is injective over $\supp(\widehat{\PP})$,
and that $\Xi^{\ssup 1}(\PP) = \Xi^{\ssup 1}(\widehat{\PP})$.
By \twoeqref{e:mapsPPtPP}{e:convTT}, when $t$ is large enough,
$\pi$ is injective over the support of $\widehat{\PP}_t:= \widehat{\PP} \circ T_t^{-1}$,
and moreover $\Xi^{\ssup 1}_{\vartheta_t}(\PP_t)(\theta) = \Xi^{\ssup 1}_{\vartheta_t} (\widehat{\PP}_t)(\theta)$
for all $\theta \in [a,b]$.
This concludes the proof.
\end{proof}
\eRV

\begin{proof}[Proof of Lemma~\ref{l:convergencenojumps}]
For $(\lambda, z) \in \R \times (\R^d \setminus \{0\})$, let
\begin{equation}\label{e:prlconvnojumps1}
\AA(\lambda, z) := \left\{ (\lambda', z') \in \R \times \R^d \colon\, 
\begin{array}{l} 
\psi_a(\lambda', z') > \psi_a(\lambda, z) \text{ or}\\
\psi_a(\lambda', z') = \psi_a(\lambda, z) \text{ and } \lambda'>\lambda
\end{array} \right\}.
\end{equation}
By the definition of $\PP^\vartheta$,
$\FF_a^{\vartheta}(\PP, \lambda, z) = \PP^\vartheta \left\{ \AA(\lambda, \vartheta(\lambda, z)) \right\}$.
Since $\vartheta_t(\lambda_t, z_t) \to z_*$ by~\eqref{e:propvarthetat1} and $\PP_t^{\vartheta_t} \to \PP$ by Lemma~\ref{l:continuityPhiskorohod},
we may assume that $\vartheta_t(\lambda, z) = z$ for all $(\lambda,z) \in \R \times \R^d$.

Now,  since $\PP \in \widetilde{\scrM}^a_{\text{P}}$, 
$\FF_a(\PP,\lambda_*, z_*) = \PP \left\{\HH^a_{\psi_a(\lambda_*, z_*)} \right\}$
and there exists a $\delta>0$ such that
\begin{equation}\label{e:prlconvnojumps2}
\PP \left\{\overline{\HH^{a}_{\psi_a(\lambda_*, z_*) - \delta}} \right\} = 1 + \PP \left\{ \overline{\HH^a_{\psi_a(\lambda_*, z_*) + \delta}} \right\}.
\end{equation}
On the other hand, since $\PP_t \to \PP$ and $(\lambda_t, z_t) \to (\lambda_*, z_*)$, 
when $t$ is large we also have
\begin{equation}\label{e:prlconvnojumps3}
\PP_t \left\{ \overline{\HH^a_{\psi_a(\lambda_*, z_*)\pm\delta}} \right\} = \PP \left\{ \overline{\HH^a_{\psi_a(\lambda_*, z_*)\pm\delta}} \right\} \; \text{ and } \; (\lambda_t, z_t) \in \HH^{a}_{\psi_b(\lambda_*, z_*)-\delta} \setminus \HH^a_{\psi_a(\lambda_*, z_*)+\delta}.
\end{equation}
In particular, for all $t$ large enough,
\begin{equation}\label{e:prlconvnojumps4}
\PP_t \left\{ \AA(\lambda_t, z_t) \right\} = \PP_t \left\{ \overline{\HH^a_{\psi_a(\lambda_*, z_*)+\delta}} \right\} 
= \PP \left\{ \overline{\HH^a_{\psi_a(\lambda_*, z_*)+\delta}} \right\} = \PP \left\{ \HH^a_{\psi_a(\lambda_*, z_*)} \right\},
\end{equation}
concluding the proof.
\end{proof}

\section*{Acknowledgments}
\nopagebreak\noindent
We thank an anonymous referee for several helpful comments and suggestions.
The research of  WK and RdS was supported by the German DFG project KO 2205/13-1 ``Random mass flow through random potential.'' MB was partially supported by the Simons Foundation and the Mathematische Forschungsinstitut Oberwolfach as a Simons visiting professor in Summer 2015, as well as by NSF grant DMS-1407558 and GA\v CR project P201/16-15238S.



\begin{thebibliography}{WWW98}

\bibitem[Ast12]{A12}
{\sc A. Astrauskas},
\newblock Extremal theory for spectrum of random discrete Schr\"odinger operator. II. Distributions with heavy tails. 
\newblock {\it J. Stat. Phys.} {\bf 146:1}, 98-117 (2012). 
\smallskip

\bibitem[Ast13]{A13}
{\sc A. Astrauskas},
\newblock Extremal theory for spectrum of random discrete Schr\"odinger operator. III. Localization properties. 
\newblock {\it J. Stat. Phys.} {\bf 150:5},  889-907 (2013).
\smallskip

\bibitem[Ast16]{A16}
{\sc A.~Astrauskas},
\newblock From extreme values of i.i.d. random fields to extreme eigenvalues of finite-volume Anderson Hamiltonian.
\newblock \textit{Probability Surveys} {\bf 13} 156--244  (2016).


\bibitem[B99]{B99}
{\sc P.~Billingsley},
\newblock \textit{Convergence of Probability Measures}, 2nd edition, 
\newblock John Wiley \& Sons, New York (1999).
\smallskip

\bibitem[BK01a]{BK01a}
{\sc M.~Biskup} and {\sc W.~K\"onig,}
\newblock {Long-time tails in the parabolic Anderson model
with bounded potential}.
\newblock {\em Ann.\ Probab.} {\bf 29:2}, 636-682 (2001).
\smallskip

\bibitem[BK01b]{BK01b}
{\sc M.~Biskup} and {\sc W.~K\"onig,}
\newblock {Screening effect due to heavy lower tails in one-dimensional parabolic Anderson model}, 
\newblock {\em J.\ Stat.\ Phys.} {\bf 102:5/6}, 1253--1270 (2001).
\smallskip

\bibitem[BK16]{BK16}
{\sc M.~Biskup} and {\sc W.~K\"onig,}
\newblock Eigenvalue order statistics for random Schr\"odinger operators with doubly-exponential tails,
\newblock {\em Commun.\ Math.\ Phys.} {\bf 341:1}, 179-218 (2016).
\smallskip

\bibitem[CM94]{CM94}
{\sc R.~Carmona} and {\sc S.A.~Molchanov},
\newblock Parabolic Anderson problem and intermittency.
\newblock Mem. Amer. Math. Soc. {\bf 108} no.\ 518 (1994).
\smallskip

\bibitem[FM14]{FM14}
{\sc A.~Fiodorov} and {\sc S.~Muirhead,}
\newblock Complete localisation and exponential shape of the parabolic Anderson model with Weibull potential field,
\newblock {\it Electron.\ J.\ Probab.\ } {\bf 19:58}, 1--27 (2014).
\smallskip

\bibitem[G99]{G99}
{\sc G.~Grimmett,}
\newblock {\textit{Percolation}}, Second edition, 
\newblock Springer, Berlin (1999).
\smallskip

\bibitem[GH99]{GH99}
{\sc J.~G\"artner} and {\sc F.~den~Hollander,}
\newblock Correlation structure of intermittency in the
parabolic Anderson model.
\newblock {\em Probab.\ Theory Relat.\ Fields} {\bf 114}, 1--54 (1999).
\smallskip

\bibitem[GKM07]{GKM07}
{\sc J.~G\"artner, W.~K\"onig} and {\sc S.~Molchanov},
\newblock Geometric characterization of intermittency in the parabolic Anderson model.
\newblock {\it Ann. Probab.}  {\bf 35:2}, 439--499 (2007).
\smallskip

\bibitem[GM90]{GM90}
{\sc J.~G\"artner} and {\sc S.~Molchanov},
\newblock Parabolic problems for the Anderson model I. Intermittency
and related topics.
\newblock {\em Commun.\ Math.\ Phys.} {\bf 132}, 613--655 (1990).
\smallskip

\bibitem[GM98]{GM98}
{\sc J.~G\"artner} and {\sc S.~Molchanov},
\newblock Parabolic problems for the Anderson model II. Second-order
asymptotics and structure of high peaks.
\newblock {\em Probab.\ Theory Relat.\ Fields}  {\bf 111}, 17--55~(1998).
\smallskip

\bibitem[HKM06]{HKM06}
{\sc R.~van der Hofstad, W.~K\"onig } and {\sc P.~M\"orters},
\newblock The universality classes in the parabolic Anderson model.
\newblock {\it Commun. Math. Phys.} {\bf  267:2}, 307-353 (2006).
\smallskip

\bibitem[HMS08]{HMS08}
{\sc R.~van der Hofstad, P.~M\"orters} and {\sc N.~Sidorova},
\newblock Weak and almost sure limits for the parabolic Anderson model with heavy-tailed potential,
\newblock {\it Ann.\ Appl.\ Prob.} {\bf 18}, 2450--2494 (2008).
\smallskip

\bibitem[K16]{K15}
{\sc W.~K\"onig},
\newblock \textit{The Parabolic Anderson Model},
\newblock {Pathways in Mathematics}, Birkh\"auser (2016).
\smallskip

\bibitem[KLMS09]{KLMS09}
{\sc W.~K\"onig, H.~Lacoin, P.~M\"orters}  and {\sc N.~Sidorova},
\newblock A two cities theorem for the parabolic Anderson model,
\newblock {\textit Ann.\ Probab.} {\bf 37:1}, 347--392 (2009).
\smallskip

\bibitem[LM12]{LM12}
{\sc H.~Lacoin} and {\sc P.~M\"orters},
\newblock A scaling limit theorem for the parabolic Anderson model with exponential potential,
\newblock in: J.-D. Deuschel et al (eds.), 
{\it  Probability in Complex Physical Systems In Honour of Erwin Bolthausen and J\"urgen G\"artner}, Springer Proceedings in Mathematics {\bf 11} 153--179, Springer (2012).
\smallskip

\bibitem[M02]{M02}
{\sc J.B.~Martin,}
\newblock Linear growth for greedy lattice animals,
\newblock {\it Stoch.\ Proc.\ Appl.} {\bf 98}, 43--66 (2002)
\smallskip


\bibitem[M94]{M94}
{\sc S.~Molchanov},
\newblock Lectures on random media.
\newblock In: D.~Bakry, R.D.~Gill, and
S.~Molchanov, {\it Lectures on Probability Theory} pp.\ 242--411,
\newblock Lecture Notes in Mathematics {\bf 1581}, Springer (1994) 
\smallskip

\bibitem[M11]{M11}
{\sc P.~M\"orters},
\newblock The parabolic Anderson model with heavy-tailed potential.
\newblock In: {\it Surveys in Stochastic Processes}, 
Proceedings of the 33rd SPA Conference in Berlin, 2009.
Edited by J. Blath, P. Imkeller, and S. R{\oe}lly. EMS Series of Congress Reports. (2011).
\smallskip


\bibitem[MOS11]{MOS11}
{\sc P.~M\"orters, M.~Ortgiese} and {\sc N.~Sidorova},
\newblock Ageing in the parabolic Anderson model,
\newblock {\it Ann.~Inst.~Henri Poincar\'e (B) Prob.~Stat.} {\bf 47:4}, 969--1000 (2011).
\smallskip

\bibitem[MP16]{MP16}
{\sc S.~Muirhead} and {\sc R.~Pymar},
\newblock Localisation in the Bouchaud-Anderson model, 
\newblock {\it Stoch.\ Proc.\ Appl.} {\bf 126:11}, 3402--3462 (2016)
\smallskip

\bibitem[MR94]{MR94}
{\sc S.~Molchanov} and {\sc A.~Ruzmaikin},
\newblock {Lyapunov exponents and distributions of magnetic fields in dynamo models.}
\newblock In \emph{The Dynkin Festschrift: Markov Processes and their Applications.}
(Ed. Mark Freidlin) pp. 287--306, Birkh\"auser (1994).
\smallskip


\bibitem[R87]{R87}
{\sc S.I.~Resnick},
\newblock \textit{Extreme Values, Regular Variation, and Point Processes},  
\newblock  Springer, New York (1987).
\smallskip

\bibitem[S98]{S98}
{\sc A.-S.~Sznitman},
\newblock {\it Brownian Motion, Obstacles and Random Media.}
\newblock Springer, Berlin (1998).
\smallskip

\bibitem[ST14]{ST14}
{\sc N.~Sidorova} and {\sc A.~Twarowski},
\newblock Localisation and ageing in the parabolic Anderson model with Weibull potential,
\newblock {\em Ann. Probab.} {\bf 42:4}, 1666--1698 (2014).
\smallskip


\end{thebibliography}
\end{document}